\documentclass[a4paper]{amsart}

\usepackage{amssymb}
\usepackage{amsthm}  
\usepackage{amsmath} 
\usepackage{amscd} 
\usepackage[all]{xypic}
\usepackage{url}
\usepackage{multirow}
\usepackage[linktocpage]{hyperref}
\usepackage{color}
\usepackage{nicematrix}

\newcommand{\omi}{\omega^{-1}}
\newcommand{\om}{\omega}

\newcommand{\si}{\sigma}

\newcommand{\ba}{\mathcal{G}}

\newcommand{\fg}{\frak g}

\newcommand{\fp}{\frak p}

\newcommand{\fk}{\frak k}
\newcommand{\fr}{\frak r}

\newcommand{\fh}{\frak h}
\newcommand{\fn}{\frak n}

\newcommand{\fs}{\frak s}
\newcommand{\s}{\mathrm{\varsigma}}

\newcommand{\fc}{\frak c}

\newcommand{\Ad}{{\rm Ad}}

\newcommand{\id}{{\rm id}}

\newcommand{\na}{\nabla}
\newcommand{\U}{\Upsilon}
\newcommand{\gr}{{\rm gr}}
\newcommand{\lz}{[\![}
\newcommand{\pz}{]\!]}
\newcommand{\Rho}{{\mbox{\sf P}}}
\newcommand{\R}{\mathbb{R}}
\newcommand{\C}{\mathbb{C}}

\newcommand{\ad}{{\rm ad}}

\newcommand{\ddd}{ (cc) }

\newtheorem{prop}{Proposition}[section]
\newtheorem*{prop*}{Proposition}
\newtheorem*{thm*}{Theorem}
\newtheorem{lem}{Lemma}[section]
\newtheorem*{lemma*}{Lemma}

\newtheorem*{cor*}{Corollary}
\newtheorem{rem}{Remark}[section]
\newtheorem*{rem*}{Remark}
\theoremstyle{definition}
\newtheorem{def*}{Definition}[section]

\theoremstyle{remark}

\begin{document}
\title{First BGG operators via homogeneous examples}
\author{Jan Gregorovi\v c and Lenka Zalabov\' a}
\address{J. G.: Faculty of Science \\ University of Hradec Kr\'alov\'e \\ Rokitansk\'eho 62, Hradec Kr\'alov\'e 50003, Czech Republic\\
L.Z.: Institute of Mathematics, 
	Faculty of Science, University of South Bohemia, 
	Brani\v sovsk\' a 1760, 370 05 \v Cesk\' e Bud\v ejovice, and
    Department of Mathematics and Statistics, 
	Faculty of Science, Masaryk University,
	Kotl\' a\v rsk\' a 2, 611 37 Brno, Czech Republic}
 \email{jan.gregorovic@seznam.cz, lzalabova@gmail.com}
 \thanks{ J.G. is supported by grant no.~19-14466Y, `Special metrics in supergravity and new G-structures' from the Czech Science Foundation.
L.Z. is supported by grant no.~20-11473S, `Symmetry and invariance in analysis, geometric modelling and control theory' from the Czech Science Foundation. 
Computations were realized in the CAS system  {\sc Maple} and package {\sc DifferentialGeometry}, \cite{dg}.
}

\begin{abstract}
We review the theory of first BGG operators and study how to approach them and find their solution on homogeneous geometries. We provide many new examples of parabolic geometries that admit solutions of first BGG operators with many interesting properties and interpretations.
\end{abstract}
\maketitle
\tableofcontents

\section*{Introduction}

BGG operators form an important class of differential operators studied intensively over the last decades, \cite{be,CSS,CD}. They can be defined on smooth filtered manifolds $M$, possibly carrying additional compatible geometric structures like projective or conformal, that can be equivalently described as parabolic geometries, \cite{parabook}. In particular, many natural properties of parabolic geometries can be formulated in terms of BGG operators,  see e.g. \cite{metrics,CGM,hol,SW}.  

In fact, BGG operators can be viewed on two different levels. 
The first approach uses the fact 
that BGG operators act between sections of distinguished vector bundles that are parts of so--called \emph{BGG sequences}, \cite{CSS,relative}. From this viewpoint, the information is compressed into a uniform theory which on the one hand allows to provide results in an efficient way but on the other hand hides the structure and source of vector bundles and operators in question.
The second approach states that BGG operators act between distinguished vector bundles described using  so--called \emph{natural prolongation}, \cite{prol,Hammerl,begn} but they can be also defined independently for an alternative description of BGG operators. 
This has the advantage that one can immediately find their solutions by solving closed systems of first--order PDEs. This viewpoint provides a deeper insight into the structure theory of BGG operators but the discussion can be only treated case by case.

We start here with a representation $V$ of a semisimple Lie algebra $\fg$ that is at the same time a compatible representation of a parabolic subgroup $P$ of some Lie group $G$ with the Lie algebra $\fg$ (so that the pair $(G,P)$ determines a type of the parabolic geometry, \cite[Section 3.1.]{parabook}), and a linear connection  $\nabla$ acting on the space $\Omega^0(V)$ of sections of a vector bundle $\mathcal{V}$ with the standard fiber $V$. Let us note that BGG theory can be generalized to more general incomes, \cite{CSS,CD,relative}. Other ingredient is the \emph{Kostant's codifferential} $\partial^*:\Omega^i(V)\to \Omega^{i-1}(V)$ acting between the spaces $\Omega^k(V)$ of $k$--forms valued in $\mathcal{V}$, \cite[Section 3.3.]{parabook}. This allows to compute the comohomology spaces $\mathcal{H}^i(V)=Ker(\partial^*)/Im(\partial^*)$ on which are the BGG sequences defined. In particular, we require the connection $\nabla$ to satisfy the compatibility condition as follows.
\begin{enumerate}
\item[$\ddd$] \label{cc} There is an operator $Q: Ker(\partial^*)\to Ker(\partial^*)$ that is a polynomial in the variable $\partial^*\circ \nabla$ such that restrictions of both compositions $Q\circ \partial^*\circ \nabla$ and $\partial^*\circ \nabla\circ Q$ to $Im(\partial^*)$ equal to identity.
\end{enumerate}

The condition \hyperref[cc]{$\ddd$} guarantees that we can define the \emph{splitting operator}
\begin{gather*}
 \mathcal{L}_0:\mathcal{H}_0(V) \to \Omega^0(V),\  \mathcal{L}_0:=\id-Q\circ \partial^*\circ \nabla
\end{gather*}
for the projection $\pi_0:\Omega^0(V)= Ker(\partial^*)\to \mathcal{H}_0(V)$ that in addition satisfies that $\nabla\circ \mathcal{L}_0$ has values in $Ker(\partial^*)$. 
Then the \emph{first BGG operator}  
\begin{gather*}
 \mathcal{D}: \mathcal{H}^0(V)\to \mathcal{H}^1(V),\ \mathcal{D}=\pi_1\circ \nabla^\Phi\circ \mathcal{L}_0
\end{gather*}
is well--defined under the assumption \hyperref[cc]{$\ddd$}, where $\pi_1: Ker(\partial^*)\to \mathcal{H}^1(V)$ is the projection. In fact, we have the following diagram
$$
\xymatrix{
0& \Omega^0(V) \ar[l]^{\partial^*}  \ar@/^/[r]^{\nabla^\Phi} \ar@/^/[d]^{\pi_0}& \Omega^1(V) \ar@/^/[l]^{\partial^*}\ar[d]_{\pi_1}  & \dots \\
& \mathcal{H}^0(V) \ar@/^/[u]^{\mathcal{L}_0}  \ar@{.>}[r]^{\mathcal{D}}& \mathcal{H}^1(V) & \dots \\
}
$$
Let us note that one can define higher order splitting operators 
and higher order BGG operators 
under assumptions analogous to \hyperref[cc]{$\ddd$}, \cite{relative,CSS}. In this paper, we restrict our considerations to the first BGG operators.

The main problem in the theory of BGG operators is that these are overdetermined differential operators and thus, it is difficult to find parabolic geometries with non--trivial solutions of BGG operators. In fact, many papers deal with the theory of BGG operators in general or with special cases like projective or conformal but they rarely provide examples. We approach this problem by considering (locally) homogeneous parabolic geometries that at least admit (by definition) solutions for BGG operators controlling infinitesimal automorphisms, \cite{deform}. For these, we develop an algebraic theory that enables us to straightforwardly construct  the (local) solutions for arbitrary first BGG operators on homogeneous parabolic geometries. This allows us to construct many interesting examples.

To get an algebraic description of  BGG operators on (locally) homogeneous geometries, we view parabolic geometries as Cartan geometries $(\ba\to M,\om)$ of type $(G,P)$ for semisimple Lie groups $G$ with Lie algebras $(\fg,[,]_\fg)$ and their parabolic subgroups $P$ with Lie algebras $\fp$. Thus the Cartan bundle $\ba$ is a $P$--principal bundle over a smooth manifold $M$ and $\om: T\ba\to \fg$ is a Cartan connection, i.e., an absolute parallelism that is $P$--equivariant and reproduces fundamental vector fields of the $P$--action. Parabolic geometries are equivalent to underlying filtered geometric structures on $M$ depending on the type of the geometry, \cite[Section 3.1.]{parabook}. Roughly said, everything can be deduced from a $|k|$--grading $\fg=\fg_{-k}\oplus \dots \oplus \fg_{k}$ of the Lie algebra $\fg$ such that $\fp=\fg_{0}\oplus \dots \oplus \fg_{k}$. We denote by $\fg_-:= \fg_{-k}\oplus \dots \oplus \fg_{-1}$ and $\fp_+:=\fg_{1}\oplus \dots \oplus \fg_{k}$ the crucial parts of the grading. 

Let us emphasize that there can be several Lie groups $K$ of automorphisms of $(\ba\to M,\om)$ acting transitively on $M$ or at least having an open orbit on $M$. We assume that, locally, $M$ is an open subset in a homogeneous space $K/H$, where the Lie algebra $(\fk,[,]_\fk)$ of $K$ consist of (not necessarily all) infinitesimal automorphisms of the parabolic geometry on $M$.
\noindent
\\[1mm]
{\bf Outline.} In Section \ref{section2}, we introduce extensions and infinitesimal extensions as powerful tools to deal with (locally) homogeneous parabolic geometries. We discuss invariant connections on associated vector bundles over homogeneous parabolic geometries and their parallel sections and we focus on their coordinate description in detail. 
In Section \ref{priklady-alp}, we construct extensions or infinitesimal extensions for several interesting examples of (locally) homogeneous parabolic geometries. In particular, we describe all projective structures on three--dimensional Heisenberg group and C--projective geometries on a non--reductive homogeneous space. Further, we focus on a $(2,3,5)$ distribution corresponding to the control problem of a ball rolling on a hyperbolic space. We consider a modification of an example included in \cite{Wilse}. We also study CR geometry and Lagrangian contact geometry on a particular examples coming from classifications \cite{homogCR,homogLagr}. Finally, we consider submaximally symmetric generalized path geometry, \cite{KT}. Let us note that for its importance, we focus on conformal geometries in a separate forthcoming paper.

In Section \ref{bggtheory}, we study first BGG operators on (locally) homogeneous parabolic geometries. In particular, we develop a method to construct explicitly local solutions of first BGG operators on arbitrary tractor bundles purely algebraically in the language of (infinitesimal) extensions and suitable invariant connections.
Let us note that our method can be viewed as an algorithm and we also implemented several procedures in {\sc Maple} to unify our computations.
In Section \ref{bggexam} we find local solutions (normal and non--normal) of first BGG operators on examples from Section \ref{priklady-alp}. We consider the standard tractor bundle and its dual and/or conjugate if it makes sense and their second tensor product. We also consider the tractor bundles related to the existence of sub--Riemannian metrics, \cite{metrics}. We also compute the full Lie algebra of infinitesimal automorphisms and the infinitesimal holonomy of the tractor connection, \cite{hol}.

In Section \ref{solcoordcap} we discuss the local/global coordinate descriptions of solutions of first BGG operators and we also focus on interpretations of normal solutions. 
In Section \ref{exam-coord}, we present coordinate descriptions of  interesting solutions for examples from Section \ref{bggexam}. 
For the projective geometry from Example \ref{solproj}, we particularly realize the normal coordinates to witness the polynomiality of the normal solutions in normal coordinates.
In the other examples, we consider the exponential coordinates introduces in Section  \ref{exponential-coord}, where we take into account the structure of the group $K$. We focus on possible choices of the group $K$ and we also describe interesting holonomy reductions in detail.




\section{Homogeneous parabolic geometries and invariant connections on vector bundles} \label{section2}
\subsection{Homogeneous parabolic geometries}
We follow here the description of homogeneous Cartan and parabolic geometries from \cite{parabook}. Consider a $K$--invariant parabolic geometry of type $(G,P)$ on an (effective) homogeneous space $K/H$. A choice of a point $u$ in the Cartan bundle provides the following data that describe the $K$--invariant parabolic geometry:
\begin{enumerate}
\item an injective Lie group homomorphism $i: H\to P$,
\item a linear map $\alpha: \fk\to \fg$ such that $\alpha\circ \Ad(h)=\Ad(i(h)) \circ \alpha$ holds for all $h\in H$, $\alpha|_\fh=di$ and $\alpha(\fk/\fh)=\fg/\fp$.
\end{enumerate}
Such a pair $(\alpha,i)$ defines a functor from the category of Cartan geometries of type $(K,H)$ into the category of parabolic geometries of type $(G,P)$. In particular, this functor maps the flat Cartan geometry $(K\to K/H,\om_K)$ of type $(K,H)$ to the $K$--invariant parabolic geometry of type $(G,P)$ as follows
\begin{align}
(K\to K/H,\om_K)\mapsto (K\times_{i(H)}P\to K/H,\om_\alpha),
\end{align}
where the property $j^*\om_\alpha=\alpha\circ \om_K$ for the natural inclusion $j: K\to K\times_{i(H)}P$  determines uniquely the Cartan connection $\om_\alpha$, c.f. \cite[Section 1.5.15.]{parabook}. Moreover, the map $\lz k,p\pz\mapsto kup$ defines an isomorphism between 
 \begin{align} \label{homog-geo}
 (K\times_{i(H)}P\to K/H,\om_\alpha) 
 \end{align}
and the $K$--invariant parabolic geometry of type $(G,P)$ on $K/H$.
Thus we always use the description \eqref{homog-geo} for $K$--invariant parabolic geometries of type $(G,P)$.

\begin{def*}
We call the pair $(\alpha,i)$ 
an \emph{extension of $(K,H)$ to $(G,P)$}.
We say that the $K$--invariant parabolic geometry $(K\times_{i(H)}P\to K/H,\om_\alpha)$ of type $(G,P)$ is given by the extension $(\alpha,i)$ of $(K,H)$ to $(G,P)$.
\end{def*}

Clearly, there are many extensions of $(K,H)$ to $(G,P)$ giving the same $K$--invariant parabolic geometry of type $(G,P)$. The basic invariant that distinguishes non--equivalent $K$--invariant parabolic geometries of type $(G,P)$ is the \emph{curvature} that is described by $P$--equivariant function
\begin{align*}
\kappa: K\times_{i(H)}P\to \wedge^2 (\fg/\fp)^*\otimes \fg.
\end{align*}
For geometries given by extensions $(\alpha,i)$ of $(K,H)$ to $(G,P)$, the restriction of $\kappa$ to $j(K)$ is a constant function
\begin{align*}
\kappa(k)(\alpha(X)+\fp,\alpha(Y)+\fp)=[\alpha(X),\alpha(Y)]_\fg-\alpha([X,Y]_\fk)
\end{align*}
for $X,Y \in \fk$ and $k \in K$, which determines the whole curvature by $P$--equivariancy. Indeed, $[\alpha(\ ),\alpha(\ )]_\fg-\alpha([\ ,\ ]_\fk)$ is an $H$--invariant element of $\wedge^2 (\fk/\fh)^*\otimes \fg\cong \wedge^2 (\fg/\fp)^*\otimes \fg$, where the $H$--action is given by the action $\Ad\circ i$ on $\fg$. Let us emphasize that there is a subalgebra $\fp_+\subset \fg$ and $P$--equivariant isomorphism $\fp_+\cong (\fg/\fp)^*$ (induced by the Killing form of $\fg$), and thus we can restrict $\kappa$ to a constant function $K\to \wedge^2 \fp_+\otimes \fg$.

\begin{def*}
The extension $(\alpha,i)$ of $(K,H)$ to $(G,P)$ is called \emph{normal} if
\begin{align}\label{norm}
2 \sum_i [Z_i,\kappa(X,X_i)]_\fg- \sum_i\kappa([Z_i,X]_\fg,X_i)=0
\end{align}
holds for all $X\in \fg$, where the elements $X_i\in \fg$ are representatives of a basis of $\mathfrak{g}/\mathfrak{p}$ and the elements $Z_i\in \fp_+$ form the corresponding dual basis of $(\mathfrak{g}/\mathfrak{p})^*$.
\end{def*}

\subsection{Locally homogeneous geometries and local description in coordinates} \label{1.2}

The extensions  $(\alpha,i)$ of $(K,H)$ to $(G,P)$ provide a global coordinate--free description of  homogeneous parabolic geometries of type $(G,P)$. In the case of locally homogeneous geometries, i.e., parabolic geometries that are locally equivalent to a homogenous parabolic geometry given by extensions  $(\alpha,i)$ of $(K,H)$ to $(G,P)$, we can forget $i: H \to P$ and consider just the map $\alpha: \fk \to \fg$. 

\begin{def*}
We call a linear map $\alpha: \fk\to \fg$ an \emph{infinitesimal extension of $(\fk,\fh)$ to $(\fg,\fp)$} if $[\alpha(X),\alpha(Y)]_\fg-\alpha([X,Y]_\fk)=0$ for all $X\in \fk,Y\in \fh$ and $\alpha(\fk/\fh)=\fg/\fp$.
\end{def*}

Indeed, if $\alpha$ is an infinitesimal extension of $(\fk,\fh)$ to $(\fg,\fp)$, then there  are groups $H \subset K$ with Lie algebras $\fh \subset \fk$ and $i: H \to P$ such that $(\alpha,i)$ is an extension of $(K,H)$ to $(G,P)$ and the parabolic geometry is locally equivalent to the homogeneous geometry given by $(\alpha,i)$. However, for general choice of $K$ and $H$, there are topological obstructions for the existence of $i: H\to P$ and for transition maps between the local homogeneous charts to be elements of $K$.

In order to obtain local coordinates on $K/H$ around $kH$, let us fix a complement 
$\Ad_{k}^{-1}\fc$ of $\fh$ in $\fk$ and consider the \emph{exponential coordinates (of the first kind)}
 $$\exp(\fc)kH=k\exp(\Ad_{k}^{-1}\fc)H.$$ 
These provide a local section $\s_{\fc,k}: \fc\to K$ as $X \mapsto k\exp(\Ad_{k}^{-1}(X))$ and we obtain a section 
$$\s:=j\circ \s_{\fc,k}: \fc \to K\times_{i(H)} P.$$ 
The Cartan connection $\omega_\alpha$ is then uniquely determined by its pullback $\s^*\omega_\alpha$ on $\fc$ and we observe the following fact. 

\begin{lem}\label{localext}
Consider a basis $\theta_i$ of the pullback $\s_{\fc,k}^*\omega_K|_{\Ad_{k}^{-1}\fc}$ of the projection of $\omega_K$ to $\Ad_{k}^{-1}\fc$ along $\fh$ at $k$. Then 
$
\s^*\omega_\alpha=\alpha(\theta_i).
$
\end{lem}
\begin{proof}
The pullback allows to identify $\Ad_{k}^{-1}\fc$ with $T(K/H)$ and in particular, the choice of a basis $\Ad_{k}^{-1}\fc$ provides a basis $\theta_i$ of $T^*(K/H)$. This provides the claimed description of $\s^*\omega_\alpha.$
\end{proof}
Clearly, the basis from Lemma \ref{localext} depends on the choice of $\fc$ and $k$ and one usually does not start in the 
exponential coordinates.  
Starting in different coordinates, it suffices to find the flows of the infinitesimal automorphisms of the parabolic geometry in directions of $\fc$ starting at $kH$ to obtain the transition from the
exponential coordinates to the starting coordinates. 

\begin{rem}
In the examples, we pullback the results of the algebraic computations from  
exponential coordinates to the starting coordinates only in the projective example. For the other examples, we simply forget the starting coordinates (in the case we use them at all) and present the results only in 
exponential coordinates or exponential coordinates of other kind, see Proposition \ref{modify-exp}.
\end{rem}

\subsection{Associated vector bundles}

Let $\lambda$ be a $P$--representation on a vector space $V$. Then we denote by
$$\mathcal{V}:=K\times_{i(H)}P\times_{\lambda}V \to K/H$$
the corresponding associated vector bundle to the Cartan bundle of the homogeneous parabolic geometry \eqref{homog-geo}. 
%
%
Let us recall that the parabolic group $P$ carries a reductive Levi decomposition $P=G_0\exp(\fp_+)$, where $\fp_+$ is the nilradical of $\fp$ and $G_0$ is reductive. The choice of the reductive Levi decomposition is equivalent to the choice of the grading element $E\in \fg$. Then $\fg$ becomes a $|k|$--graded Lie algebra $\fg=\oplus_{j=-k}^k \fg_j$ with a grading given by $[E,X]=jX$ for $X\in \fg_j$. In particular, $\fg_0$ is the Lie algebra of $G_0$ and $\fp_+=\oplus_{j>0} \fg_j$.

We can decompose the $P$--representation on the vector space $V$ as 
$V=\oplus_j V_j$
according to (not necessarily integral) eigenvalues of $d\lambda(E)$. Then we can form the corresponding $P$--invariant filtration
$$V^\ell:=\oplus_{j\geq \ell} V_j.$$
Therefore, we obtain a natural filtration of $\mathcal{V}$ of the form 
\begin{align*}
\mathcal{V}^\ell:=K\times_{i(H)}P\times_{\lambda} V^\ell.
\end{align*} 
\begin{rem}
It may be sometimes useful to shift the eigenvalues of $d\lambda(E)$ such that each indecomposable sub--representation of $P$ has integral non--negative grading. 
In this article, we do not need to shift the eigenvalues and we always refer to the homogeneity with respect to $d\lambda(E)$.
\end{rem}
\noindent
Typical examples of associated vector bundles are
\begin{itemize} 
\item 
the tangent bundle $T(K/H)=K\times_{i(H)}P\times_{\Ad} \fg/\fp$ with the filtration of the form 
$T(K/H)^{-\ell}=K\times_{i(H)}P\times_{\Ad} \fg^{-\ell}/\fp$ for $\ell>0$, and
\item  the cotangent bundle $T^*(K/H)=K\times_{i(H)}P\times_{\Ad} \fp_+$ with the filtration of the form  $T^*(K/H)^\ell=K\times_{i(H)}P\times_{\Ad} \fg^\ell$ for $\ell>0$.
\end{itemize}
Another important example is 
\begin{itemize}
\item the \emph{adjoint tractor bundle} $\mathcal{A}:=K\times_{i(H)}P\times_{\Ad} \fg$ with the filtration of the form
$ \mathcal{A}^\ell=K\times_{i(H)}P\times_{\Ad} \fg^\ell$ for $-k\leq \ell \leq k$. 
\end{itemize}
 We can also build other bundles using duals and tensor products of representations and the above filtrations induce naturally filtrations on these new bundles.

The description of $\mathcal{V}$ can be further simplified to the associated vector bundle
 $$\mathcal{V}=K\times_{\lambda\circ i(H)}V \to K/H$$
 for the induced representation of $H$. Then it is equivalent to speak about sections of the associated vector bundle $K\times_{\lambda\circ i(H)}V \to K/H$ and $H$--equivariant functions $K\to V$ for the right multiplication on $K$ and the action $\lambda\circ i$ on $V$. In our notation,
$s\in \Gamma(\mathcal{V})^\ell$
means that the $H$--equivariant function $s: K\to V^\ell\subset V$ is a section of $\mathcal{V}^\ell\subset \mathcal{V}$.

\begin{def*}
We say that the extension $(\alpha,i)$ of $(K,H)$ to $(G,P)$ is \emph{regular} if
$\kappa\in \Gamma(\wedge^2T(K/H)^*\otimes \mathcal{A})^1.$
\end{def*}

Let us note that regular normal extensions  $(\alpha,i)$ of $(K,H)$ to $(G,P)$ correspond to regular normal $K$--homogeneous parabolic geometries of type $(G,P)$, and thus to corresponding underlying geometric structures on $K/H$, \cite[Section 3.1.]{parabook}. 

\subsection{Natural differentiation on associated vector bundles}\label{sec2.1}

Let us start with the adjoint tractor bundle of the flat Cartan geometry of type $(K,H)$ given as
$\mathcal{K}:=K\times_{\Ad(H)}\fk.$
We identify $t\in \Gamma(\mathcal{K})$ with projectable vector fields for the right $H$--action on $K$ using $\om_K$, i.e., with lifts of vector fields on $K/H$.
\begin{def*}
We define the \emph{fundamental derivative} $D^\fk : \Gamma(\mathcal{V})\to \Gamma(\mathcal{K}^*\otimes \mathcal{V})$ as $ D^\fk_t s=\om_K^{-1}(t).s$
for $t\in  \Gamma(\mathcal{K})$ and $s\in  \Gamma(\mathcal{V})$, where $.$ denotes the directional derivative in the direction of vector field $\om_K^{-1}(t)$ on $K$. 
\end{def*}
The fundamental derivative is well--defined due to $H$--invariance of $\om_K^{-1}(t)$.

An infinitesimal extension $\alpha: \fk\to \fg$ provides a natural inclusion
$ \alpha: \Gamma(\mathcal{K})\to \Gamma(\mathcal{A})$ and we can pullback the filtration $\mathcal{A}^\ell$ to a filtration $\mathcal{K}^\ell$ of $\mathcal{K}$. Therefore, we get
$$D^\fk_t s=-d\lambda\circ \alpha(t)(s)$$
for  $t\in \Gamma(\mathcal{K}^0)$ and $s\in  \Gamma(\mathcal{V})$. 
Then
the \emph{fundamental derivative} $D^\fg:  \Gamma(\mathcal{V})\to \Gamma(\mathcal{A}^*\otimes \mathcal{V})$
satisfies for $t\in \Gamma(\mathcal{K})$, $t_\fp\in  \Gamma(\mathcal{A}^0)$ and $s\in  \Gamma(\mathcal{V})$
$$D^\fg_{\alpha(t)+t_\fp}s=D^\fk_{t}s-d\lambda(t_\fp)(s).$$

Let us emphasize that the fundamental derivatives are $K$--invariant operators, preserve the filtration $ \Gamma(\mathcal{V})^\ell$ and have further natural properties. This is not generally true for the directional derivative on $K/H$. Indeed, the directional derivative of sections of $\mathcal{V}$ is a natural operator if and only if the homogeneous space $K/H$ is reductive, i.e., there is an $H$--invariant complement $\frak{c}$ of $\fh$ in $\fk$. In such case, the $H$--equivariant inclusion
$\alpha(\frak{c}) \subset \fg$
defines a distinguished $K$--invariant connection on $\mathcal{V}$. In general, there is the following characterization of $K$--invariant connections on $\mathcal{V}$.

\begin{lem}\label{inv-con}
There is one-to-one correspondence between $K$--invariant linear connections $\Gamma(\mathcal{V})\to \Gamma(T^*(K/H)\otimes \mathcal{V})$ and $H$--equivariant maps $\Phi: \fk\to \frak{gl}(V)$ satisfying $\Phi(Y)=d\lambda\circ \alpha(Y)$ for all $Y\in \fh$. Precisely, each $K$--invariant linear connection is given as 
$$\nabla^\Phi:=D^\fk+\Phi.$$
The $K$--invariant linear connection $\nabla^\Phi$ preserves the filtration $ \mathcal{V}^\ell$ if and only if $$\Phi\in \frak{gl}(V)^0.$$
The difference of two $K$--invariant linear connections is given by an $H$--equivariant map $\psi: \fk \to \frak{gl}(V)$ satisfying $\psi(Y)=0$ for all $Y\in \fh$. 

A parallel section $s$ of the $K$--invariant linear connection $\nabla^\Phi=D^\fk+\Phi$ is $K$--invariant if and only if $\Phi(Y)(s)=0$ for all $Y\in \fk$.
\end{lem}
\begin{proof}
Since there is a natural inclusion $\Gamma(T^*(K/H)\otimes \mathcal{V})\to \Gamma(\mathcal{A}^*\otimes \mathcal{V})$, we can extend a $K$--invariant linear connection $\nabla:\Gamma(\mathcal{V})\to \Gamma(T^*(K/H)\otimes \mathcal{V})$ to take values in $\Gamma(\mathcal{A}^*\otimes \mathcal{V})$ and use the pullback to $\mathcal{K}^*$. Comparing $D^\fk$ with $\nabla$ then gives $\Phi$ with the claimed properties using the properties of the fundamental derivative. Then the converse statement and the statement on the preserving of filtrations follows.

A function $s\in\Gamma(\mathcal{V})$ is $K$--invariant if and only if $s:K\to V$ is constant. This means $\nabla^\Phi s=0$ and $\nabla^\Phi s=D^\fk s+\Phi(s)=0+\Phi(s)=0$ holds.
\end{proof}

In general, $H$--equivariant maps $\Phi: \fk\to \frak{gl}(V)$ satisfying $\Phi(Y)=d\lambda\circ \alpha(Y)$ for all $Y\in \fh$ do not have to exist for various $V$. Such maps always exist on reductive spaces $K/H$, and corresponding $\Phi$ is defined by taking $0$ on the invariant complement $\fc$ and $d\lambda\circ \alpha$ on $\fh$.

\subsection{Forms valued in associated vector bundles and parallel sections of invariant linear connections} \label{paralel}

Consider spaces of $V$--valued forms on $K/H$ that we identify with the spaces
\begin{align*}
\Omega^k(V):=\Gamma(K\times_{(\Ad^k\otimes \lambda)\circ i(H)}\wedge^k\fp_+\otimes V).
\end{align*}
These are again spaces of sections of associated vector bundles and $\Omega^0(V)=\Gamma(\mathcal{V})$. We recall that the \emph{Kostant's codifferential}
$\partial^*: \Omega^{k+1}(V)\to \Omega^{k}(V)$  is defined pointwise via $\partial^*: \wedge^{k+1}\fp_+\otimes V\to \wedge^{k}\fp_+\otimes V$ as
 \begin{align*}
 \partial^*(X_0&\wedge \dots  \wedge X_k\otimes v)=\sum_{j} (-1)^{j+1}X_0\wedge \dots  \wedge \hat X_j\wedge \dots \wedge X_k\otimes d\lambda(X_j)(v)\\ 
 &+\sum_{j<l} (-1)^{j+l} [X_{j},X_{l}]_\fg \wedge \dots \wedge \hat X_{j}\wedge \dots\wedge \hat X_{l}\wedge \dots \wedge X_k\otimes v
 \end{align*}
and 
denote $\pi_i$ projections onto the cohomology spaces $\mathcal{H}^i(V)=Ker(\partial^*)/Im(\partial^*)$.

The $K$--invariant linear connection $\nabla^\Phi$ is in fact a map $\Omega^0(V)\to \Omega^1(V)$ and there is the corresponding \emph{covariant exterior derivative} $d^{\Phi}: \Omega^k(V)\to \Omega^{k+1}(V)$  
\begin{align*}
(d^{\Phi}\Theta)&(\xi_0,\dots,\xi_k)
=\sum_j(-1)^j\nabla^\Phi_{\xi_j}(\Theta(\xi_0,\dots,\hat \xi_j,\dots,\xi_k))\\
&+\sum_{j<l}(-1)^{j+l}\Theta([\xi_j,\xi_l],\xi_0,\dots,\hat \xi_j,\dots,\hat \xi_l,\dots,\xi_k),
\end{align*}
where $\hat \xi_j$ means that the vector field $\xi_j$ is omitted.
In particular, there is the following description the curvature of $\na^\Phi$.

\begin{lem}
The curvature of the linear connection $\nabla^\Phi$ on $\mathcal{V}$ is described by the function
\begin{align*}
R^\Phi:=d^{\Phi}\circ \nabla^\Phi: K\to \wedge^2(\fk/\fh)^*\otimes \frak{gl}(V)
\end{align*}
and for $X_0,X_1\in \fk$ takes form 
$$R^{\Phi}(X_0,X_1)=[\Phi(X_0),\Phi(X_1)]-\Phi([X_0,X_1]_\fk).$$
 Moreover, $R^\Phi$ satisfies the Bianchi identity 
$d^{\Phi}\circ R^\Phi=0.$
\end{lem}
\begin{proof}
Identities $D^\fk\Phi=0$ and $(D^\fk)^2(s_1,s_2)- (D^\fk)^2(s_2,s_1)-D^\fk_{[s_1,s_2]_\fk}=0$ and $[s_1,s_2]=D^\fk_{s_1}s_2-D^\fk_{s_2}s_1+\alpha([s_1,s_2]_\fk)$ 
for $s_1,s_2\in \Gamma(\mathcal{K})$ from \cite{parabook} imply the second formula and the rest follows.
\end{proof}

Let us describe parallel sections of a $K$--invariant linear connection $\nabla^\Phi=D^\fk+\Phi$, i.e., sections $s\in \Gamma(\mathcal{V})$ such that $\nabla^\Phi s=0$, in detail.

\begin{def*}
Let us consider the set
$$\mathcal{S}^0:=\{v\in V: R^{\Phi}(X_0,X_1)(v)=0,\  X_0,X_1\in \fk\}.$$
Then for $k>0$ we consider sets
$$\mathcal{S}^k:=\{v\in \mathcal{S}^{k-1}: \Phi(X)(v)\in \mathcal{S}^{k-1},\ X\in \fk\}.$$
\end{def*}

The endomorphisms used in the definition of sets $\mathcal{S}^i$ are exactly the generators of the infinitesimal holonomy of $\nabla^\Phi$ which is the same at all points due to homogeneity, \cite{KoNo2}. Moreover, since $V$ is finite--dimensional, we get $\mathcal{S}^k=\mathcal{S}^{k+1}=\dots =:\mathcal{S}^\infty$ for $k$ large enough. In particular, it follows from the definition of $\mathcal{S}^0$ that $\mathcal{S}^\infty$ is a representation of $\fk$ and we denote $\mathcal{S}^\infty_K$ the maximal sub--representation of $\mathcal{S}^\infty$ that integrates to $K$. This allows to interpret elements of $\mathcal{S}^\infty$ and $\mathcal{S}^\infty_K$ as follows.

\begin{prop}[\cite{KoNo2,H}] \label{prop1.1}
For every $k\in K$, the set $\mathcal{S}^\infty$ is the set of values $s(k)$ for all $s\in \Gamma(\mathcal{V})$ satisfying $\nabla^\Phi s=0$ on some (simply connected) neighbourhood of $k\in K$.  Moreover, $\mathcal{S}^\infty_K$ consists of values of globally defined parallel sections.

The local section $s\in \Gamma(\mathcal{V})$ with value $s(k)\in \mathcal{S}^\infty$ satisfying $\nabla^\Phi s=0$ on some (simply connected) neighborhood of $k$ is given by 
\begin{align} \label{parsec}
s(k\exp(\Ad_{k}^{-1}X))=\exp(- \Phi(\Ad_{k}^{-1}X))s(k)
\end{align}
for $X$  in some neighborhood of $0$ in $\fk$. If we consider a complement $\fc$ of $\fh$ in $\fk$ providing local coordinates on $K/H$ via exponential map, then \eqref{parsec}
for all $X$ in some neighborhood of $0$ in $\fc$ gives the description of the section $s$ in local coordinates.

The trivial sub--representations of  $\mathcal{S}^\infty$ are contained in $\mathcal{S}^\infty_K$ and correspond to $K$--invariant parallel sections.
\end{prop}
In general, the exponential map does not have to cover the whole $K$ and one needs to consider several local descriptions for different elements $k$ (or exponential coordinates of different kind) to describe a global section.

\section{Examples of homogeneous parabolic geometries and extension functors}
\label{priklady-alp}
Homogeneous and locally homogeneous parabolic geometries play an important role in various classification results. Simplest examples are Lie groups carrying additional invariant geometric structures, \cite{KoNo}. It also turns out that submaximal models of parabolic geometries are locally homogeneous. Construction of local models via deformations of suitable algebras and classification in the complex setting is given in \cite{KT}. In the real setting, one has to proceed case by case, see e.g. \cite{KWZ,subCR,subCproj}. There are also known classification of locally homogeneous parabolic geometries of several types in special dimensions. For example, in \cite{homogCR,homogLagr} the authors study special classes of Lagrangian contact structures and CR--structures in dimension $5$ by means of Cartan's reduction method and Petrov--like methods. In \cite{Wilse} the  author studies real $(2,3,5)$ distributions in dimension $5$ and also distinguishes those that correspond to rolling bodies.

We demonstrate here on examples how to describe (locally) homogeneous parabolic geometries in a uniform way using (infinitesimal) extensions. The construction is done in several steps as follows.

\begin{enumerate}
\item Find a Lie algebra $\fk$ of infinitesimal automorphisms that acts transitively on the (open subset of the) homogeneous space. 
We find directly suitable $\fk$ in Examples \ref{proj-alp}, \ref{cproj-alp}, we modify $\fk$ from known classification in Example \ref{235-alp} and we use the known classifications of such $\fk$ in Examples \ref{CR-alp}, \ref{path-alp}.
\item Find 
\begin{itemize}
\item an associated grading $\gr(\fk)$ of $\fk$, or 
\item $H$--invariant filtration of $\fk^i$ such that $\gr(\fk)_-=\fg_-$. 
\end{itemize}
These provide $\alpha|_\fh$ or even $i: H\to P$ in the case that the additional structures on $\fg_{-1}$ defining the parabolic geometry are preserved by $H$ as in Examples \ref{CR-alp},  \ref{path-alp}.
\item If $\gr(\fk)$ determines an infinitesimal extension $\gr(\alpha)$ of  $(\fk,\fh)$ to $(\fg,\fp)$, as in all the examples except Example \ref{cproj-alp}, it remains find the unique (up to isomorphism) normal extension $\alpha$ as the sum of $\gr(\alpha)$ and the uniquely determined $H$--equivariant map $\gr(\fk)_-\to \fg$ such that $\partial^*\kappa_\alpha=0$ holds for the curvature.
\item We show in Example \ref{cproj-alp}, how to proceed in the case that  $\gr(\fk)$ determines only infinitesimal extension $\gr(\alpha)$ of  $(\gr(\fk),\gr(\fh))$ to $(\fg,\fp)$ for the associated graded Lie algebras.
\end{enumerate}


\subsection{Extension in projective geometry}\label{proj-alp}
Regular normal parabolic geometries of type $(PGL(n+1,\R),P_1)$ correspond to projective geometries $(M,[\na])$, where $[\nabla]$ is a class of torsion--free linear connections sharing the same geodesics as unparametrized curves on $M$, \cite[Section 4.1.5.]{parabook}. We write elements of 
$\fg=\frak{sl}(n+1,\R)$ as $(1,n)$--block matrices with a $|1|$--grading as follows $$\left[ \begin{smallmatrix} a & Z  \\ X & A \end{smallmatrix} \right],$$
where $A \in \frak{gl}(n,\R)$ and $a \in \R$ such that $a+tr(A)=0$ form $\fg_0$, $X \in \R^{n} \simeq \fg_{-1}$ and $Z \in \R^{n*} \simeq \fg_{1}$. Then $\fp_1$ corresponds to block upper triangular matrices.

We describe here all invariant projective geometries on the Lie group $H_{12}$ of unimodular lower triangular $3\times 3$ matrices, i.e., homogeneous space $K/H$ with $K=H_{12}$ and  $H=\{\id\}$. Thus, $\fc=\fk$ provides a unique choice of the complement and then the local section $\s_{\fc,k}: \fc\to K$ is just a choice of exponential coordinates on $K$. However, the usual parametrization of matrices provides different starting coordinates  $(y_1,y_2,y_3)$ on $H_{12}$, where left--invariant vector fields and the Maurer--Cartan form of $H_{12}$ define the following frame and coframe on $H_{12}$
\[X_1:= \partial_{y_1}+y_2\partial_{y_3},\ \ X_2:=\partial_{y_2},\ \ X_3:=\partial_{y_3},\] 
\[\theta_1:=dy_1,\ \ \theta_2:=dy_2,\ \ \theta_3:=-y_2dy_1+dy_3,\]
where we denote  $\partial_y := \frac{\partial}{\partial_{y}}$. A Cartan bundle of a parabolic geometry  of type $(PGL(n+1,\R),P_1)$ on $H_{12}$ always admits trivialization $H_{12}\times P_1$ such that, if we denote by $\tau$ the section $\tau: k\mapsto (k,e)$, $k\in H_{12}$, then the pullback $\tau^*\omega$ of an arbitrary Cartan connection $\omega$ of type $(PGL(n+1,\R),P_1)$ is given by the $(1,n)$--block matrix of one--forms as 
\begin{align} \label{pul}
\tau^*\omega=\left[ \begin{smallmatrix} a^k\theta_k & \Rho^k_j\theta_k \\ \theta_i & A^{jk}_i\theta_k  \end{smallmatrix} \right].
\end{align}
The components $a^k,A^{jk}_i$ determine the connection form of a torsion--free linear connection in the projective class $[\nabla]$, and the component $\Rho^k_j$ is (up to the sign convention) the projective $\Rho$--tensor of such linear connection. This is the usual description of Cartan geometries via adapted frames, \cite{S}, and the following Proposition that is a consequence of Lemma \ref{localext} provides a comparison 
with the description using extensions. 

\begin{prop}\label{proj-CC}
A Cartan connection $\omega$  on $H_{12}\times P_1$ with the pullback $\tau^*\omega$ given by \eqref{pul} is an invariant Cartan connection if and only if $a^k, A^{jk}_i,\Rho^k_j$ are constant. In such case, $$\alpha(x_1X_1+x_2X_2+x_3X_3):=\left[ \begin{smallmatrix} a^kx_k& \Rho^k_jx_k \\ x_i & A^{jk}_ix_k  \end{smallmatrix} \right]$$ defines a normal extension of $(H_{12},\{\id\})$ to $(PGL(n+1,\R),P_1)$.
\end{prop}

Let us emphasize that $\gr(\alpha)(x_1X_1+x_2X_2+x_3X_3)=\left[ \begin{smallmatrix} 0&0 \\ x_i & 0 \end{smallmatrix} \right]$ and the rest of $\alpha$ is a map $\gr(\fk)_{-1}\to \fp_1$ that is trivially $H$--equivariant. This map is not determined uniquely by the  normalization condition \eqref{norm}, because there is a freedom given by the action of $\exp(Z)$ for $Z\in \fg_1$, i.e., change of the section $\tau$ as $k \to (k,\exp Z)$. To get a unique extension describing the Cartan connection, we can assume that $a=0$ which completely removes this freedom and guarantees exactness of the connection in the projective class. Under these conditions, the following matrix of one--forms describes all invariant projective geometries on $H_{12}$
\[
 \left[ \begin {smallmatrix} 0&\Rho^k_1\theta_k&\Rho^k_2\theta_k&
  \Rho^k_3\theta_k\\   
\theta_1&a_{1}\theta_1+\theta_2a_{2}+a_{3}\theta_3&a_{2}\theta_1+\theta_2a_{4}+a_{5}\theta_3&a_{3}\theta_1+a_{5}\theta_2+a_{6}\theta_3\\ 
\theta_2&a_{7}\theta_1- ( a_{1}+a_{8}) \theta_2+\theta_3a_{9}& \theta_3a_{11}- ( a_{2}+a_{10} ) \theta_2- ( a_{1}+a_{8}) \theta_1&a_{9}\theta_1+\theta_2a_{11}+a_{12}\theta_3\\ 
\theta_3&a_{13}\theta_1+\theta_2a_{14}+\theta_3a_{8}&\theta_2a_{15}+\theta_3a_{10}+ ( a_{14}-1 ) \theta_1& \theta_2a_{10}+a_{8}\theta_1-( a_{3}+a_{11} ) \theta_3\end{smallmatrix} \right],
\]
where the normalization condition \eqref{norm} determines the relations between elements $A^{jk}_i$ (torsion--freeness) and completely determines $\Rho$ as
\begin{align*}
\Rho^1_1&=  a_{13}a_{3}+a_{7}a_{2}+a_{14}a_{9}+a_{1}^2+a_{1}a_8+a_{8}^2, \\
\Rho^2_1&={\scriptstyle \frac12} ( a_{13}a_{5}+a_{7}a_{4}+a_{14}a_{11}+ a_{14}a_{3}+a_{15}a_{9}+a_{1}a_{2}+a_{1}a_{10}+2a_{8}a_{10}-a_{3} ),\\
\Rho^3_1&={\scriptstyle \frac12} (a_{13}a_{6}+a_{7}a_{5}+a_{2}a_{9}+a_{14}a_{12}-a_{1}a_{11}+a_{1}a_{3}-2a_{8}a_{11}+a_{9}a_{10} ),\\
\Rho^1_2&= {\scriptstyle \frac12} (a_{13}a_{5}+a_{7}a_{4}+a_{14}a_{11}+a_{14}a_{3}+a_{15}a_{9}+a_{1}a_{2}+a_{1}a_{10}+2a_{8}a_{10}-a_{3} ),  \\
\Rho^2_2&=  -a_{4} a_{1}-a_{4}a_8+a_{14}a_{5}+a_{15}a_{11}+a_{2}^2+a_{2}a_{10}+a_{10}^2-a_{5},\\
\Rho^3_2&= {\scriptstyle \frac12}(a_{4}a_{9}+a_{14}a_{6}+a_{15}a_{12}-a_{1}a_{5}-a_{2}a_{11}+a_{2}a_{3}-a_{11}a_{10}-a_{10}a_{3}-a_{6} ),\\
\Rho^1_3&= {\scriptstyle \frac12} (a_{13}a_{6}+a_{7}a_{5}+a_{2}a_{9}+a_{14}a_{12}-a_{1}a_{11}+a_{1}a_{3}-2a_{8}a_{11}+a_{9}a_{10} ), \\
\Rho^2_3&={\scriptstyle \frac12} (a_{4}a_{9}+a_{14}a_{6}+a_{15}a_{12}- a_{1}a_{5}-a_{2}a_{11}+a_{2}a_{3}-a_{11}a_{10}-a_{10}a_{3}-a_{6} ),\\
\Rho^3_3&= a_{8}a_{6}+ a_{9}a_{5}+a_{11}^2+a_{11}a_{3}+a_{12} a_{10}+a_{3}^2 .
\end{align*}

For further computations, we fix a particular connection on $H_{12}$ for which by its shape we can expect many solutions of the BGG operators. We consider the connection given by $\tau^*\omega$ that has the curvature $\tau^*\kappa$ as follows
\[
\tau^*\omega= \left[ \begin{smallmatrix} 0&\theta_1&0&0\\ 
\theta_1&0&0&0\\ 
\theta_2&- \theta_2& - \theta_1&0\\ 
\theta_3&\theta_3& -\theta_1& \theta_1 \end{smallmatrix} \right],
\ \ \ \ 
\tau^*\kappa= \left[ \begin{smallmatrix} 0&0&0&0\\  
0&0&0&0\\ 
0&0&0&0\\ 
0&4\theta_1\wedge \theta_2&0&0 \end{smallmatrix} \right],
\]
where $\theta_1\wedge \theta_2(X,Y)=\frac12(\theta_1(X)\theta_2(Y)-\theta_1(Y)\theta_2(X))$.

\subsection{Extension in C-projective geometry}\label{cproj-alp}

Regular normal parabolic geometries of type $(PGL(n+1,\C),P_1)$ correspond to C-projective geometries $(M,[\na])$, where $[\nabla]$ is a class of minimal almost complex connections sharing the same C-planar curves as unparametrized complex curves on $M$, \cite{CEMN}. We write elements of 
$\fg=\frak{sl}(n+1,\C)$ as $(1,n)$--block matrices with a $|1|$--grading as follows $$\left[ \begin{smallmatrix} a & Z  \\ X & A \end{smallmatrix} \right],$$
where $A \in \frak{gl}(n,\C)$ and $a \in \C$ such that $a+tr(A)=0$ form $\fg_0$, $X \in \C^{n} \simeq \fg_{-1}$ and $Z \in \C^{n*} \simeq \fg_{1}$. Then $\fp_1$ corresponds to block upper triangular matrices.

We describe here all invariant projective geometries on a non--reductive homogeneous space $K/H=(SL(2,\R)\rtimes \R^2)/(H_1\times \R)$, where $SL(2,\R)$ acts on $\R^2$ by the standard representation, $H_1$ is the Lie group of unimodular lower triangular matrices in $SL(2,\R)$ and $\R$ is the kernel of the action of $H_1$ on $\R^2$. Further, we describe all invariant C-projective geometries on its complexification $K(\C)/H(\C)=(SL(2,\C)\rtimes \C^2)/(H_1(\C)\times \C)$. We start with the basis $e_1,\dots,e_5$ of $\fk$ over $\R$ and simultaneously of $\fk(\C)$ over $\C$ such that
\begin{gather*}
[e_1,e_3]=e_3,[e_1,e_4]=-e_4,[e_2,e_5]=-e_2,[e_3,e_4]=-e_1,\\
[e_3,e_5]=-\frac12 e_3,[e_4,e_5]=\frac12 e_4.
\end{gather*}
The Lie algebras of $H$ and $H(\C)$ are then spanned by $e_4,e_5$.

Since the homogeneous space is not reductive, the construction of the extension is more complicated than in the previous example and we approach it completely algebraically. Firstly, we consider the associated graded algebras $\gr(\fk)$ and $\gr(\fh)$, where we define 
\begin{align*}
\gr(\fk)_{-1}&=\langle x_1e_1+x_2e_2+x_3e_3 \rangle,\\
\gr(\fk)_0&=\gr(\fh)=\langle x_4e_4+x_5e_5 \rangle
\end{align*}
for $x_1,\dots,x_5$ in $\R$ or $\C$, respectively. Thus
\begin{gather*}
[e_1,e_4]={\bf 0}
\end{gather*}
becomes the difference between brackets on $\fk$ and $\gr(\fk)$.

We identify $\gr(\fk)_{-1}=\fg_{-1}$ and find elements of $\fg_0$ acting as the elements $e_4,e_5\in \gr(\fh)$ and we obtain the graded infinitesimal extension
\[
\gr(\alpha)(x_1,\dots,x_5)=\left[ \begin{smallmatrix} -\frac38 x_5&0&0&0\\  
x_1&-\frac38 x_5&0&x_4\\ 
x_2&0&\frac58 x_5&0\\ 
x_3&0&0&\frac18 x_5\end{smallmatrix} \right]
\]
of  $(\gr(\fk),\gr(\fh))$ to $(\frak{sl}(4,\R),\fp_{1})$ and of $(\gr(\fk)(\C),\gr(\fh)(\C))$ to $(\frak{sl} (4,\C),\fp_{1})$.

This clearly induces Lie algebra homomorphimsm $i: H\to P_1$ on $\gr(\fh)$. However, $\gr(\alpha)$ is not an infinitesimal extension of $\fk$, because the condition $\gr(\alpha)\circ \Ad(h)=\Ad(i(h)) \circ \gr(\alpha)$ is not satisfied for all $h\in H$. We need to add to $\gr(\alpha)$ a map $\gr(\fk)_{-1}\to \fp_1$ of a similar shape as in the projective example to obtain $\alpha$.
Again, we require $a=0$ which we achieve by considering the action of $\exp(Z)$ for $Z\in \fg_1$. 
Then the normalization conditions \eqref{norm} together with the conditions $\alpha\circ \Ad(h)=\Ad(i(h)) \circ \alpha$ for all $h\in H$ provide non--trivial equations we need to solve to obtain $\alpha$. 

The computation shows that there is the (real or complex) one--parameter family $(\alpha_s,i)$  of extensions of $(K,H)$ to $(PGL(4,\R),P_{1})$ and extensions  of $(K(\C),H(\C))$ to $(PGL(4,\C),P_{1})$, respectively

\[
\alpha_s(x_1,\dots,x_5)=\gr(\alpha)(x_1,\dots,x_5)+\left[ \begin{smallmatrix} 0&0&0&0\\  
0&-\frac12 x_1&0&0\\ 
0&0&0&sx_3\\ 
0&-\frac12 x_3&0&\frac12 x_1\end{smallmatrix} \right]
\]
that have the curvature
\[
\kappa_s(x_1,\dots,x_5,y_1,\dots,y_5)= \left[ \begin{smallmatrix} 0&0&0&0\\  
0&0&0&0\\ 
0&0&0&\frac32s(x_3y_1-x_1y_3)\\ 
0&0&0&0 \end{smallmatrix} \right].
\]
However, there is an automorphism 
 $\sigma_s(x_1,x_2,x_3,x_4,x_5)=(x_1,x_2,sx_3,\frac{x_4}{s},x_5)$ of $\fk$  such that $\Ad_p^{-1}\circ \alpha_1\circ \sigma_s=\alpha_s$ where $p$ is a diagonal matrix $diag(\frac{1}{\sqrt[4]{s}},\frac{1}{\sqrt[4]{s}},\frac{1}{\sqrt[4]{s}},\sqrt[4]{s^3})$. This implies that all $(\alpha_s,i)$ for $s\neq 0$ define the same (up to isomorphism) C-projective structure. So we work with $s=1$ in the following sections.
 
 \subsection{Extension in $(2,3,5)$ distributions} \label{235-alp}

Regular normal parabolic geometries of type $(G_2,P_{1})$ correspond to $(2,3,5)$ distributions, i.e., maximally non--integrable $2$--dimensional distributions on $5$--dimensional manifolds, \cite[Section 4.3.2.]{parabook}. We write elements of the exceptional (split) Lie algebra
$\fg=\fg_{2}(2)$ as $(1,2,1,2,1)$--block matrices with a $|3|$--grading as follows $$\left[ \begin{smallmatrix}
 a & Z_1&  z_2 & Z_3& 0  \\
X_1  & A& -JZ_1^t& -\frac{z_2}{2}IJ & IZ_3^t \\
y_2 & X_1^tJ & 0  & Z_1IJ & z_2\\
X_3&  -\frac{y_2}{2}IJ& IJX_1^t& -IA^tI & IZ_1^t\\
0& X_3^tI& y_2& X_1^tI & -a 
  \end{smallmatrix} \right], \ \ I=\left[ \begin{smallmatrix}
0& -1\\
-1& 0
  \end{smallmatrix} \right],\ \ J=\left[ \begin{smallmatrix}
0& \sqrt{2}\\
-\sqrt{2}& 0
  \end{smallmatrix} \right],$$
 where $A \in \frak{gl}(2,\R)$ and $a \in \R$ such that $a+tr(A)=0$ form $\fg_0$, $X_i \in \R^{2} \simeq \fg_{-i}$, $Z_i \in \R^{2*} \simeq \fg_{i}$, $y_2 \in \R\simeq \fg_{-2}$, $z_2 \in \R \simeq \fg_2$. Then $\fp_{1}$ corresponds to block upper triangular matrices. 
  
In this example, we consider $K/H$ for $K=SL(2,\R)\times SO(3)$ and $H=SO(2)$ that is known to carry one--parameter family of $K$--invariant $(2,3,5)$ distributions for a particular embedding of $H$ into $K$, \cite{Wilse}. However, we choose a simpler embedding than the one in \cite{Wilse}. We consider  the stabilizer $SO(2)\times SO(2)$ of the origin in the product of the hyperbolic space $SL(2,\R)/SO(2)$ and the sphere $SO(3)/SO(2)$ and embed $H$ as the diagonal $diag(SO(2))$ into it.

\begin{rem}
This space can be viewed as the configuration space for the control problem of a ball rolling on a hyperbolic space. This is a control problem with the growth vector $(2,3,5)$ and the parameter describes ratio between scalar curvatures of the two rolling spaces. We choose the parameter such that the corresponding BGG equations shall have solutions.
\end{rem}
Let $(H,X,Y)$ be an $\frak{sl}(2,\mathbb{R})$--triple, i.e., $[H,X]=-2X,[H,Y]=2Y,[X,Y]=-H$, and $(A,B,C)$ a basis of $\frak{so}(3)$ such that $[A,B]=C,[B,C]=A,[C,A]=B$. Then we consider
\begin{align*}
\fh=\fk^{0}&=\langle{\scriptstyle \frac12}(X-Y)+C \rangle\\
\fk^{-1}/\fk^0&=\langle H-2A, X+Y-2B \rangle.
\end{align*}
It is easy to check that $\fk^{-1}/\fk^0$ is $H$--invariant and thus defines a $2$--dimensional distribution on $K/H=(SL(2,\R)\times SO(3))/diag(SO(2))$. We compute 
\begin{align*}
[H-2A, X+Y-2B]&= 2\sqrt{2}({\scriptstyle -\frac{1}{\sqrt{2}}}X+ {\scriptstyle \frac{1}{\sqrt{2}}}Y+{\scriptstyle \frac{2}{\sqrt{2}}}C),\\
[{\scriptstyle -\frac{1}{\sqrt{2}}}X+{\scriptstyle \frac{1}{\sqrt{2}}}Y+{\scriptstyle \frac{2}{\sqrt{2}}}C, H-2A]&={\scriptstyle -\frac{3\sqrt{2}}{2}}({\scriptstyle \frac23}X+{\scriptstyle \frac23}Y+{\scriptstyle \frac43} B),\\
[{\scriptstyle -\frac{1}{\sqrt{2}}}X+{\scriptstyle \frac{1}{\sqrt{2}}}Y+{\scriptstyle \frac{2}{\sqrt{2}}}C, X+Y-2B]&={\scriptstyle \frac{3\sqrt{2}}{2}}({\scriptstyle \frac23}H+ {\scriptstyle \frac43} A)
\end{align*}
Thus we get a $(2,3,5)$ distribution and the given elements of the Lie algebra satisfy the same bracket relations as $\fg_-$ in our matrix representation.
Let us write
$$
\fk_{-3}\oplus\fk_{-2}\oplus\fk_{-1}\oplus\fk_{0}:=\langle x_1e_1+x_2e_2\rangle \oplus \langle x_3e_3\rangle \oplus \langle x_4e_4+x_5e_5\rangle \oplus \langle x_6e_6\rangle,
$$
where
\begin{gather*}
e_1:={\scriptstyle \frac23}X+{\scriptstyle \frac23}Y+{\scriptstyle \frac43} B,\ \ \ 
e_2:={\scriptstyle \frac23}H+{\scriptstyle \frac43} A,\\
e_3:={\scriptstyle -\frac{1}{\sqrt{2}}}X+{\scriptstyle \frac{1}{\sqrt{2}}}Y+{\scriptstyle \frac{2}{\sqrt{2}}}C,\\
e_4:= H-2A,\ \ \
e_5:=X+Y-2B,\\ 
e_6:= {\scriptstyle \frac12}(X-Y)+C.
\end{gather*}
We obtain the expression for $\gr(\alpha):\fk\to \fg$ as follows
\begin{align} \label{ex_nenorm5}
\gr(\alpha)(x_1,\dots,x_6)=  \left[ \begin {smallmatrix} 
0&0&0&0&0&0&0\\  
x_4&0&- x_6&0&0&0&0\\ 
 x_5& x_6&0&0&0&0&0\\  
x_3&-\sqrt {2} x_5&\sqrt {2} x_4&0&0&0&0\\
  x_1&-\frac{\sqrt {2}}{2}  x_3 &0&\sqrt {2} x_4&0& x_6&0\\
  x_2&0&\frac{\sqrt {2}}{2}  x_3 &-\sqrt {2}x_5&- x_6&0&0\\
 0&- x_2&-x_1& x_3&- x_5&- x_4&0\end {smallmatrix} \right]
\end{align}
by finding element of $\fg_0$ acting as $e_6$ on the associated grading. This is a regular extension $(\gr(\alpha),i)$ of $(K,H)$ to $(G_2,P_{1})$ for the natural inclusion $i: H=SO(2)\subset GL(2,\R)=G_0\subset P_1$ due to $H$--invariance of the chosen grading.
For the chosen basis of $\fk$, the (non--zero) Lie brackets are as follows
\begin{gather*}
[e_1, e_2] ={\bf -\frac{8\sqrt{2}}{9}e_3}, [e_1, e_3] ={\bf -\frac{2\sqrt{2}}{3}e_4}, [e_1, e_4] ={\bf \frac{8}{3}e_6}, [e_1, e_6] = e_2,\\
 [e_2, e_3] = {\bf \frac{2\sqrt{2}}{3}e_5}, [e_2, e_5] ={\bf -\frac{8}{3}e_6}, [e_2, e_6] = -e_1, [e_3, e_4] = -\frac{3\sqrt{2}}{2}e_1,\\
 [e_3, e_5] = \frac{3\sqrt{2}}{2}e_2, [e_4, e_5] =2\sqrt{2}e_3, [e_4, e_6] = -e_5, [e_5, e_6] = e_4.
\end{gather*}
The bold font denotes the difference from the brackets on $\gr(\fk)$.  In this case, \eqref{ex_nenorm5} is not a normal extension. So, we still need to add a map  $\fk_{-3}\oplus\fk_{-2}\oplus \fk_{-1}\to \fg$ of positive homogeneity such that the normality condition \eqref{norm} is satisfied. This involves solving a system of linear equations (in several steps according to the homogeneity) and provides the following normal regular extension $(\alpha,i)$ given by
\begin{align} 
\alpha(x_1,\dots,x_6)&= \gr(\alpha)(x_1,\dots,x_6) \nonumber+ \frac49 \left[ \begin {smallmatrix} 0&  x_2&   x_1&- x_3&   x_5&   x_4&0\\  
0&0&0&-  \sqrt {2} x_1&\frac{\sqrt {2}}{2}  x_3 &0&- x_4\\ 
0&0&0& \sqrt {2} x_2&0&-\frac{\sqrt {2}}{2} x_3&-  x_5\\  
0&0&0&0& \sqrt {2} x_2&- \sqrt {2} x_1&-  x_3\\
0&0&0&0&0&0&- x_1\\
0&0&0&0&0&0&- x_2\\
 0&0&0& 0&0&0&0\end{smallmatrix} \right].
\end{align}
 The curvature is as follows, where we denote $z_{jk}=x_jy_k-x_ky_j$
\begin{gather} \label{235-cur}
\kappa(x_1,\dots,x_6,y_1,\dots,y_6)=\frac89
\left[ \begin {smallmatrix} 
0&  0&  0&0&  0&   0&0\\  
0&z_{42}+z_{15}&z_{14}+z_{52}&0&0&0&0\\ 
0&z_{25}+z_{41}&z_{51}+z_{24}& 0&0&0&0\\  
0&0&0&0&0&0&0\\
0&0&0&0&z_{42}+z_{15}&z_{41}+z_{25}&0\\
0&0&0&0&z_{14}+z_{52}&z_{51}+z_{24}&0\\
 0&0&0& 0&0&0&0\end {smallmatrix} \right].
\end{gather}

\subsection{Extension in CR geometry and theory of systems of PDEs} \label{CR-alp} \label{lagr-alp}
We consider two different real forms of parabolic geometries of type $(PGL(n+2,\C),P_{1,n+1})$.

\begin{enumerate}
\item[$(CR)$] Regular normal parabolic geometries of type $(PSU(p+1,q+1),P_{1,n+1})$, $n=p+q$, correspond to non--degenerate partially integrable almost CR structures of hypersurface type of (real) dimension $2n+1$ and signature $(p,q)$, \cite[Section 4.2.4.]{parabook}. In particular, these describe all real hypersurfaces in $\C^{p+q+1}$ with everywhere non--degenerate Levi form with signature $(p,q)$.
\item[$(LC)$] Regular normal parabolic geometries of type $(PGL(n+2,\R),P_{1,n+1})$ correspond to Lagrangean contact structures, \cite[Section 4.2.3.]{parabook}. In particular, these describe all complete systems of $2$nd order PDEs of one unknown function of several variables (considered up to point transformations)
\[
\frac{\partial^2u}{\partial y^i\partial y^j}= f_{ij} (y, u, \partial u),\ \  1 \leq i, j \leq n.
\]
\end{enumerate}
Torsion--free parabolic geometries of type $(PGL(n+2,\C),P_{1,n+1})$ obtained from (1) by complexification (in real analytic setting) describe compatible complete systems of 2nd order PDEs of one unknown complex function of several complex variables (considered up to point transformations) describing  Segre varieties, while those obtained from (2) are just the same equations but interpreted over complex numbers.

 We write elements of 
$\frak{su}(p+1,q+1)$ as $(1,n,1)$--block matrices with contact grading as follows $$\left[ \begin{smallmatrix} a & Z & {\rm i}z \\ X & A & -JZ^* \\ {\rm i}x & -X^*J & -\bar{a} \end{smallmatrix} \right],$$
where $A \in \frak{u}(p,q)$ and $a \in \C$ such that $a+tr(A)-\bar{a}=0$ form $\fg_0$, $X \in \C^{n} \simeq \fg_{-1}$, $Z \in \C^{n*} \simeq \fg_{1}$, $x \in \R \simeq \fg_{-2}$, $z \in \R \simeq \fg_2$. Here $J$ is the diagonal matrix of order $n$ with the first $p$ entries equal to $1$ and last $q$ entries equal to $-1$.
We write elements of 
$\frak{sl}(n+2,\R)$ as $(1,n,1)$--block matrices with contact grading as follows $$\left[ \begin{smallmatrix} a & Z_1 & z \\ X_1 & A & Z_2 \\ x & X_2 & b \end{smallmatrix} \right],$$
where $A \in \frak{gl}(n)$ and $a \in \R$ such that $a+tr(A)+b=0$ form $\fg_0$, $X_1,X_2^t \in \R^{n}$ form $\fg_{-1}$, $Z_1,Z_2^t$ form $\R^{n*} \simeq \fg_{1}$, $x \in \R \simeq \fg_{-2}$ and $z \in \R \simeq \fg_2$.
In both cases, $\fp_{1,n+1}$ corresponds to block upper triangular matrix. 

Let us emphasize that although the gradings look very similar, there is a significant difference in the Levi brackets 
$\wedge^2 \fg_{-1}\to \fg_{-2}$ that among other things causes the most differences in the two examples which we will focus on.
Let us consider the tubular hypersurface in $\mathbb{C}^3$ with coordinates $(z_1,z_2,w)$
$$\Re(w)=\Re(z_2)^2-\ln(\Re(z_1)),$$
and the  system of  PDEs in two variables with jet coordinates $(y_1,y_2,u,u_1,u_2)$
$$u_{11}=u_1^2,\ \ u_{12}=0,\ \ u_{22}=0.$$
Both geometries are known to be locally homogeneous, \cite{homogCR,homogLagr}, and it turns out that their Lie algebras of infinitesimal automorphisms only differ in one--dimensional subalgebra.  In particular, their complexification is the same geometry of type $(PGL(n+2,\C),P_{1,n+1})$ and the one--dimensional subalgebras in CR case becomes an imaginary part and in LC case a real part of the complexification. This provides remarkable differences and similarities which allow us to compare these geometries.

For the CR example, we rearrange the holomorphic vector fields from \cite[Table 8]{homogCR} generating the Lie algebra of infinitesimal automorphisms so that their evaluation at $o_{CR}:=(1,0,0) \in \mathbb{C}^3$
allows us to identify the associated grading $\gr(\fk_{CR})$ of the Lie algebra $\fk_{CR}$ as
\begin{align*}
\fh_{CR}=\fk_{CR}^0&=\langle {\rm i}z_2{\partial_{z_2}}+{\rm i}z_2^2{\partial_{w}}, {\rm i}(z_1^2-1){\partial_{z_1}}-2{\rm i}(z_1-1){\partial_{w}}\rangle(o_{CR}),\\
\fk_{CR}^{-1}/\fk_{CR}^0&=\langle \sqrt{2}z_1{\partial_{z_1}}-\sqrt{2}{\partial_{w}},{\partial_{z_2}}+2z_2{\partial_{w}},{ \scriptstyle\frac{\sqrt{2}}{2}}{\rm i}(1+z_1^2){\partial_{z_1}}-\sqrt{2}{\rm i}z_1{\partial_{w}},{\rm i}{\partial_{z_2}}\rangle(o_{CR}),\\
\fk_{CR}^{-2}/\fk_{CR}^{-1}&=\langle -{\rm i}{\partial_{w}}\rangle(o_{CR}).
\end{align*} 
Observe that the following complex structure on $\fk_{CR}^{-1}/\fk_{CR}^0$ corresponds to the complex tangent bundle
\begin{align*}
{\rm i}(\sqrt{2}z_1{\partial_{z_1}}-\sqrt{2}{\partial_{w}})(o_{CR})&={\rm i}\sqrt{2}{\partial_{z_1}}-{\rm i}\sqrt{2}{\partial_{w}},\\
{\rm i}({\partial_{z_2}}+2z_2{\partial_{w}})(o_{CR})&={\rm i}{\partial_{z_2}}.
\end{align*}

Analogously, for the LC example, we rearrange infinitesimal point transformations from \cite[Table A.2]{homogLagr} generating the Lie algebra of infinitesimal automorphisms so that their evaluation at $o_{LC}:=(0,0,0,0,0)$ allows us to identify the associated grading $\gr(\fk_{LC})$ of the Lie algebra $\fk_{LC}$ as
\begin{align*}
\fh_{LC}=\fk_{LC}^0&=\langle -y_1{\partial_{y_1}}+u_1{\partial_{u_1}},-y_2{\partial_{y_2}}+u_2{\partial_{u_2}}\rangle(o_{LC}),\\
\fk_{LC}^{-1}/\fk_{LC}^0&=\langle \partial_{y_1},\partial_{y_2},  -y_1^2 \partial_{y_1}+y_1 \partial_{u}+(2u_1y_1+1)\partial_{u_1},y_2\partial_{u}+\partial_{u_2}\rangle(o_{LC}),\\
\fk_{LC}^{-2}/\fk_{LC}^{-1}&=\langle \partial_{u}\rangle(o_{LC}).
\end{align*}
Observe that the following subspaces of $\fk_{LC}^{-1}/\fk_{LC}^0$  correspond to the distinguished (Lagrangean) subdistributions of the tangent space
\begin{align*}
E&=\langle \partial_{y_1},\partial_{y_2}\rangle(o_{LC}),\\
V&=\langle -y_1^2 \partial_{y_1}+y_1 \partial_{u}+(2u_1y_1+1)\partial_{u_1},y_2\partial_{u}+\partial_{u_2}\rangle(o_{LC}).
\end{align*}
Let us emphasize here two facts:
\begin{itemize}
\item[(i)] The Lie brackets on $\fk_{CR}$ and $\fk_{LC}$ are \emph{minus} the brackets of the corresponding vector fields.
\item[(ii)] The rearrangements are invariant for  $\fh_{CR}$ and $\fh_{LC}$, respectively, which is specific for the examples.
\end{itemize}
Le us write
$$
\fk_{-2}\oplus\fk_{-1}\oplus\fk_{0}:=\langle x_1e_1\rangle \oplus \langle x_2e_2+x_3e_3+x_4e_4+x_5e_5\rangle \oplus \langle x_6e_6+x_7e_7\rangle
$$
for both gradings of $\gr(\fk_{CR})$ and $\gr(\fk_{LC})$, where in the CR case we consider
\begin{gather*}
e_1:=-{\rm i}{\partial_{w}},\\ 
e_2:=\sqrt{2}z_1{\partial_{z_1}}-\sqrt{2}{\partial_{w}},\ \ 
e_3 :={\partial_{z_2}}+2z_2{\partial_{w}},\\  
e_4:={\scriptstyle \frac{\sqrt{2}}{2}}{\rm i}(1+z_1^2){\partial_{z_1}}-\sqrt{2}{\rm i}z_1{\partial_{w}},\ \
e_5:={\rm i}\partial_{z_2},\\
e_6:={\scriptstyle\frac{3}{2}}{\rm i}(z_1^2-1){\partial_{z_1}}+ {\rm i}z_2{\partial_{z_2}}+{\rm i}(z_2^2-3(z_1-1)){\partial_{w}},\\
e_7:={\scriptstyle \frac{1}{2} }{\rm i}(z_1^2-1){\partial_{z_1}}+ 3{\rm i}z_2{\partial_{z_2}}+{\rm i}(3z_2^2-(z_1-1)){\partial_{w}},
\end{gather*}
and in the LC case we consider
\begin{gather*}
e_1:=\partial_{u},\\ 
e_2:=\partial_{y_1},\ \ 
e_3:=\partial_{y_2},\\ 
e_4:=-y_1^2 \partial_{y_1}+y_1 \partial_{u}+(2u_1y_1+1)\partial_{u_1},\ \ 
e_5:=y_2\partial_{u}+\partial_{u_2},\\
e_6:={\scriptstyle \frac32} y_1{\partial_{y_1}}-{\scriptstyle \frac32}u_1{\partial_{u_1}}+{\scriptstyle \frac12} y_2{\partial_{y_2}}-{\scriptstyle \frac12}u_2{\partial_{u_2}},\\
e_7:={\scriptstyle \frac12} y_1{\partial_{y_1}}-{\scriptstyle \frac12}u_1{\partial_{u_1}}+{\scriptstyle \frac32} y_2{\partial_{y_2}}-{\scriptstyle \frac32}u_2{\partial_{u_2}},
\end{gather*}
Then we obtain
\begin{align} \label{ex_nenorm}
\gr(\alpha_{CR})(x_1,\dots,x_7)= \left[\begin{smallmatrix}
 -{\rm i}(x_6+x_7)& 0& 0 & 0\\
  x_2+{\rm i}x_4& 2{\rm i}x_6& 0 & 0\\
x_3+{\rm i}x_5& 0& 2{\rm i}x_7 & 0\\
 {\rm i}x_1& -x_2+{\rm i}x_4& -x_4+{\rm i}x_5&  -{\rm i}(x_6+x_7)
\end{smallmatrix}\right]
\end{align}
\begin{align} \label{ex_nenorm2}
\gr(\alpha_{LC})(x_1,\dots,x_7)= \left[\begin{smallmatrix}
 -x_6-x_7& 0& 0 & 0\\
  x_2& 2x_6& 0 & 0\\
x_3& 0& 2x_7 & 0\\
x_1& x_4& x_5&  -x_6-x_7
\end{smallmatrix}\right].
\end{align}
 by identification $\fk_{-2}\oplus\fk_{-1}=\fg_{-2}\oplus\fg_{-1}$ and by finding elements of $\fg_0$ acting as $e_6,e_7$ on it. The $H$--invariance of the gradings  $\gr(\fk_{CR})$ and $\gr(\fk_{LC})$ implies that $\gr(\alpha_{CR})$ and $\gr(\alpha_{LC})$ is a regular infinitesimal extensions of $(\fk_{CR},\fh_{CR})$ to $(\frak{su}(1,3),\fp_{1,3})$ and of $(\fk_{LC},\fh_{LC})$ to $(\frak{sl}(4,\R),\fp_{1,3})$, respectively.

For the chosen bases, the non--zero Lie brackets (that are minus Lie brackets of vector fields) are
\begin{gather*}
[e_2, e_4] = -2e_1{\bf-\frac34 e_6+\frac14 e_7}, [e_2, e_6] = -3e_4, [e_2, e_7] = -e_4,  [e_3, e_5] = -2e_1,\\
  [e_3, e_6] = -e_5, [e_3, e_7] =-3e_5, [e_4, e_6] = 3e_2, [e_4, e_7] = e_2, [e_5,e_6] = e_3, [e_5, e_7] = 3e_3.
\end{gather*}
on $\fk_{CR}$, where in bold font is the difference of Lie brackets in $\fk_{CR}$ and $\gr(\fk_{CR})$, and
\begin{gather*}
[e_2, e_4] =  -e_1{\bf +\frac34 e_6 -\frac14 e_7}, [e_2, e_6] = -3e_ 2, [e_2, e_7] = -e_{2},  [e_3, e_5] = {\bf -}e_1,\\
  [e_3, e_6] = -e_{3}, [e_3, e_7] =-3e_{ 3}, [e_4, e_6] = 3e_{ 4}, [e_4, e_7] = e_{ 4}, [e_5,e_6] = e_{ 5}, [e_5, e_7] = 3e_{ 5}.
\end{gather*}
on  $\fk_{LC}$, where in bold font is the difference of Lie brackets in $\fk_{LC}$ and  $\gr(\fk_{LC})$.

Extensions \eqref{ex_nenorm} and \eqref{ex_nenorm2} are not normal. So, we still need to add maps  $\fk_{-2}\oplus \fk_{-1}\to \fg$ of positive homogeneity such that the normality conditions \eqref{norm} are satisfied. This involves solving a system of linear equations (in several steps according to the homogeneity) and provides the following normal regular infinitesimal extensions
\begin{align} \label{CR-rozsireni}
\alpha_{CR}(x_1,\dots,x_7)= \gr(\alpha_{CR})(x_1,\dots,x_7)
+ \left[\begin{smallmatrix}
  \frac{1}{24}{\rm i}x_1 & \frac{5}{24}(x_2-{\rm i}x_4)&  \frac{1}{24}({\rm i}x_5-x_3) & \frac{13}{576}{\rm i}x_1\\
  0& -\frac{1}{6}{\rm i}x_1& 0 &*\\
0& 0&  \frac{1}{12}{\rm i}x_1 &  *\\
 0& 0& 0&   \frac{1}{24}{\rm i}x_1
\end{smallmatrix}\right]
\end{align}

\begin{align} \label{CR-rozsireni2}
\alpha_{LC}(x_1,\dots,x_7)= \gr(\alpha_{LC})(x_1,\dots,x_7)
+\left[\begin{smallmatrix}
  -\frac{1}{12}x_1 & \frac{5}{12}x_4&  -\frac{1}{12}x_5 & \frac{13}{144}x_1\\
  0& \frac{1}{3}x_1& 0 &\frac{5}{12}x_2\\
0& 0&  -\frac{1}{6}x_1 &  -\frac{1}{12}x_3\\
 0& 0& 0&   -\frac{1}{12}x_1
\end{smallmatrix}\right].
\end{align}
Let us again emphasize that the differences of $\alpha_{CR}$ and $\alpha_{LC}$ are only caused by the usual choice of gradings of $\frak{su}(1,3),\frak{sl}(4,\R)$ we made at the beginning.

 The curvatures are as follows, where we denote $z_{jk}=x_jy_k-x_ky_j$ and entries denoted by $*$ follow our conventions
\begin{gather} \label{CR-cur}
\kappa_{CR}(x_1,\dots,x_7,y_1,\dots,y_7)=
\left[\begin{smallmatrix}
   0 &  \frac{1}{48}(z_{14}+{\rm i}z_{12})&\frac{1}{48}(z_{51}+{\rm i}z_{31}) &\frac{{\rm i}}{24}(z_{42}+z_{35})\\
   0&\frac{{\rm i}}{3}(z_{53}+z_{24})&\frac{1}{6}(z_{23}+z_{45}+{\rm i}(z_{43}+z_{52})) & *\\
 0&  *& \frac{{\rm i}}{3}( z_{42}+z_{35}) &    *\\
  0& 0& 0&   0
 \end{smallmatrix}\right]
\end{gather}
\begin{gather} \label{CR-cur2}
\kappa_{LC}(x_1,\dots,x_7,y_1,\dots,y_7)=
\left[\begin{smallmatrix}
   0 &  \frac{1}{12}z_{41}&\frac{1}{12}z_{15} &\frac{{\rm i}}{12}(z_{42}+z_{35})\\
   0&\frac{1}{3}(z_{35}+z_{42})&\frac{1}{3}z_{25} & \frac{1}{12}z_{12}\\
 0&  \frac{1}{3}z_{34}& \frac{1}{3}( z_{53}+z_{24}) &   \frac{1}{12}z_{31}\\
  0& 0& 0&   0
 \end{smallmatrix}\right].
\end{gather}
Let us emphasize that the normalization conditions do not determine infinitesimal extensions $\alpha_{CR},\alpha_{LC}$ uniquely, because $\exp(\fg_2)$ acts trivially on the curvatures. The additional choice to determine $\alpha_{CR},\alpha_{LC}$ uniquely is to assume that image of $e_1$ in $\fg_0$ acts trivially on $\fg_{-2}$, i.e., the corresponding Weyl connection is closed.

\subsection{Extension in theory of systems of ODE's} \label{path-alp}

Regular normal parabolic geometries of type $(PGL(n+2,\R),P_{1,2})$ correspond to path geometries, \cite[Section 4.4.3.]{parabook}. We write elements of 
$\fg=\frak{sl}(n+2,\R)$ as $(1,1,n)$--block matrices with a $|2|$--grading as follows $$\left[ \begin{smallmatrix} a & z&  Z_2  \\ x & b& Z_1 \\ X_2 & X_1& A   \end{smallmatrix} \right],$$
where $A \in \frak{gl}(n)$ and $a \in \R$ such that $a+tr(A)+b=0$ form $\fg_0$, $X_i \in \R^{n} \subset \fg_{-i}$, $Z_i \in \R^{n*} \subset \fg_{i}$, $x \in \R \subset\fg_{-1}$, $z \in \R \subset \fg_1$. Then $\fp_{1,2}$ corresponds to block upper triangular matrices. Let us recall that parabolic geometries of type $(PGL(n+2),P_{1,2})$ can be (locally) viewed as $2$nd order systems of ODE of $n$ unknown functions in one variable (considered up to point transformations), i.e.
\[
\ddot{u}_i= f_{i} (t, u_j, \dot{u}_j),\ \  1 \leq i \leq n.
\]
Let us consider the following system of ODEs with jet coordinates $(t,u_1,u_2,p_1,p_2)$
$$\ddot{u}_1=0,\ \ \ddot{u}_2=p_1^3,$$
that is known to be submaximally (locally) homogeneous, \cite{KT}. We rearrange the vector fields from \cite[Proposition 5.3.2]{KT} generating the Lie algebra of infinitesimal point transformation so that their evaluation at $o:=(0,0,0,0,0)$ allows us to identify the underlying path geometry
\begin{align*}
\fh^{1}_f=\fk^{1}_f&= \langle t^2\partial_{t}+tu_1\partial_{u_1}+tu_2\partial_{u_2}+{\scriptstyle \frac12}u_1^3\partial_{u_2}-(p_1t-u_1)\partial_{p_1}+\\ &\ \ \ \ (u_2 -p_2t+{\scriptstyle \frac32} p_1u_1^2)\partial_{p_2} \rangle(o),\\
\fh^{0}_f/\fh^{1}_f=\fk^{0}_f/\fk^{1}_f&=\langle u_1{\partial_{u_2}}+p_1{\partial_{p_2}}, -t{\partial_{t}}-u_1{\partial_{u_1}}-2u_2{\partial_{u_2}}-p_2{\partial_{p_2}},\\ 
&\ \ \ \ 3t{\partial_{t}}+2u_1{\partial_{u_1}}+3u_2{\partial_{u_2}}-p_1{\partial_{p_1}}
\rangle(o),\\
\fk^{-1}_f/\fk^{0}_f&=\langle -3\partial_{t},  -{\scriptstyle \frac13}t \partial_{u_1}-{\scriptstyle \frac12}u_1^2 \partial_{u_2}-{\scriptstyle \frac13}\partial_{p_1}-u_1p_1\partial_{p_2},-{\scriptstyle \frac13}(t\partial_{u_2}+\partial_{p_2})\rangle(o),\\
\fk^{-2}_f/\fk^{-1}_f&=\langle \partial_{u_1},\partial_{u_2}\rangle(o).
\end{align*}
Nevertheless, we choose a smallest Lie subalgebra $\fk$ of $\fk_f$ which allows us to find the infinitesimal extension describing the parabolic geometry. Thus we require $\fk^{-1}/\fk^{0}=\fk^{-1}_f/\fk^0_f$ and $\fk^{-2}/\fk^{-1} =\fk^{-2}_f/\fk^{-1}_f$, 
and it turns out that $$\fh=\fk^{0}=\langle u_1{\partial_{u_2}}+p_1{\partial_{p_2}}\rangle(o).$$  
 This choice makes no difference locally, and we show later what difference it makes from the global viewpoint.

Observe that the following subspaces of $\fk^{-1}/\fk^0$ correspond to the distinguished subdistributions of the tangent space, i.e., they are preserved by $\fh$
\begin{align*}
E&=\langle -3\partial_{t}\rangle(o),\\
V&=\langle -{\scriptstyle \frac13}t \partial_{u_1}-{\scriptstyle \frac12} u_1^2 \partial_{u_2}-{\scriptstyle \frac13}\partial_{p_1}-u_1p_1\partial_{p_2},-{\scriptstyle \frac13}(t\partial_{u_2}+\partial_{p_2})\rangle(o).
\end{align*}
Let us  write
$$
\fk_{-2}\oplus\fk_{-1}\oplus\fk_{0}:=\langle x_1e_1+x_2e_2\rangle \oplus \langle x_3e_3+x_4e_4+x_5e_5\rangle \oplus \langle x_6e_6\rangle,
$$
where
\begin{gather*}
e_1:=\partial_{u_1},\ \ 
e_2:=\partial_{u_2},\\ 
e_3:=-3\partial_{t},\ \
e_4:= -{\scriptstyle \frac13}t \partial_{u_1}-{\scriptstyle \frac12}u_1^2 \partial_{u_2}-{\scriptstyle \frac13}\partial_{p_1}-u_1p_1\partial_{p_2},\ \
e_5:=-{\scriptstyle \frac13}(t\partial_{u_2}+\partial_{p_2}),\\ 
e_6:=u_1{\partial_{u_2}}+p_1{\partial_{p_2}}.
\end{gather*}
Again, by identifying $\fk_{-2}\oplus\fk_{-1}=\fg_{-2}\oplus\fg_{-1}$ and finding element of $\fg_0$ acting as $e_6$, we obtain a regular infinitesimal extension of $(\fk,\fh)$ to $(\frak{sl}(4,\R),\fp_{1,2})$
\begin{align} \label{ex_nenorm3}
\gr(\alpha)(x_1,\dots,x_6)= \left[\begin{smallmatrix}
0& 0& 0 & 0\\
  x_3& 0& 0 & 0\\
x_1& x_4& 0 & 0\\
x_2& x_5& x_6& 0
\end{smallmatrix}\right].
\end{align}

For the chosen basis of $\fk$, the (non--zero) Lie brackets (that are minus Lie brackets of vector fields) are as follows
\begin{gather*}
[e_1, e_4] ={\bf e_6}, [e_1, e_6] = -e_2, [e_3, e_4] = -e_1, [e_3, e_5] = -e_2, [e_4, e_6] = -e_5.
\end{gather*}
The bold font denotes the difference from the bracket on $\gr(\fk)$, which is in this case the harmonic curvature of the geometry, i.e., this is already a normal extension. Thus $\alpha=\gr(\alpha)$ and the curvature takes form
\begin{gather} \label{Ode-cur}
\kappa(x_1,\dots,x_6,y_1,\dots,y_6)=
\left[\begin{smallmatrix}
   0 &  0 &0&0\\
   0&0&0 & 0\\
 0&  0& 0 &  0\\
  0& 0& -x_1y_4+x_4y_1&   0
 \end{smallmatrix}\right].
\end{gather}

\section{First BGG operators on tractor bundles and their solutions} \label{bggtheory}

In the section, we assume that $(\alpha,i)$ is a regular normal extension of $(K,H)$ to $(G,P)$. Further, we assume that the vector bundle $$\mathcal{V}:=K\times_{\lambda\circ i(H)}V$$ is such that the representation $d\lambda$ of $\fp$ on $V$ extends to the representation $\rho$ of $\fg$ on $V$.

\subsection{Tractor bundles and connections}
Our assumptions specify a special class of vector bundles that always admit a natural connection.

\begin{def*}
We say that the vector bundle $\mathcal{V}$ is a \emph{tractor bundle} and that the $K$--invariant linear connection
\begin{align*}
\nabla^{\rho\circ \alpha}=D^\fk+\rho\circ \alpha
\end{align*}
is the \emph{tractor connection} on the tractor bundle $\mathcal{V}$. 
\end{def*}

Let us emphasize that $\rho\circ \alpha(Y)=d\lambda\circ \alpha(Y)$ for all $Y\in \fh$ and thus $\nabla^{\rho\circ \alpha}$ is indeed $K$--invariant according to Lemma \ref{inv-con}. 
\begin{rem}
In fact, $\rho \circ \alpha$ is a representation of $\fk$ if and only if $\kappa=0$. That is, $\mathcal{V}$ is not a tractor bundle from the viewpoint of the category of Cartan geometries of type $(K,H)$, in general.
\end{rem}
The representation  $\rho$ induces an algebraic differential $\partial_{\fg}: \wedge^{k}\fg^*\otimes V\to \wedge^{k+1}\fg^*\otimes V$ 
\begin{align*}
\partial_{\fg} &\phi(X_0, \dots, X_k)=\sum_{j}
(-1)^j\rho(X_j)(\phi(X_0, \dots,\hat X_j,\dots, X_k))\\
&+(-1)^{j+l}\sum_{j<l}\phi([X_j,X_l]_\fg,X_0, \dots,\hat X_j,\dots,\hat X_l,\dots, X_k).
\end{align*}
This implies that the components of homogeneity $0$ of the covariant exterior derivative corresponding to the tractor connection and an algebraic differential  $\partial_{\fg_-}: \wedge^{k}\fg_-^*\otimes V\to \wedge^{k+1}\fg_-^*\otimes V$ coincide.

\begin{def*}
The operator
\begin{align*}
\Box:= \partial^*\circ \partial_{\fg_-}+ \partial_{\fg_-}\circ \partial^*: \wedge^{k}\fp_+\otimes V\to \wedge^{k}\fp_+\otimes V
\end{align*}
is called the \emph{Kostant's Laplacian}.
\end{def*}

The crucial fact for the construction of BGG operators on tractor bundles is that the restriction of $\Box$ to $Im(\partial^*)$ is invertible, because the eigenspaces of $\Box$  are $G_0$--subrepresentations of $\wedge^{k}\fg\otimes V$. In particular, the inversion of $\Box$ is a polynomial in `variable' $\partial^*\circ \partial_{\fg_-}$ and this polynomial (its coefficients) will play the role of the operator $Q$ required in condition \hyperref[cc]{$\ddd$}.

\subsection{Standard and non--strandard BGG operators}
The component of homogeneity $0$ of $\partial^*\circ \nabla^{\rho\circ \alpha}$ acting on $Ker(\partial^*)$ coincides with the Kostant's Laplacian $\Box$. Thus $ \nabla^{\rho\circ \alpha}-\partial_{\fg_-}$ is an operator of homogenity $>0$. Define $Q: Ker(\partial^*)\to Ker(\partial^*)$ as a polynomial in $\partial^*\circ \nabla^{\rho\circ \alpha}$ having the same coefficients as the inverse of $\Box$ also viewed as a polynomial in $\partial^*\circ \partial_{\fg_-}$.
Then $Q$ satisfies the condition \hyperref[cc]{$\ddd$}.

Moreover, if we consider another $K$--invariant linear connection $\nabla^{\rho\circ \alpha+\psi}$ for some $H$--invariant element $$\psi\in (\fk^*\otimes \frak{gl}(V))^1$$ vanishing on $\fh$, then we can define $Q^\psi$ as the polynomial in $\partial^*\circ \nabla^{\rho\circ \alpha+\psi}$ with the same coefficients as $Q$. For the same reasons,  $Q^\psi$ also satisfies the condition \hyperref[cc]{$\ddd$}. We denote by $\mathcal{L}_0$ and $\mathcal{L}_0^\psi$, respectively, the corresponding \emph{splitting operators} and define following classes of first BGG operators.

\begin{def*}
The $K$--invariant linear connection $\nabla^\Phi$ on the tractor bundle $\mathcal{V}$ for the representation $\rho$ of $\fg$ on $V$ is \emph{compressable} if
$$\Phi=\rho\circ \alpha+\psi$$ for some $H$--invariant element $\psi\in (\fk^*\otimes \frak{gl}(V))^1$ vanishing on $\fh$. 
The operator $$\mathcal{D}^{st}:=\pi_1\circ \nabla^{\rho\circ \alpha}\circ \mathcal{L}_0$$ for the $K$--invariant compressable connection $\na^{\rho \circ \alpha}$ is the \emph{standard first BGG operator}, while the other operators $$\mathcal{D}=\pi_1\circ \nabla^{\rho\circ \alpha+\psi}\circ \mathcal{L}_0^\psi$$ for compressable $K$--invariant linear connections $\nabla^{\rho\circ \alpha+\psi}$ are \emph{non--standard first BGG operators}.
\end{def*}

Clearly, if $\psi(s)\in Im(\partial^*)$ for all $H$--equivariant functions $s: K\to V$, then also the compressable $K$--invariant linear connection $\nabla^{\rho\circ \alpha+\psi}$ satisfies $\mathcal{L}_0=\mathcal{L}_0^\psi$ and $$\mathcal{D}^{st}=\pi_1\circ \nabla^{\rho\circ \alpha+\psi}\circ \mathcal{L}_0.$$

\subsection{The prolongation connection} \label{sec-prol}

Assume $s$ is a parallel section of a compressable $K$--invariant linear tractor connection $\nabla^{\rho\circ \alpha+\psi}$ on a tractor bundle $\mathcal{V}$.
Then 
$$\nabla^{\rho\circ \alpha+\psi}\circ \mathcal{L}_0^\psi(\pi_0(s))=\nabla^{\rho\circ \alpha+\psi} s-\nabla^{\rho \circ \alpha+\psi}\circ Q^\psi\circ \partial^*\circ \nabla^{\rho\circ \alpha+\psi} s =0.$$ This shows that $\si:=\pi_0(s)$ for $s$ such that $\nabla^{\rho\circ \alpha+\psi} s=0$ is a solution of the first BGG operator 
$$D=\pi_1\circ \nabla^{\rho\circ \alpha+\psi}\circ \mathcal{L}_0^\psi,$$
 that is $D(\si)=0$. The converse is not true in general. Indeed, $\mathcal{L}_0^\psi(\sigma)$ does not have to be a parallel section of  $\nabla^{\rho\circ \alpha+\psi}$ even if $\si$ satisfies  $D(\sigma)=0$. However, if $\nabla^{\rho\circ \alpha+\psi}\circ \mathcal{L}_0^\psi(\sigma)$ is contained in $Ker(\partial^*\circ \nabla^{\rho\circ \alpha+\psi})$, or equivalently, if $\partial^*\circ R^{\rho\circ \alpha+\psi}\circ \mathcal{L}_0^\psi(\sigma)=0$ holds for the curvature $R^{\rho\circ \alpha+\psi}$ of $ \nabla^{\rho\circ \alpha+\psi}$, then $\nabla^{\rho\circ \alpha+\psi}\circ \mathcal{L}_0^\psi(\sigma)=\mathcal{L}_1^\psi\circ D(\sigma)=0$
for $\si$ such that  $D(\sigma)=0$.

\begin{prop}[\cite{new-norm,casimir}]
There is a unique $\Psi \in (\fk^*\otimes \frak{gl}(V))^1$ vanishing on $\fh$ such that $\partial^*\circ R^{\rho\circ \alpha+\Psi}(s)=0$
 for all $s\in V$.
Then all solutions of $D=\pi_1\circ \nabla^{\rho\circ \alpha+\Psi}\circ \mathcal{L}_0^{\Psi}$ are in bijective correspondence with parallel sections of the connection $\nabla^{\rho\circ \alpha+\Psi}$.
\end{prop}
\begin{def*}
The connection $\nabla^{\rho\circ \alpha+\Psi}$  is the \emph{prolongation connection} of $D$.
\end{def*}
We show how to construct $\Psi$ from arbitrary $\psi$ by the following iterative process: 
\begin{enumerate}
\item
Set
$\Psi_0:=\psi$.
\item Compute 
$$\Psi_k:=\Psi_{k-1}-\frac{1}{a_k}(\partial^*\otimes \id_{V^*})\circ R^{\rho\circ \alpha+\Psi_{k-1}},$$
where $a_k$ depends on eigenvalues of the action of $\Box$ on the space $\fp_+\otimes V$. 
\item Since the image of $\partial^*\otimes \id_{V^*}$ does not lower the homogeneity and coincides with $\Box$ in the lowest nonzero homogeneity of the image, we get $\Psi:=\Psi_k$ in finitely many steps after ordering the eigenvalues according to homogeneity, because there is only a finite number of $\fg_0$--components in $\fp_+\otimes V$.
\end{enumerate}

\begin{rem}
In general, the eigenvalues of the action of $\Box$ on the space $\fp_+\otimes V$ can be computed using formulas from \cite{casimir}. However, in many cases, it suffices to choose directly $a_k$  that kills some part of $(\partial^*\otimes \id_{V^*})\circ R^{\rho\circ \alpha+\Psi_{k}}$ in the lowest homogeneity.
\end{rem}

There is the following condition for the tractor connection to coincide with the prolongation connection on homogeneous parabolic geometries.

\begin{prop}\label{norm2}
If $\psi=0$ and $\sum_i \rho(\kappa(X,X_i))(\rho(Z_i)s)=0$ holds for all $X\in \fg$ and $s\in V$, then $\Psi=0$, where
the elements $X_i\in \fg$ are representatives of a basis of $\mathfrak{g}/\mathfrak{p}$ and the elements $Z_i\in \fp_+$ form the corresponding dual basis of $(\mathfrak{g}/\mathfrak{p})^*$.
\end{prop}
\begin{proof}
In the first step, for $\psi=0$ we get $R^{\rho\circ \alpha}=(\id_{\wedge^2 \fk/\fh^*}\otimes \rho)(\kappa)$ and thus $(\partial^*\otimes \id_{V^*})\circ R^{\rho\circ \alpha}(X,s)=2\sum_i \rho(\kappa(X,X_i))(\rho(Z_i)s)$, because $\partial^*\kappa=0$.
\end{proof}

\subsection{Automorphisms connection}

There is a distinguished connection for
$$\Phi^{aut}(s)(t)=\ad(\alpha(s))(t)-\iota_\kappa(\alpha(s))(t)$$ 
on the adjoint tractor bundle $\mathcal{A}$ whose parallel sections are infinitesimal automorphisms of the parabolic geometry, where 
$\iota_\kappa(\alpha(s))(t):=\kappa(\alpha(s),t)$ denotes the inclusion into the curvature, \cite{deform}. 
\begin{prop}
Infinitesimal automorphisms of the parabolic geometry on $K/H$ correspond to solutions of 
$$D^{aut}:=\pi_1\circ \nabla^{\ad\circ \alpha-\iota_\kappa}\circ \mathcal{L}_0^{-\iota_\kappa}.$$
In particular, $\fk\subset \mathcal{S}^\infty$ for $\nabla^{\ad\circ \alpha-\iota_\kappa}$.
\end{prop}
\begin{proof}
Since $-\iota_\kappa=\Psi$ for the corresponding prolongation connection, the infinitesimal automorphisms correspond to solutions of $D^{aut}$. Thus by our assumption, $D^{aut}$ contains $\alpha(\fk)$ as solutions. Indeed,  $$\ad(\alpha(X))(\alpha(Y))-\kappa(\alpha(X),\alpha(X))=\alpha([X,Y]_\fk)$$ holds for $X,Y\in \fk$ and thus $\fk\subset \mathcal{S}^\infty$.
\end{proof}
The first BGG operator  
$D^{aut}$
does not have to be the standard, although there are known sufficient conditions for $D^{aut}=D^{st}$. For example, \cite[Theorem 3.5.]{deform} states that $D^{aut}=D^{st}$ for torsion--free geometries that satisfy a certain condition on the first cohomology $H^1(\fp_+,\fg)$. However, we show on projective and C-projective examples in Sections \ref{solproj} and \ref{cproj-S} that $D^{aut}\neq D^{st}$  for geometries that do not satisfy such  condition.

\subsection{Normal solutions of BGG operators and holonomy reductions} \label{normal} There is the following important class of solutions, \cite{CGH1}.

\begin{def*}
The solutions $\nu$ of $D^{st}$ such that $$\nabla^{\rho\circ \alpha}\mathcal{L}_0(\nu)=0$$ are called \emph{normal solutions} of the standard first BGG operator. 
\end{def*}

Normal solutions satisfy an additional identity $\Psi(\mathcal{L}_0(\nu))=0$ for $\Psi \in (\fk^*\otimes \frak{gl}(V))^1$ describing the prolongation connection. This often means that the solutions have some additional distinguished properties, \cite{CGM,SW}.

Another important property of normal solutions is their compatibility with tensor products. Consider tractor bundles $\mathcal{V}$ and $\mathcal{W}$ for $\fg$--representations $\rho_V$ and $\rho_W$ extending representations  $\lambda^V$ of $P$ on $V$ and $\lambda^W$ of $P$ on $W$. Then $$\nabla^{(\rho_V\otimes \rho_W)\circ \alpha}(s\otimes t)=\nabla^{\rho_V\circ \alpha}(s)\otimes t+s\otimes \nabla^{\rho_W\circ \alpha}(t)$$
on the tractor bundle $\mathcal{V}\otimes \mathcal{W}$. Thus if $s=\mathcal{L}_0(\nu_V)$ and $t=\mathcal{L}_0(\nu_W)$ are solutions of $D^{st}$ on $\mathcal{V}$ and $\mathcal{W}$, then $s\otimes t$ is a solution of  $D^{st}$ on  $\mathcal{V}\otimes \mathcal{W}$. This is usually referred as \emph{BGG coupling} and to find all the normal solutions on all tractor bundles it suffices to consider the basic representations of the Lie algebra $\fg$ that generate all the others, \cite{coupling}.

On each (locally) homogeneous parabolic geometry, we can directly compute the infinitisimal holonomy of $\nabla^{\rho\circ \alpha}$ viewed as a subalgebra of $\fg$ using the formulas for the sets $\mathcal{S}^i$ from Section \ref{paralel}. This also provides the connected component of identity of the holonomy group of $\nabla^{\rho\circ \alpha}$ in $G$, which provides (by definition) the following interpretation of normal solutions as holonomy reductions, \cite{hol}.

\begin{def*}
Let $P_v$ be the stabilizer of $v\in V$ and denote $\mathcal{O}:=G/P_v=Gv\subset V$. Then we say that an $H$--equivariant function $s: K\to \mathcal{O}\subset V$ parallel with respect to the induced (non--linear) connection on $K\times_{\rho\circ i(H)}\mathcal{O}$ is a \emph{holonomy reduction of $G$--type} $\mathcal{O}$.
\end{def*}

We discuss the geometric interpretation of holonomy reductions on examples in detail in next sections.

\section{Examples of first BGG operators and local solutions} \label{bggexam}

In this Section, we continue in the discussion of (locally) homogeneous parabolic geometries from Section \ref{priklady-alp} described by the extensions $(\alpha,i)$ or infinitesimal extensions $\alpha$. We consider several different tractor bundles and study corresponding BGG operators and local solution sets $\mathcal{S}^\infty$ as follows. 
\begin{enumerate}
\item We compute the prolongation connection using the algorithm from Section \ref{sec-prol}. We compute and present here only its difference from the tractor connection. In Examples \ref{solproj}, \ref{cproj-S} we also give the difference from the automorphism connection on the adjoint tractor bundle. 
\item We use the method described in Section \ref{paralel} to compute the sets $S^\infty$ for (local) solutions and normal (local) solutions of the corresponding BGG operators. We also compute the full Lie algebra of infinitesimal automorphisms and the infinitesimal holonomy of the tractor connection  in $\fg$.
\item We present the examples with different amount of details depending on the parts of the theory we would like to demonstrate on them.
\end{enumerate}
We implemented the computations in {\sc Maple} and in most cases, we do not significantly edit the outputs, so the notation like names of variables and their indexing and ordering (lexicographic) usually follow the conventions in {\sc Maple}. We denote by $w_i$ real parameters and $z_i$ complex parameters and we emphasize to the reader the cases in which the notation has a deeper sense.

 \subsection{Local solutions in projective geometry}\label{solproj}
 Let us present here solution spaces for BGG operators on several interesting tractor bundles for the projective example we have chosen in Section \ref{proj-alp}. Let us emphasize that the map $\alpha: \fk\to \frak{sl}(4,\R)$ which is crucial for the computations is of the form
 \[
 \alpha(x_1,x_2,x_3)=
 \left[ \begin{smallmatrix} 0&x_1&0&0\\ 
x_1&0&0&0\\ 
x_2&- x_2& - x_1&0\\ 
x_3&x_3& -x_1& x_1 \end{smallmatrix} \right].
\]

\noindent
{\bf (1)  Standard representation.} Let us start with the standard tractor bundle, i.e., $V=\R^4$ for the standard representation $\rho$ of $\frak{sl}(4,\R)$. We get $\R^4= [V_{{3 \over 4}} \:|\: V_{-{1 \over 4}}]$ written as $(1,3)$--block vectors. 
The prolongation connection coincides with the standard tractor connection and $\Phi=\rho \circ \alpha$ takes the form
\begin{gather*}
\Phi(x_1,x_2,x_3)([w_1\:|\:w_2,w_3,w_4])=\alpha(x_1,x_2,x_3)\cdot [w_1\:|\:w_2,w_3,w_4]^t\\
=[x_1w_2\:|\:x_1w_1, w_1x_2-w_2x_2-w_3x_1, w_1x_3+w_2x_3-w_3x_1+w_4x_1].
\end{gather*}
Therefore, all solutions are normal. To compute them using the methods from Section \ref{paralel}, we need to recall that the curvature acts on $\R^4$ as $$\rho(\kappa(x_1,x_2,x_3,y_1,y_2,y_3))([w_1\:|\:w_2,w_3,w_4])=
\left[ \begin{smallmatrix} 0&0&0&0\\  
0&0&0&0\\ 
0&0&0&0\\ 
0&2(x_1y_2-x_2y_1)&0&0 \end{smallmatrix} \right]\cdot 
 \left[ \begin{smallmatrix} w_1 \\  w_2 \\ w_3 \\ w_4 \end{smallmatrix} \right],
$$
which also provides one generator of the infinitesimal holonomy of the tractor connection.
From this and the formula for $\Phi$, it is easy to compute $\mathcal{S}^0=\{ [w_1\:|\:0,w_3,w_4] \}$ and we get
$$\mathcal{S}^1=\mathcal{S}^\infty=\{[0\:|\:0,w_3,w_4]\}.$$
We also get that the infinitesimal holonomy of the tractor connection is a two--dimensional subalgebra of $\frak{sl}(4,\R)$ consisting of elements
$$\left[ \begin{smallmatrix} 0&0&0&0\\  
0&0&0&0\\ 
0&0&0&0\\ 
h_2&h_1&0&0 \end{smallmatrix} \right].$$
Altogether, there is  a two--parameter family of normal solutions for the corresponding first standard BGG operator in this case. We show in Section \ref{proj-bgg}, how this operator and these solutions look like explicitly in coordinates.

\noindent
{\bf (2) Dual representation.} Let us now consider the dual tractor bundle, i.e., $V=\R^{4*}$ for the dual representation $\rho$ of $\frak{sl}(4,\R)$. We get $\R^{4*}= [V_{-{3 \over 4}} \:|\: V_{{1 \over 4}}]$ written as $(1,3)$--block vectors. 
Let us note that there is no natural duality between standard and dual tractor bundle. Again, prolongation connection coincides with tractor connections. In this case, $\Phi=\rho \circ \alpha$ takes the form
\begin{gather*}
\Phi(x_1,x_2,x_3)([w_1\:|\:w_2,w_3,w_4])=- [w_1\:|\:w_2,w_3,w_4]\cdot \alpha(x_1,x_2,x_3)\\
=[-w_{{2}}x_{{1}}-w_{{3}}x_{{2}}-w_{{4}}x_{{3}}\:|\:-w_{{1}}x_{{1}}+w_{{3}}
x_{{2}}-w_{{4}}x_{{3}},w_{{3}}x_{{1}}+w_{{4}}x_{{1}},-w_{{4}}x_{{1}}]
\end{gather*}
and again, all solutions are normal. As in the case of standard tractor bundle, it is easy computation to get  $$\mathcal{S}^0=\mathcal{S}^\infty=\{[w_1\:|\:w_2,w_3,0]\}.$$
Let us emphasize that the corresponding three--parameter family of normal solutions provides (on a dense open subset where the value of the solution in $\mathcal{H}^0(\R^{4*})$ is non--vanishing) a Ricci--flat affine connection in the projective class for each solution, \cite{Hammerl}. We discuss this in Section \ref{proj-bgg} in detail.

\noindent
{\bf (3) Second symmetric power of standard representation.} 
Let us consider the tractor bundle for the representation $V=S^2\R^4$ and represent its elements as symmetric $(1,3)$--block matrices $$ S^2\R^4\simeq   \left[
\begin{array}{c|c} V_{{3 \over 2}} & V_{{1 \over 2}}\\ \hline * & V_{-{1 \over 2}} \end{array}  \right].$$
Here $\rho$ is given by matrix multiplication $\Phi(x_1,x_2,x_3)(W)=(\rho\circ \alpha)(x_1,x_2,x_3)(W)=\alpha(x_1,x_2,x_3)\cdot W+W\cdot \alpha(x_1,x_2,x_3)^t.$
 The prolongation connection does not coincide with the tractor connection and using the method from Section \ref{sec-prol} 
we compute the difference $\Psi(x_1,x_2,x_3)$ between the two connections as
$$
{\footnotesize
\left [ \begin {array}{c|ccc} w_{{1}}&w_{{2}}&w_{{4}}&w_{{7}}\\ 
\hline
*&w_{{3}}&w_{{5}}&w_{{8}}
\\ *&*&w_{{6}}&w_{{9}}
\\ *&*&*&w_{{10}}\end {array}
 \right] \mapsto 
 \left[ \begin {array}{c|ccc} 0&0&0&-{2\over 3}x_{{2}}w_{3}+{2\over 3}x_{1}w_
{5}\\ \hline *&0&0&0\\ *&*&0&0
\\ *&*&*&0
\end {array} \right] 
}.
$$
 This however does not change the image of $\rho\circ \kappa$ in $\frak{gl}(S^2\R^4)$ and therefore
 $$ 
\mathcal{S}^0=\{ 
{\footnotesize
\left[
\begin{array}{c|ccc}
w_1&0&w_4&w_7\\
\hline 
*&0&0&0\\
*&*&w_6& w_9\\
*&*&*& w_{10}\\
\end{array} \right]
}
\}
$$
for both the tractor and prolongation connection. Since $\Psi$ acts trivially on $\mathcal{S}^0$, we conclude that all solutions are normal and further computations show that
$$ 
\mathcal{S}^1=\mathcal{S}^\infty=\{ 
{\footnotesize
\left[
\begin{array}{c|ccc}
0&0&0&0\\
\hline 
*&0&0&0\\
*&*&w_6& w_9\\
*&*&*& w_{10}\\
\end{array} \right]
}
\}.
$$
Let us emphasize that in this case, normal solution have known geometric interpretation, \cite{CGM}.
Again, we comment on this  in Section \ref{proj-bgg}.
\

\noindent
{\bf (4) Second symmetric power of dual representation.} 
Analogously, let us consider the tractor bundle for the representation $V=S^2\R^{4*}$ and represent its elements as symmetric $(1,3)$--block  matrices
$$S^2\R^{4*}\simeq   \left[
\begin{array}{c|c} V_{-{3 \over 2}} & V_{-{1 \over 2}}\\ \hline * & V_{{1 \over 2}} \end{array}  \right].$$
Here $\rho$ and the tractor connection are defined in the usual way, too. Again, the prolongation connection does not coincide with the tractor connection and the difference $\Psi(x_1,x_2,x_3)$ between the two connections takes the form
$$
{\footnotesize
\left[ \begin {array}{c|ccc} w_{{1}}&w_{{2}}&w_{{4}}&w_{{7}}\\ 
\hline
*&w_{{3}}&w_{{5}}&w_{{8}}
\\ *&*&w_{{6}}&w_{{9}}
\\ *&*&*&w_{{10}}\end {array}
 \right] \mapsto 
 \left[ \begin {array}{c|ccc} 0&0&0&0\\ \hline *&-{4 \over 3}x_{{2
}}w_{{7}}&{2 \over 3}x_{{1}}w_{{7}}&0\\ *&*&0&0\\ *&*&*&0\end {array} \right] 
}.
$$
However, all solutions are normal and we compute that
$$
\mathcal{S}^0=\mathcal{S}^\infty=\{ {\footnotesize
\left[
\begin{array}{c|ccc} w_1&w_2&w_4&0\\ \hline *&w_3&w_5&0
\\ *&*&w_6& 0\\ *&*&*& 0\end{array} \right] }
\}.
$$
In this case, there also is a known interpretation of solutions, \cite{Hammerl}.

\noindent
{\bf (5) Second skew--symmetric power of standard representation.} 
Let us now consider the tractor bundle for representation $\wedge^2\R^4$. We represent its elements as skew--symmetric $(1,3)$--block matrices $$ \wedge^2\R^4 \simeq 
 \left[
\begin{array}{c|c} 0 & V_{{1 \over 2}}\\ \hline * & V_{-{1 \over 2}} \end{array}  \right].$$ 
Here $\rho$ and tractor connection are defined in the usual way, too. Again, the prolongation connection does not coincide with the tractor connection and the difference $\Psi(x_1,x_2,x_3)$ between the two connections takes the form
$$
{\footnotesize
\left[ \begin {array}{c|ccc} *&w_{{1}}&w_{{2}}&w_{{4}}\\ 
\hline
*&*&w_{{3}}&w_{{5}}
\\ *&*&*&w_{{6}}
\\ *&*&*&*\end {array}
 \right] \mapsto 
 \left[ \begin {array}{c|ccc} *&0&0&-2w_3x_1\\ \hline *&*&0&0\\ *&*&*&0\\ *&*&*&*\end {array} \right] 
}.
$$
In this case, the image of the curvature $R^{\rho\circ \alpha}=\rho\circ \kappa$ is different than the image of the curvature $R^{\rho\circ \alpha+\Psi}$ in $\frak{gl}(\wedge^2\R^4)$ and takes the form
$$
{\footnotesize
\left[ \begin {array}{c|ccc} *&w_{{1}}&w_{{2}}&w_{{4}}\\ 
\hline
*&*&w_{{3}}&w_{{5}}
\\ *&*&*&w_{{6}}
\\ *&*&*&*\end {array}
 \right] \mapsto 
 \left[ \begin {array}{c|ccc} *&0&0&aw_1\\ \hline *&*&0&0\\ *&*&*&bw_3\\ *&*&*&*\end {array} \right] 
},
$$
where we have $b=-a$ for the tractor connection and $b=0$ for the prolongation connection. Therefore,
$$
\mathcal{S}^0=\{  {\footnotesize \left[
\begin{array}{c|ccc} *&0& {\bf w_2}& {\bf w_4}
\\ \hline *&*& w_3& {\bf w_5}\\ *&*&*&{\bf w_6} \\ *&*&*&*\end {array} \right] }
\}
,$$
where in bold are the entries that coincide for both the tractor connection and the prolongation connection. Finally, we compute that
$$
\mathcal{S}^1=\mathcal{S}^\infty=\{  {\footnotesize \left[
\begin{array}{c|ccc} *&0& w_2& {\bf w_4}
\\ \hline *&*& w_3& {\bf w_5}\\ *&*&*&{\bf w_6} \\ *&*&*&*\end {array} \right] }
\},
$$
where the normal solutions are in bold font.

\noindent
{\bf (6) Adjoint representation.} 
Let us now consider the adjoint representation that sits as the trace--free component in the tractor bundle $\fg+\R\id \simeq V= \R^4 \otimes \R^{4*}$, where 
$V_i=\fg_i$ for the $|1|$--grading from Section \ref{proj-alp} and $V_0=\fg_0+\R\id$. There are three different connections:

\noindent
$\bullet$ The tractor connection, i.e., $\Phi=\ad \circ \alpha$ for the adjoint action on $\R^4 \otimes \R^{4*}$.

\noindent
$\bullet$ The prolongation connection, i.e., $\Phi^{prol}=\ad \circ \alpha+\Psi$, where 
$$\Psi(x_1,x_2,x_3):{\footnotesize
 \left[ \begin {array}{c|ccc} w_{{1}}&w_{{5}}&w_{{9}}&w_{{13}}
\\ \hline w_{{2}}&w_{{6}}&w_{{10}}&w_{{14}}
\\ w_{{3}}&w_{{7}}&w_{{11}}&w_{{15}}
\\ w_{{4}}&w_{{8}}&w_{{12}}&w_{{16}}\end{array}
 \right] 
 \mapsto 
 \left[ \begin {array}{c|ccc} 0&{1 \over 2}x_{2}w_{14}-x_{1}w_{15}&{1 \over 2}x_{1}w_{14}&0\\ \hline 0&0&0&0\\ 0
&0&0&0\\ 0&x_{2}w_{2}&-x_{1}w_{2}&0
\end {array} \right] 
.
}
$$

\noindent
$\bullet$ The automorphism connection for $\Phi^{aut}=\ad \circ \alpha-\iota_\kappa$ providing the infinitesimal automorphisms, where 
$$\iota_\kappa(x_1,x_2,x_3):{\footnotesize
 \left[ \begin {array}{c|ccc} w_{{1}}&w_{{5}}&w_{{9}}&w_{{13}}
\\ \hline w_{{2}}&w_{{6}}&w_{{10}}&w_{{14}}
\\ w_{{3}}&w_{{7}}&w_{{11}}&w_{{15}}
\\ w_{{4}}&w_{{8}}&w_{{12}}&w_{{16}}\end{array}
 \right] 
 \mapsto 
 \left[ \begin {array}{c|ccc} 0&0&0&0\\ \hline 0&0&0&0\\ 0
&0&0&0\\ 0&-2w_2x_2+2w_3x_1&0&0
\end {array} \right] 
.
}
$$
We compute that tractor,  prolongation and automorphism connections have different solution sets as follows, where the solutions for the prolongation connection are characterized by $A=0$, the solutions for the automorphism connection are characterized by $A=1$ and normal solutions are in bold font
$$
\mathcal{S}^\infty=\{ {\footnotesize \left[
\begin{array}{c|ccc}  {\bf w_6}& w_5&0&0
\\ \hline w_5& {\bf w_6}&0&0\\ {\bf w_3}
& {\bf w_7-w_3}& {\bf w_{11}+w_6}-w_5&0\\ {\bf w_4}& {\bf w_8+w_4}& {\bf w_{12}}-w_5
& w_5+{\bf w_6}+Aw_{11} \end {array} \right] } \}
.$$
Therefore, for $A=1$, subtracting the trace gives all infinitesimal automorphisms of this projective geometry.

\subsection{Local solutions in C--projective geometry} \label{cproj-S}
We present solution spaces for BGG operators on C-projective example we have  chosen in Section \ref{cproj-alp}.

\noindent
{\bf (1)  Standard and conjugate representations.} Let us start with the standard tractor bundle and its conjugate for the standard representation $V=\C^4$ and its conjugate representation $V=\overline{\C^4}$ of $\frak{sl}(4,\C)$. In both cases, we can write $V=[V_{{3 \over 4}} \:|\: V_{-{1 \over 4}}]$ given as $(1,3)$--block vectors. The prolongation connection coincides with the standard tractor connection in both cases. In the standard case we get
\begin{gather*}
\Phi(x_1,\dots,x_5)([z_1 \:|\: z_2,z_3,z_4])=\alpha(x_1,\dots,x_5)\cdot [z_1 \:|\: z_2,z_3,z_4]^t
=[{\scriptstyle -{3 \over 8}}x_5z_1\:|\: x_1z_1 \\
 -{\scriptstyle{1 \over 2}}x_1z_2-{\scriptstyle{3 \over 8}}x_5z_2+x_4z_4, x_2z_1+{\scriptstyle {5 \over 8}}x_5z_3+x_3z_4,x_3z_1-{\scriptstyle{1 \over 2}}x_3z_2+{\scriptstyle{1 \over 2}}x_1z_4+{\scriptstyle {1 \over 8}}x_5z_4]
\end{gather*} 
and the conjugate case $
\Phi(x_1,\dots,x_5)([z_1 \:|\: z_2,z_3,z_4])=\overline{\alpha(x_1,\dots,x_5)}\cdot [z_1 \:|\: z_2,z_3,z_4]^t
$ looks analogously. In both cases, all solutions are normal and take form $$\mathcal{S}^\infty=\{[z_1 \:|\: 2z_1,z_3,0]\}.$$
We obtain the infinitesimal holonomy of the tractor connection as the complex two--dimensional subalgebra of $\frak{sl}(4,\C)$ consisting of elements
$$\left[ \begin{smallmatrix} 0&0&0&0\\  
0&0&0&0\\ 
h_2&-\frac12 h_2&0&h_1\\ 
0&0&0&0 \end{smallmatrix} \right].$$

\noindent
{\bf (2) Dual and conjugate dual representations.} Let us now consider dual tractor bundle and its conjugate for the dual and conjugate dual representations $V=\C^{4*}$ and $V=\overline{\C^{4*}}$ of $\frak{sl}(4,\C)$. We get $V= [V_{-{3 \over 4}} \:|\: V_{{1 \over 4}}]$ written as $(1,3)$--block vectors in both cases. Let us note that there is no natural duality between standard and dual tractor bundle. Again, prolongation connection coincides with tractor connections in both cases and are given by $\Phi(x_1,\dots,x_5)([z_1\:|\:z_2,z_3,z_4])=-[z_1\:|\:z_2,z_3,z_4]\cdot \alpha(x_1,\dots,x_5)$
and $\Phi(x_1,\dots,x_5)([z_1\:|\:z_2,z_3,z_4])=-[z_1\:|\:z_2,z_3,z_4]\cdot \overline{\alpha(x_1,\dots,x_5)},$ respectively. In both cases, all solutions are normal and take form 
$$\mathcal{S}^\infty=\{[z_1\:|\: z_2,0,z_4]\}$$
Let us note that the geometric interpretation of the solutions is described in \cite{CEMN}.

\noindent
{\bf (3)  Adjoint representation.}  
Let us consider the adjoint representation $\frak{sl}(4,\C)$, where $V_i=\fg_i$ for the $|1|$--grading from Section \ref{cproj-alp}. Similarly as in the projective case, there are three different connections:

\noindent
$\bullet$ The tractor connection, i.e., $\Phi=\ad \circ \alpha$ for the adjoint action on $\frak{sl}(4,\C)$.

\noindent
$\bullet$ The prolongation connection, i.e., $\Phi^{prol}=\ad \circ \alpha+\Psi$, where $\Psi(x_1, \dots, x_5)$
$$
{\footnotesize
\left[ \begin {array}{c|ccc} * &z_{{15}}&z_{{14}}&z_{{13}}
\\ \hline z_{{1}}&z_{{4}}&z_{{7}}&z_{{8}}
\\ z_{{2}}&z_{{9}}&z_{{5}}&z_{{10}}
\\ z_{{3}}&z_{{11}}&z_{{12}}&z_{{6}}
\end {array}
 \right] 
 \mapsto 
 \left[ \begin {array}{c|ccc} * &{3 \over 8}  x_3 z_{12}&0&-{3 \over 4} x_3 z_7-{3 \over 8}x_1 z_{12} \\  \hline 0&0&0&0\\ 0&-{3 \over 4}x_2z_2&0&-{3 \over 4}x_1z_3 \\ 0&0&0&0\end {array} \right] 
}.
$$

\noindent
$\bullet$ The automorphism connection for $\Phi^{aut}=\ad \circ \alpha-\iota_\kappa$ providing the infinitesimal automorphisms, where 
$$\iota_\kappa(x_1,x_2,x_3,x_4,x_5):{\footnotesize
\left[ \begin {array}{c|ccc} * &z_{{15}}&z_{{14}}&z_{{13}}
\\ \hline z_{{1}}&z_{{4}}&z_{{7}}&z_{{8}}
\\ z_{{2}}&z_{{9}}&z_{{5}}&z_{{10}}
\\ z_{{3}}&z_{{11}}&z_{{12}}&z_{{6}}
\end {array}
 \right] 
 \mapsto 
 \left[ \begin {array}{c|ccc} 0&0&0&0\\ \hline 0&0&0&0\\ 0
&0&0&\frac32(x_3z_1-x_1z_3)\\ 0&0&0&0
\end {array} \right] 
.
}
$$
We compute that the tractor and prolongation connections have the same (complex) six--dimensional solution set
$$
\mathcal{S}^\infty=\{
{\footnotesize \left[
\begin{array}{c|ccc} * & -{1 \over 4}z_{1}-2z_5 & 0 & z_{13} \\
\hline  z_1 & -{1 \over 2}z_{1}-3z_5 & 0 & 2z_{13}     \\
z_2 & z_9 & z_5 & z_{10} \\
0 & 0 & 0 & z_5
\end {array} \right] } \}
$$
and automorphism connection has the following (complex) seven--dimensional solution set
$$
\mathcal{S}^\infty=\{
{\footnotesize \left[
\begin{array}{c|ccc} 
* & 0 & 0 & 0 \\ \hline  
z_1 & z_4 & 0 & z_{5}     \\
z_2 & z_6 & -\frac53z_4-\frac56z_1 & z_{7} \\
z_{3} & -\frac12z_3 & 0 & -\frac13z_4+\frac13z_1
\end {array} \right] } \}.
$$

\noindent
{\bf (4) Representation on Hermitian matrices.} 
Let us consider the representation of $\frak{sl}(4,\C)$ on the space of $4\times 4$--Hermitian matrices that forms an irreducible subrepresentation of $\C^{4*}\otimes \overline{\C^{4*}}$ consisting of elements $S+{\rm i}A$, where $S$ is symmetric and $A$ is skew--symmetric. The tractor connection coincides with the prolongation connection and thus all solutions are normal. If we view the Hermitian matrices as $(1,3)$--block matrices, the grading of the tractor bundle and the space of solutions are as follows
$$
V \simeq
{\footnotesize \left[
\begin{array}{c|ccc} 
V_{-{3 \over 2}} & V_{-{1 \over 2}} \\ \hline  
* & V_{{1 \over 2}}
\end {array} \right] },\ \ \
\mathcal{S}^\infty=\{
{\footnotesize \left[ \begin {array}{c|ccc} w_{{7}}&{\rm i}w_{{1}}+w_{{8}}&0&{\rm i}w_{{4}}+w_{{
13}}\\ \hline *&w_{{9}}&0&{\rm i}w_{{5}}+w_{{14}}
\\ *&*&0&0\\ 
*&*&*&w_{{16}}\end{array} \right] 
 \} }.
$$

\subsection{Local solutions in $(2,3,5)$ distributions}
Let us present solution spaces for BGG operators on example from Section \ref{235-alp}.

\noindent
{\bf (1)  Standard and dual representations.} \label{235-S}
Let us consider the standard representation on $\R^7$, where we have $\R^7=[V_2 \:|\: V_1\:|\: V_0 \:|\: V_{-1}\:|\: V_{-2}]$ written as $(1,2,1,2,1)$--block vectors.
 There is a natural duality between standard and dual representation. Moreover, general theory states that the prolongation connection always coincides with the tractor connection and solutions are normal, \cite{SW2}. We get
$$\mathcal{S}^\infty=\{[{\scriptstyle {4 \over 9}} w_7\:|\: 0,0\:|\:0\:|\: 0,0\:|\: w_7]\}$$ 
and compute that the infinitesimal holonomy of the tractor connection is the following representation of $\frak{su}(2,1)$ on $\R^7$
$$ \left[ \begin {smallmatrix} 
0&{4 \over 9}  h_7&{4 \over 9}  h_6&-{4 \over 9}  h_8&{4 \over 9}  h_5&{4 \over 9}  h_4&0\\
  h_4& h_1& h_2&-{4 \sqrt {2}\over 9}  h_6&{2 \sqrt {2} \over 9}  h_8&0&-{4 \over 9}  h_4\\ 
 h_5& h_3&- h_1& {4\sqrt {2}  \over 9}   h_7 &0&-{2\sqrt{2} \over 9} h_8&-{4 \over 9}  h_5\\
  h_8&-\sqrt {2}h_5 &\sqrt {2} h_4&0&{4 \sqrt {2} \over 9}  h_7 &-{4 \sqrt {2} \over 9}  h_6&-{4 \over 9} h_8\\ 
 h_6&-{\sqrt {2} \over 2}  h_8&0&\sqrt {2} h_4& h_1&- h_2&-{4 \over 9}  h_6\\  
h_7&0&\sqrt {2} h_8&-\sqrt {2}h_5 &- h_3&- h_1&-{4 \over 9}  h_7\\
 0&- h_7&- h_6& h_8&- h_5&- h_4&0\end {smallmatrix} \right].$$
%
Let us note that solutions correspond to almost Einstein scales of the induced conformal structure, see \cite{SW} for details. In particular, in \cite[Theorem D\_]{SW} the authors discuss the holonomy reduction in question.
\begin{rem}
With respect to the parameter representing the ratio between the scalar curvatures of the ball and hyperbolic space, there is only one other case when there are solutions of the first BGG operator on $\R^7$ which is flat, i.e., $\mathcal{S}^\infty=\R^7$.
\end{rem}
\noindent
{\bf (2)  Symmetric powers.} Let us consider the symmetric and skew--symmetric powers of the standard representation $S:=S^2 \R^7$ and $W:=\wedge^2 \R^7=\fg_2(2)\oplus\R^7$, and represent its elements as symmetric and skew--symmetric $(1,2,1,2,1)$--block matrices, respectively
$$
{\footnotesize
S \simeq 
\left[
\begin{array}{c|c|c|c|c}  
V_{4} & V_{3} & V_{2} & V_1 & V_0\\ \hline
* & V_{2} & V_{1} & V_0 & V_{-1} \\ \hline
* & * & V_{0} & V_{-1} & V_{-2}\\ \hline
* & * & * & V_{-2} & V_{-3} \\ \hline
* & * & * & * & V_{-4}
\end{array}  \right],  \ \ \ 
W \simeq 
\left[
\begin{array}{c|c|c|c|c}  
0& V_{3} & V_{2} & V_1 & V_0 \\ \hline
* & V_{2} & V_1 & V_0& V_{-1} \\ \hline
* & * & 0 & V_{-1} & V_{-2}\\ \hline
* & * & * & V_{-2} & V_{-3} \\ \hline
* & * & * & * & 0
\end{array}  \right].
}
 $$
In the symmetric case, the prolongation connection does not equal to the tractor connection, however, all solutions are normal. Therefore we will not explictly write down the difference of the connections. We get that $\mathcal{S}_S^\infty$ consists of 
$$
{\footnotesize \left[
\begin{array}{c|cc|c|cc|c} w_1&0&0&0&0&0&w_{22}
\\ \hline *&0&0&0&0&-{9 \over 4}w_1+w_{22}&0
\\ *&*&0&0&-{9 \over 4} w_1+w_{22}&0&0
\\\hline *&*&*&{9 \over 4}w_1-w_{22}&0&0&0
\\ \hline *&*&*&*&0&0&0
\\ *&*&*&*&*&0&0
\\ \hline *&*&*&*&*&*&{\frac {81}{16}}w_1
\end{array} \right] },
$$
where the case $w_{1}=0$ defines the conformal metric  sitting in trivial subrepresentation of $S$, \cite{CK}, and $w_{22}=0$ sits in its irreducible complement.

In the skew--symmetric case, prolongation connection does not equal to tractor connection, too. 
Their difference $\Psi(x_i)(w_i)$ is given by the matrix that has everywhere zeros except positions
{\footnotesize 
\begin{align*}
[2,5]=&-{\frac {8\sqrt {2}}{27}}x_{{1}}w_{{10}}-{\frac {8\sqrt {2}}{9}}x_{{2}}w_{{14}}-{\frac {8}{27}}x_{{1}}w_{{17}}+{\frac {8}{9}}x_{{2}}w_{{18}}+{\frac {8}{9}}x_{{4}}w_{{20}}-{8\over 3}x_{{5}}w_{{21}}
\\
[2,6]=&{\frac {8 \sqrt {2}}{27}}x_{{2}}w_{{10}}+{\frac {8\sqrt {2}}{27}}x_{{1}}w_{{14}}+{\frac {8}{27}}x_{{2}}w_{{17}}-{\frac {8}{27}}x_{{1}} w_{{18}}+{\frac {8}{9}}x_{{5}}w_{{20}}-{\frac {8}{9}}x_{{4}}w_{{21
}}\\ 
[3,6]=&{\frac {8 \sqrt {2}}{9}}x_{{1}}w_{{10}}+{\frac {8 \sqrt {2}}{27}}x_{{2}}w_{{14}}+{\frac {8}{9}}x_{{1}}w_{{17}}-{\frac {8}{27}}x_{{2}}w_{{18}}-{8 \over 3}x_{{4}}w_{{20}}+{
\frac {8}{9}}x_{{5}}w_{{21}}
\end{align*}
}
and $[3,5]=-[2,6]$ and corresponding skew--symmetric positions.
All solutions of $\mathcal{S}^\infty_W$ take form
$$
 {\footnotesize \left[
\begin{array}{c|cc|c|cc|c}
*&{2\sqrt {2} \over 9}w_{10}&{2\sqrt {2} \over 9}w_{14}&w_4&{\sqrt {2} \over 2}w_5&-{\sqrt {2} \over 2}w_{6}&0\\ \hline
*&*& w_3&w_{5}&w_{8}&0&{\sqrt {2} \over 2} w_{10}\\ 
*&*&*& w_{6}&0&-w_{8}&-{\sqrt {2} \over 2} w_{14}\\ \hline 
*&*&*&*&w_{10}&w_{14}&{9 \over 4}w_{4}\\ \hline
*&*&*&*&*&-{9 \over 4}w_{3}&{\frac {9\sqrt {2}}{8}} w_5\\ 
*&*&*&*&*&*&-{\frac {9}{8}}\sqrt {2} w_6\\ \hline
*&*&*&*&*&*&*
\end{array} \right] }
$$
and normal solutions are those satisfying $w_3=-\sqrt{2}w_4$ and the rest vanishes. We know from the point (1) that the normal solutions sit in $\R^7$ and it is not hard to observe (using the conformal metric) that the other solutions are just reparametrization of $\alpha(\fk).$ Let us remark that this allows us to see  in hindsight that the difference $\Psi(x_i)(w_i)$ is exactly the inclusion $-\iota_{\kappa}$ into the curvature.

\noindent
{\bf (3)  Metrizability and $S^2(\ad)$.} Let us finally consider an example of a more complicated representation $V=S^2(\ad)=C\oplus  S$, where $C$ is the Cartan component of $S^2(\ad)$.  We already computed the solutions in $S$ in the point (2) and the solutions in $C$ are related to the problem of  submetrizability considered in \cite{metrics}.

The dimension of the corresponding tractor bundle $S^2(\ad)$ is $105$ and its elements can be represented as symmetric $(2,1,2,4,2,1,2)$--block matrices such that on the block position $[i,j]$ sits $V_{i+j-8}$.
This however does not provide a reasonable way to present the solutions. Instead, we analyse the solutions using the standard methods (highest weights and root lattices) from the structure theory of representations of semisimple Lie algebras. Firstly, there is only a two--dimensional family of normal solutions, which we know exhausts all solutions in $S$, and there is a $20$--dimensional family of additional solutions in $C$, which are encoded as representations of $\frak{sl}(2,\R)\oplus \frak{so}(3)$ together with the projections to $V_{-6}=S^2\R^2$ as follows.
\noindent \\ $\bullet$
A representation $S^4\R^2$ (realized as homogeneous polynomials of degree $4$ in variables $p_1,p_2$) of $\frak{sl}(2,\R)$ together with the projection 
$$v_1p_1^4+v_2p_1^3p_2+v_3p_1^2p_2^2+v_4p_1p_2^3+v_5p_2^4\mapsto \left[\begin{smallmatrix}  
6v_1-2v_3+6v_5 & -3v_2+3v_4\\
* &4v_3
\end{smallmatrix}  \right].$$
\noindent \\ $\bullet$ A representation $\frak{sl}(2,\R)\otimes \frak{so}(3)=(H,X,Y)\otimes (A,B,C)$ of $\frak{sl}(2,\R)\oplus \frak{so}(3)$ together with the projection
\begin{gather*}
v_6H\otimes A+v_7H\otimes B+v_8H\otimes C+v_9X\otimes A+v_{10}X\otimes B+v_{11}X\otimes C\\
+v_{12}Y\otimes A+v_{13}Y\otimes B+v_{14}Y\otimes C
 \mapsto \left[\begin{smallmatrix}  
2v_{10}+2v_{13}& 2v_7+v_9+v_{12} \\
* & 4v_6
\end{smallmatrix}  \right].\end{gather*}
\noindent \\ $\bullet$ A representation $S^2\R^3$ of $\frak{so}(3)$ (realized as homogeneous polynomials of degree $2$ in variables $p_1,p_2,p_3$) together with the projection
\begin{gather*}
(v_{15}-v_{18})p_1^2+(v_{15}+v_{18}-v_{19})p_2^2+(v_{15}+v_{19})p_3^2+v_{20}p_1p_2+v_{21}p_1p_3+v_{22}p_2p_3 \\
\mapsto \left[\begin{smallmatrix}  
v_{15}+2v_{18}-2v_{19} & -v_{22} \\
* &v_{15}+2v_{19}
\end{smallmatrix}  \right].\end{gather*}
\noindent \\ $\bullet$ The trivial representations $v_{16},v_{17}$ in $S$ having trivial projection to $V_{-6}$.

 \subsection{Local solutions in CR geometry and theory of systems of PDEs}\label{CR-S}\label{lagr-S}
Let us now turn to solution spaces for BGG operators on tractor bundles for the CR example and LC example from Section \ref{CR-alp}. We use the same methods as in the previous case, however, we do not give that much details because the calculations and results are much longer. Let us emphasize that additional differences appear here due to differences in the representation theory of the two different real forms.

\noindent
{\bf (1)  Standard, dual and conjugate representations.}  The standard representation of $\frak{su}(1,3)$ on $V=\C^4$ is complex and admits an invariant Hermitian form that provides isomorphism $\bar{V}\cong V^*$. We write $\C^{4}= [V_{1} \:|\: V_{0} \:|\: V_{-1}]$ and $\C^{4*}= [V_{-1} \:|\: V_{0} \:|\: V_{1}]$ as $(1,2,1)$--block vectors. Further, we consider the standard representation of $\frak{sl}(4,\R)$ on $\R^4$ and the dual representation on $\R^{4*}$. 
We write $\R^{4}= [V_{1} \:|\: V_{0} \:|\: V_{-1}]$ and  $\R^{4*}= [V_{-1} \:|\: V_{0} \:|\: V_{1}]$ as $(1,2,1)$--block vectors. 

In all cases, the prolongation connections coincide with the tractor connections. We observe from the formulas \eqref{CR-cur} for $\kappa_{CR}$ and \eqref{CR-cur2} for $\kappa_{LC}$ that $\mathcal{S}^0_{\C^4}=\{[0\:|\:0,0\:|\:z]\}$ and $\mathcal{S}^0_{\C^{4*}}=\{[z\:|\:0,0\:|\:0]\}$ for $z\in \mathbb{C}$, and $\mathcal{S}^0_{\R^4}=\{[w\:|\: 0, 0 \:|\:0]\}$  and $\mathcal{S}^0_{\R^{4*}}=\{[0 \:|\: 0, 0 \:|\: w]\}$ for $w\in \mathbb{R}$. Then it follows from the formulas \eqref{CR-rozsireni} for $\alpha_{CR}$ and \eqref{CR-rozsireni2} for $\alpha_{LC}$ that $\mathcal{S}^1=\mathcal{S}^\infty=\{0\}$ in all cases and there are no non--trivial solutions. In particular, the infinitesimal holonomy of the tractor connections is the full $\frak{su}(1,3)$ and $\frak{sl}(4,\R)$, respectively. 
More comments on the interpretation of solutions in the CR case can be found in \cite{Fef}.

\noindent
{\bf (2)  Second powers.}  Let us consider the tractor bundles for the symmetric representations $V_{CR}=S^2\C^4$ and $V_{LC}=S^2\R^4$ and represent elements as complex and real symmetric $(1,2,1)$--block matrices, respectively, as follows
$$
V \simeq\left[
\begin{array}{c|c|c}  V_{2} & V_1 &V_0\\ \hline * & V_{0} & V_{-1} \\ \hline  * & * & V_{-2} \end{array}  \right].
$$
It turns out that in both cases, the prolongation connection does not coincide with the tractor connection. In fact, there are no normal solutions while the spaces of all solutions take form
$$S^\infty_{CR}=\{ {\footnotesize \left[
\begin{array}{c|cc|c}
z_1 & z_2&z_4 &\frac {24}{11} z_1 \\ \hline 
*&0&z_5& 8z_2\\ 
*&*&0&-8z_4\\ \hline
*&*&*& -{\frac {576}{11} z_1}\end {array}
 \right] } \},
 \ \ \ \ 
\mathcal{S}^\infty_{LC}=\{ {\footnotesize \left[
\begin{array}{c|cc|c} w_{1}&{w_2}&{w_4}&-{\frac {12}{
11}}{w_1}\\ \hline *&0&{w_5}&-4{w_2}
\\ *&*&0&4{w_4}
\\ \hline *&*&*&-{\frac {144}{11}}{ w_1}\end {array} \right] } \}.
$$
Finally, let us note that there are no non--trivial solutions in the cases $V_{CR}=\wedge^2\C^4$ and $V_{LC}=\wedge^2\R^4$, respectively.

\noindent
{\bf (3)  Adjoint representation.}  
Let us now briefly mention the adjoint tractor bundles $V=\fg$, where $V_i=\fg_i$ for the contact gradings from Section \ref{CR-alp}.
The prolongation connections do not coincide with the tractor connections in both cases, but they coincide with the automorphism connections. There are no normal solutions. Spaces of all solutions recover (up to parametrization) the images $\mathcal{S}^\infty_{CR}=\alpha_{CR}(\fk_{CR})$ and $\mathcal{S}^\infty_{LC}=\alpha_{LC}(\fk_{LC})$, i.e., there are no additional infinitesimal automorphisms.

 \noindent
{\bf (4) Metrizability, $\otimes^2 \wedge^2 \C^4$ and $\otimes^2 \wedge^2 \R^4$.} 
Let us finally consider examples for more complicated representations, namely $V_{CR}=\otimes^2 \wedge^2 \C^4=S^2 \wedge^2 \C^4\oplus  \C^4\otimes \C^{4*}$ and $V_{LC}=\otimes^2 \wedge^2 \R^4=S^2 \wedge^2 \R^4\oplus  \R^4\otimes \R^{4*}$, which are different from the viewpoint of representation theory. For $V_{CR}$, both summands are complexifications of real representations of $\frak{su}(1,3)$, while for $V_{LC}$, both summands are real representations of $\frak{sl}(4,\R)$. In both cases, the second summand decomposes to (the complexification of) the adjoint representation and trivial representation whose solutions we know, and the first summand is (the complexification of) the space that is related to the problem of submetrizability considered in \cite{metrics}.

Complex dimension of $V_{CR}=\otimes^2 \wedge^2 \C^4$ and real dimension of $V=\otimes^2 \wedge^2 \R^4$ both equal to $36$ and we consider complex coordinates $z_i$, $i=1\dots 36$ such that 
\begin{align*}
V_{-2}&= \langle z_{29}, z_{30}, z_{35}, z_{36} \rangle,\\
V_{-1}&=\langle z_{17}, z_{18}, z_{23}, z_{24}, z_{27}, z_{28}, z_{33}, z_{34} \rangle,\\
V_0&= \langle z_5, z_6, z_{11},z_{12}, z_{15}, z_{16}, z_{21},z_{22}, z_{25}, z_{26}, z_{31}, z_{32} \rangle,\\
V_1&= \langle z_3, z_4, z_9, z_{10}, z_{13}, z_{14}, z_{19}, z_{20}\rangle,\\
V_{2}&= \langle z_{1}, z_{2}, z_{7}, z_{8} \rangle.
\end{align*}
in the CR case and real coordinates $w_i$, $i=1\dots 36$ with the same components in the LC case.

In both cases, the prolongation connection does not coincide with the tractor connection. In the CR case we get the following complex nine--parameter family of solutions 
\begin{gather*}
\mathcal{S}^{\infty}_{CR}=[0,z_2,-z_{13},z_4,0,z_6,z_7,0,z_9,z_{10},z_{11},0,z_{13},-z_9,0,{\scriptstyle{216\over 31}}z_2-z_{21}+ z_6-\\ 
z_{11}+{\scriptstyle{216 \over 31}}z_7,-24z_{13},-{\scriptstyle{24 \over 5}}z_9,-z_4,-z_{10},z_{21},0,{\scriptstyle{24 \over 5}}z_4,-24z_{10},0,-z_6,24z_{13},\\
-{\scriptstyle{24 \over 5}}z_4,0,
-{\scriptstyle{20160 \over 403}}z_2-{\scriptstyle{2304 \over 403}}z_7,-z_{11},0,{\scriptstyle{24 \over 5}}z_9,24z_{10},-{\scriptstyle{20160 \over 403}}z_2-{\scriptstyle{2304 \over 403}}z_7,0],
\end{gather*}
where the normal ones satisfy that $z_{6}=-z_{11}=z_{21}$ and the rest vanishes. In the LC case, we get the following real nine--parameter family of solutions 
\begin{gather*}
\mathcal{S}^{\infty}_{LC}=[
w_1, 0, w_3, w_4, w_5, 0, 0, w_8,w_9,w_{10},0,
w_{22}-w_5+w_{15}-{\scriptstyle {108 \over 31}}w_1+{\scriptstyle{108\over 31}}w_{8}, \\
w_{10}, -w_4, w_{15}, 0, {\scriptstyle{12 \over 5}}w_{10}, 12w_4, w_{9}, -w_3,
0, w_{22}, 12w_9, {\scriptstyle{12 \over 5}}w_3, w_5, 0,  {\scriptstyle{12 \over 5}}w_3,\\
12w_4, {\scriptstyle{5040 \over 403}}w_1-{\scriptstyle{576 \over 403}}w_{8}, 
 0, 0, w_{22}-w_5+w_{15}-{\scriptstyle {108 \over 31}}w_1+{\scriptstyle{108\over 31}}w_{8}, 
 -12w_9,\\
 -{\scriptstyle{12 \over 5}}w_{10}, 0, {\scriptstyle{5040 \over 403}}w_1-{\scriptstyle{576 \over 403}}w_{8}
\end{gather*}
where the normal ones satisfy $w_5=w_{15}=w_{22}$ and the rest vanishes.

Thus (in the real form) we have dimension $9$ of solution space and we know that a subspace of dimension $7$ sits in the adjoint representation and the normal solutions are clearly in trivial representation in $\C^4\otimes \C^{4*}$ and $\R^4\otimes \R^{4*}$, respectively. Consequently, the additional dimension of solutions provides the submetrics and it is not hard to observe that it is a trivial representation of $\fk_{CR}$ and $\fk_{LC}$, respectively. So we can directly project the solutions to $V_{-2}$ and decompose both 
\begin{gather*}
[0,-{\scriptstyle{20160 \over 403}}z_2-{\scriptstyle{2304 \over 403}}z_7,-{\scriptstyle{20160 \over 403}}z_2-{\scriptstyle{2304 \over 403}}z_7,0],\\ [{\scriptstyle{5040 \over 403}}w_1-{\scriptstyle{576 \over 403}}w_{8}, 0,  0, {\scriptstyle -{576 \over 403}}w_1+{\scriptstyle {5040\over 403}}w_{8}]
\end{gather*} 
into $\fg_{-2}$ and the space of submetrics. In the CR case, the submetric should be Hermitian and from action of $\fg_0$ is clear that solutions in $\fg_{-2}$ are of the form $[0,a,a,0]$ and submetrics are of the form $[0,-b{\rm i} ,b{\rm i} ,0]$ for $a,b\in \R$. In the Lagrangean case, we get that the solutions in $\fg_{-2}$ are of the form $[a,0,0,a]$ and submetrics are of the form $[b,0,0,-b]$ for $a,b\in \R$ for the same reasons.

\subsection{Local solutions in the theory of systems of ODEs}\label{path-S}
Let us present solution spaces for BGG operators on example from Section \ref{path-alp}.

\noindent
{\bf (1)  Standard and dual representations.} 
Let us consider the standard representation $\rho$ of $\frak{sl}(4,\R)$ on $V=\R^4$ and the dual representation $\rho^*$ on $V=\R^{4*}$. We get $\R^4=[V_{{5 \over 4}}\: | \: V_{{1 \over 4}} \: | \: V_{-{3 \over 4}}]$ and 
$\R^{4*}=[V_{-{5 \over 4}}\: | \: V_{-{1 \over 4}} \: | \: V_{{3 \over 4}}]$ written as $(1,1,2)$--block vectors. Prolongation connections coincide with tractor connections in both cases, and we have
\begin{gather*}
\Phi(x_1,\dots,x_6)([w_1 \:|\: w_2 \:|\:w_3,w_4])=\alpha(x_1,\dots,x_6)\cdot [w_1\:|\:w_2 \:|\:w_3,w_4]^t\\
=[0 \:|\: x_3w_2 \:|\: x_1w_1+x_4w_2, x_2w_1+x_5w_2+x_6w_3 ],\\
\Phi^*(x_1,\dots,x_6)([w_1, w_2,w_3,w_4])=-[w_1\:|\:w_2 \:|\:w_3,w_4]\cdot \alpha(x_1,\dots,x_6) \\
=[-x_1w_3-x_2w_4-x_3w_2, -x_4w_3-x_5w_4, -x_6w_4, 0].
\end{gather*}
Solutions are normal in both cases and we get
$$\mathcal{S}^\infty_{\R^4}=\{[0\:|\:0\:|\:0,w_4]\}, \ \ \ \ \mathcal{S}^\infty_{\R^{4*}}=\{[w_1\:|\:w_2\:|\:w_3,0]\}.$$
The infinitesimal holonomy of the tractor connection is a three--dimensional subalgebra of $\frak{sl}(4,\R)$ consisting of elements
$$\left[ \begin{smallmatrix} 0&0&0&0\\  
0&0&0&0\\ 
0&0&0&0\\ 
h_3&h_2&h_1&0 \end{smallmatrix} \right].$$

\noindent
{\bf (2)  Second powers.}  Let us firstly consider the tractor bundle for symmetric and skew--symmetric representations $V=S^2\R^4$ and $V=\wedge^2\R^{4}$ and represent their elements as symmetric $(1,1,2)$--block matrices and skew--symmetric $(1,1,2)$--block matrices, respectively, as follows 
$$
S^2\R^4\simeq 
\left[
\begin{array}{c|c|c}  V_{{5 \over 2}} & V_{{3 \over 2}} & V_{{1 \over 2}}\\ \hline * & V_{{1 \over 2}} & V_{-{1 \over 2}} \\ \hline  * & * & V_{-{3 \over 2}} \end{array}  \right], \ \ \ \ 
\wedge^2\R^4 \simeq 
\left[
\begin{array}{c|c|c}  0 & V_{{3 \over 2}} & V_{{1 \over 2}}\\ \hline * &0 & V_{-{1 \over 2}} \\ \hline  * & * & V_{-{3 \over 2}} \end{array}  \right].
$$
In the skew--symmetric case, the prolongation connection coincides with the tractor connection.
In the symmetric case, they do not coincide and their difference $\Psi(x_1, \dots, x_6)$ takes the form
$$
{\footnotesize
\left[ \begin {array}{c|c|cc} 
w_1&w_{{2}}&w_{{4}}&w_{{7}}\\ 
\hline
*&w_3&w_{{5}}&w_{{8}}
\\ \hline *&*&w_6&w_{{9}}
\\ *&*&*&w_{10}\end {array}
 \right] \mapsto 
 \left[ \begin {array}{c|c|cc} 0&0&0&{1 \over 4}x_4w_6\\ \hline *&0&0&-{1 \over 4}x_1w_6\\ \hline *&*&0&0\\ *&*&*&0\end {array} \right] 
}.
$$
However, all solutions are normal in both cases and take form
$$
S^\infty_{S^2\R^4}=\{  {\footnotesize \left[
\begin{array}{c|c|cc} 0&0&0&0\\ \hline *&0&0&0
\\ \hline *&*&0&0\\ *&*&*& w_{10}
\end {array}  \right] }
 \},\ \ \ 
 S^\infty_{\wedge^2\R^{4}}=\{  {\footnotesize \left[
\begin{array}{c|c|cc} *&0&0&w_{4}\\ \hline *&*&0&
w_{5}\\ \hline *&*&*&w_{6}\\ 
*&*&*&*
\end {array} \right] }
 \}.
$$

Let us now swap to duals and consider symmetric and skew--symmetric representations $V=S^2\R^{4*}$ and $V=\wedge^2\R^{4*}$. We represent their elements as symmetric $(1,1,2)$--block matrices and skew--symmetric $(1,1,2)$--block matrices, respectively, as follows 
$$
S^2\R^{4*}\simeq 
\left[
\begin{array}{c|c|c}  V_{-{5 \over 2}} & V_{-{3 \over 2}} & V_{-{1 \over 2}}\\ \hline * & V_{-{1 \over 2}} & V_{{1 \over 2}} \\ \hline  * & * & V_{{3 \over 2}} \end{array}  \right], \ \ \ \ 
\wedge^2\R^4 \simeq 
\left[
\begin{array}{c|c|c}  0 & V_{-{3 \over 2}} & V_{-{1 \over 2}}\\ \hline * &0 & V_{{1 \over 2}} \\ \hline  * & * & V_{{3 \over 2}} \end{array}  \right].
$$
In the skew--symmetric case, the prolongation connection again coincides with the tractor connection.
In the symmetric case, they do not coincide and their difference $\Psi(x_1, \dots, x_6)$ takes the form
$$
{\footnotesize
\left[ \begin {array}{c|c|cc} 
w_1&w_{{2}}&w_{{4}}&w_{{7}}\\ 
\hline
*&w_3&w_{{5}}&w_{{8}}
\\ \hline *&*&w_6&w_{{9}}
\\ *&*&*&w_{10}\end {array}
 \right] \mapsto 
 \left[ \begin {array}{c|c|cc} 0&0&0&0\\ \hline *&0&0&0\\ \hline *&*&w_7x_4-w_8x_1&0\\ *&*&*&0\end {array} \right].
}
$$
All solutions are as follows, where we write normal solutions in bold font
$$
S^\infty_{S^2\R^{4*}}=\{  {\footnotesize \left[
\begin{array}{c|c|cc} {\bf w_1}&{\bf w_{{2}}}&{\bf w_{{4}}}& w_{{7}}\\ 
\hline
*&{\bf w_3}& {\bf w_{{5}}}& w_{{8}}
\\ \hline *&*&{\bf w_6}&0
\\ *&*&*&0
\end {array}  \right] }
 \}
,\ \ \ 
 S^\infty_{\wedge^2\R^{4*}}=\{  {\footnotesize \left[
\begin{array}{c|c|cc} *&{\bf w_1}&{\bf w_2}&0\\ \hline *&*&{\bf w_3}&
0\\ \hline *&*&*&0\\ 
*&*&*&*
\end {array} \right] }
 \}.
$$

\noindent
{\bf (3)  Adjoint representation.}  
Let us consider the adjoint representation  sitting as a trace--free component  in $\fg+\R\id \simeq V = \R^4 \otimes \R^{4*}$, where $V_i=\fg_i$ for the $|2|$--grading from Section \ref{path-alp}. 
The prolongation connection does not coincide with the tractor connection and their difference $\Psi(x_1, \dots, x_6)$ takes the form
$$
{\footnotesize
\left[ \begin {array}{c|c|cc} w_{{1}}&w_{{5}}&w_{{9}}&w_{{13}}
\\ \hline w_{{2}}&w_{{6}}&w_{{10}}&w_{{14}}
\\ \hline w_{{3}}&w_{{7}}&w_{{11}}&w_{{15}}
\\ w_{{4}}&w_{{8}}&w_{{12}}&w_{{16}}
\end {array}
 \right] 
 \mapsto 
 \left[ \begin {array}{c|c|cc} 0&0&0&0\\ \hline 0&0&0&0\\ \hline 0&0&0&0\\ 0&0&x_1w_7-x_4w_3&0\end {array} \right] 
}
$$
which we can observe coincides with insertion into the curvature $-\iota_\kappa$.
We get that all solutions take form
$$\mathcal{S}^\infty=\{
{\footnotesize \left[
\begin{array}{c|c|cc} 3w_{11}-w_{16}-w_{6}&w_{5}&0
&0\\ \hline w_{2}&w_{6}&0&0\\ \hline w_{3}&w_{7}&w_{11}&0\\  w_{4}&w_{8}&w_{12}&w_{16}\end {array} \right] } \}
$$
and normal are those with $w_6=w_{11}=w_{16}$ and $w_2=w_3=w_5=w_7=0.$ The solutions with $w_6=w_{11}=w_{16}$ and rest vanishing correspond to the trace and if we remove it, we observe that we recovered the infinitesimal automorphisms that we excluded from $\fk$.

\section{Coordinate description of solutions of BGG operators}\label{solcoordcap}

\subsection{Exponential coordinates of other kind} \label{exponential-coord}

From the discussion in Sections \ref{paralel} and \ref{sec-prol} it follows that global solutions of first BGG operators $D$ are given by elements of $\mathcal{S}^\infty_K$ for the prolongation connection given by $$\Phi=\rho\circ \alpha+\Psi.$$ 
There are other kinds of exponential coordinates than those used in Proposition \ref{prop1.1} in formula \eqref{parsec} available. These coordinates take into account the structure of the Lie algebra $\fk$ and the Lie group $K$ and allow us to get a global covering. Moreover, these suggest a particular choice for $\fc$ such that $\Ad^{-1}_k \fc$ is complementary to $\fh$ for coordinates centered at $kH\in K/H$.

For our purposes, it is sufficient to restrict our considerations to the situation that there is a decomposition
$$K=S\exp(\fr)\exp(\fn),$$
 where $S$ is semisimple, $\fr$ consists of elements of the radical of $K$ acting reductively in the representation, and $\exp(\fn)$ consists of elements of the radical of $K$ acting nilpotently in the representation. Thus $\exp(\fr)\exp(\fn)$ is the Jordan decomposition of the radical of $K$. Without loss of generality, we can choose $\fc$ that decomposes into $\fs\oplus \fr \oplus \fn$ in such a way that $$\exp(\Ad_{k}^{-1}X)=\exp(X_\fs)\exp(X_\fr)\exp(X_\fn)$$ for $X \in \fc$. In such coordinates, we get the following expression for the solutions
\begin{align*}
&s(k\exp(X_\fs)\exp(X_\fr)\exp(X_\fn)H)=
\\ 
&\pi_0(\exp(-\Phi(X_\fn))\exp(-\Phi(X_\fr))\exp(-\Phi(X_\fs))s(k)).
\end{align*}
This provides the following simplifications in the computation of the exponentials.
\begin{enumerate}
\item The map $\Phi$ restricted to $\fr$ acts simultaneously diagonalizable over $\C$ on the space of solutions. Therefore, we find corresponding eigenspaces and the transition matrix $T$ to the basis, where $\exp(-\Phi(X_\fr))$ acts on each eigenspace by the exponential of the corresponding eigenvalue.
\item It holds $\exp(-\Phi(X_\fn))=\sum_{j=0}^n(-1)^j\Phi(X_\fn)^j$ for some finite $n$.
\end{enumerate}

The semisimple part $S$ needs further decomposition for the simplification of the exponential. On the other hand, $\Phi$ is a representation of $\fs$ which can be analyzed in terms of highest weights, which provides a basis of the space of solutions that allows to use an explicit realization of the representation of the Lie group $K$ (although not in exponential coordinates). To write this action in exponential coordinates we consider the Iwasawa decomposition $S=\textsf{ KAN}$, \cite[Section 2.3.5.]{parabook}. We recall that $\textsf{K}$ is the maximal compact subgroup of $S$, $\textsf{A}$ is maximal diagonalizable subgroup of $S$ and $\textsf{N}$ is nilpotent subgroup of $S$. The exponential coordinates in $\textsf{A}$ and $\textsf{N}$ can be considered together with the possible part $\exp(\fr)\exp(\fn)$ as above. So it remains to choose coordinates of $\textsf{K}$ where one usually decomposes $\textsf{K}$ into building blocks like rotations. This allows to express all the elements of $\textsf{K}$ as their composition. Altogether, we obtain the following statements.
\begin{prop} \label{modify-exp}
Let $K$ be the Lie algebra with the above decomposition $K=\textsf{KAN}\exp(\fr)\exp(\fn)$. Then there is a complement $\fc$ of $\fh$ in $\fk$ such that each $X\in \fc$ can be written as the sum $X_{\textsf{k}_1}+\dots +X_{\textsf{k}_j}+X_{\textsf{a}\oplus \fr}+X_{\textsf{n}\oplus \fn}$ of elements of $\fc$ that are 
parts of decomposition. Moreover, these exponential coordinates cover the whole $K/H$ and 
\begin{align*}
&s(\exp(X_{\textsf{k}_1})\dots \exp(X_{\textsf{k}_j})\exp(X_{\textsf{a}\oplus \fr})\exp(X_{\textsf{n}\oplus \fn})H)=
\\ 
&\pi_0(\exp(-\Phi(X_{\textsf{n}\oplus \fn}))\exp(-\Phi(X_{\textsf{a}\oplus \fr}))\exp(-\Phi(X_{\textsf{ k}_j}))\dots\exp(-\Phi(X_{\textsf{ k}_1}))s(e)).
\end{align*}
\end{prop}
\begin{proof}
The existence of the decomposition $K=\textsf{KAN}\exp(\fr)\exp(\fn)$ clearly implies that such  complement exists. The existence of the complete covering of $\textsf{AN}\exp(\fr)\exp(\fn)$ is clear by definition.
Since the Weyl subgroup of $\textsf{K}$ is contained in the image of the exponential map, the combination of single exponential map with finite number of elements of the Weyl subgroup allows to cover $\textsf{K}$ by compactness.
\end{proof}
Let us emphasize that the exponential coordinates from Proposition \ref{modify-exp} are a priory injective only on the part $\textsf{N}$ and does not have to be injective on any other part which can pose an obstruction for global existence of the solution if $\exp(-\Phi(\ ))$ has different values for different representatives of the same point. We discuss this phenomena on explicit examples in Section \ref{exam-coord}.

\subsection{Normal solutions in normal coordinates and curved orbit decomposition}
\label{curved-decomp}
In the space of solutions of standard operator $D^{st}$, the normal solutions form a particular $\fk$ --subrepresentation characterized by the property that $\Psi$ annihilates these solutions.  This means that the map $\alpha$ intertwines the $\fk$--representation on the space of normal solutions with the $\fg$--representation on $V$. 

Since $\fg_-\subset \fg$ is a nilpotent complement of $\fp$ in $\fg$, it provides nice coordinates on $G/P$. Replacing the exponential map on $G$ by the flow 
of constant vector fields $\omi(X)$ for $X\in \fg_-$ we obtain a comparison of the model space $G/P$ with the particular parabolic geometry (regardless of whether the geometry is homogeneous or not). In particular, these flows define the normal coordinates for parabolic geometries. 

\begin{lem}[\cite{CGH1}]
In normal coordinates, normal solutions take form
$$\nu(p(Fl_1^{\omi(X)}(k)))=\exp(-X)\nu(k)=\sum_{j=0}^n(-1)^j\rho(X)^j\nu(k)$$
for some neighborhood of $0$ in $\fg_-$, where $p$ is the projection from the Cartan bundle on the manifold. 
In particular, normal solutions are polynomial in normal coordinates.
\end{lem}
In fact, the polynomial expressions for the normal solutions have the same form as on the flat model $(G\to G/P,\om_G)$. This means that in normal coordinates, it suffices to test whether linear combinations of these canonical polynomial forms are solutions of $D^{st}.$

Solutions have polynomial singularities that correspond to different $P$--orbits 
of $\nu$ in the $G$--type $\mathcal{O}$ of the corresponding holonomic reduction.
Via the comparison in normal coordinates, the manifold $K/H$ decomposes into initial submanifolds $M_\beta$ according to the $P$--type, where $\beta$ is representative of the coset $P\backslash \mathcal{O}.$ We denote $G_\beta$ the stabilizer of $\alpha\in \mathcal{O}$.
The normal solution $\nu$ provides a unique Cartan connection $\om_\beta$ of type $(G_\beta,G_\beta\cap P)$ on $\nu^{-1}(\beta)\cap p^{-1}(M_\beta)\to M_\beta$ such that $j^*\om_\alpha=\om_\beta$, where $j$ is the inclusion of $M_\beta$ into $M$. Note that only automorphisms preserving the $P$--type restrict to automorphisms of these Cartan geometries, i.e., in general, this Cartan geometry does not have to be homogeneous anymore.


In applications, the main problem is to find suitable normal coordinates. This is not simple problem even on homogeneous parabolic geometries. Indeed, if $\fc$ is a complement of $\fh$ in $\fk$, then each element $X$ of $\fg_-$ can be uniquely written as $X=\alpha(X_\fc)+X_\fp$ for $X_\fc\in \fc$ and $X_\fp\in \fp$, and the component of $X_\fp$ outside of $\alpha(\fh)$ is obstruction for the flow to be expressed by exponential map of elements of $\fk$, i.e., finding the normal coordinates requires solving of system of ODEs.

\section{Coordinate description of Solutions of BGG operators on examples} \label{exam-coord}
We use the methods from Section \ref{solcoordcap} to give coordinate description of local solutions we obtained in Section \ref{bggexam}. 
In the case of the projective geometry from Example \ref{solproj}, we also present the formulas for the actual BGG operators and show how these look like in several distinguished coordinates. In particular, we realize the normal coordinates and observe the polynomiality of solutions as mentioned Section \ref{curved-decomp}. In the case of the dual standard bundle, we discuss the corresponding holonomy reductions and their interpretation in detail.

In the other examples, we consider the exponential coordinates introduces in Section  \ref{exponential-coord}. In the Example \ref{exp-cproj}, we moreover discuss holonomy reduction for a specific normal solution in a representation of Hermitian type (real representation of complex Lie algebra). In the  Example \ref{exp-234}, we describe all sub--Riemannian metrics in detail. In Examples
\ref{exp-crlagr} and \ref{exp-path}, we discuss possible choices of transitive groups and  dependence of global existence of the solutions on such choice.

\subsection{First BGG operators and their solutions in different coordinates in projective geometry} \label{proj-bgg}

For the Example presented in Sections \ref{proj-alp} and \ref{solproj}, we compare our method of finding solutions of first BGG operators to more direct methods and describe solutions in three different coordinates which are natural for the realization of the computation. 

In the first place, we have the coordinates $(y_1,y_2,y_3)$ in which we introduced our example in Section \ref{proj-alp}.
 We also have the expontial coordinates $(x_1,x_2,x_3)$ introduced in Proposition \ref{proj-CC} that are obtained by the flow (for time $1$) of the right--invariant vector field $x_1\partial_{y_1}+x_2( \partial_{y_2}+y_1 \partial_{y_3})+x_3 \partial_{y_3}$ which form the complement $\fc$. The transition between these two coordinate systems is provided by the diffeomorphism $y_1=x_1, y_2=x_2, y_3=\frac12x_1x_2+x_3.$ 
Denoting by $\theta_1,\theta_2,\theta_3$ the pullbacks of the Maurer--Cartan form on $K$ in these coordinates, the expression for $\s^*\omega$ and $\s^*\kappa$ in this coframe does not change, which highlights that fact that the extension $(\alpha,i)$ does not depend on the choice of the coordinates, see Section \ref{proj-alp}.

We also consider the normal coordinates $(n_1,n_2,n_3)$ with transition map
\begin{gather*}
y_1 = -\frac12\ln(1-n_1)+\frac12\ln(n_1+1), \ \ \ \ 
y_2 = -\frac{n_2}{n_1-1},\\ 
y_3 = -\frac{n_2\ln(n_1+1)+n_2\ln(1-n_1)}{4n_1}+\frac{n_2-4n_3}{4n_1(n_1+1)}-\frac{3n_2}{4n_1(n_1-1)}+\frac{n_3-n_2}{n_1}.
\end{gather*}
given by the projection of the flow of 
$$\omega^{-1}(\left[\begin{smallmatrix}
0& 0& 0 & 0\\
n_1& 0& 0 & 0\\
n_2&0& 0 & 0\\
n_3& 0& 0& 0
\end{smallmatrix}\right])$$ 
starting at $\lz e,e \pz\in K\times P.$ 
Since the flow is not contained in $K$, we compute also the projection of the flow to $P$ as
\begin{gather*}
p(n_1,n_2,n_3)= \left[ \begin {smallmatrix} \sqrt {1- n_1}\sqrt { n_1+1}&-{
\frac { n_1}{\sqrt {1- n_1}\sqrt { n_1+1}}}&0&0
\\ 
0&{\frac {1}{\sqrt {1- n_1}\sqrt { n_1+1
}}}&0&0\\
0&{\frac { n_2\sqrt { n_1+1}}{
 (1 - n_1 ) ^{{3\over 2}}}}&{\frac {\sqrt { n_1+1}}{
\sqrt {1- n_1}}}&0\\ 
0&{\frac {{ n_1}^{2} n_2-{ n_1}^{2} n_3+2 n_1 n_3- n_3}{
 ( 1- n_1 ) ^{{3\over 2}} (  n_1+1 ) ^{{3\over 2}}}}&
{\frac { n_1}{\sqrt {1- n_1}\sqrt { n_1+1}}}&{\frac {
\sqrt {1- n_1}}{\sqrt { n_1+1}}}\end {smallmatrix} \right].
\end{gather*}
This provides a different trivialization $K\times P$ of the Cartan bundle and corresponding pulbacks $\tilde \s^*\omega:=\Ad_{p(n_1,n_2,n_3)}^{-1}\s^*\omega+\delta p(n_1,n_2,n_3)$ and $\tilde \s^*\kappa=\Ad_{p(n_1,n_2,n_3)}^{-1}\s^*\kappa$  take form 
\[
\tilde \s^*\omega= \left[ \begin {smallmatrix}
0&0&0&0
\\ 
dn_1&0&0&0\\
dn_2&0&0&0\\ 
q_1(n_2dn_1-n_1dn_2)+dn_3
&
 q_2(n_2dn_1-n_1dn_2)
&0&0\end{smallmatrix} \right],
\]
\[
\tilde \s^*\kappa((x_1,x_2,x_3),(y_1,y_2,y_3))= \left[ \begin {smallmatrix}
0&0&0&0
\\ 
0&0&0&0\\
0&0&0&0\\ 
0&
 \frac{2(x_1y_2-x_2y_1)}{4(1-n_1)^3(n_1+1)}
&0&0\end{smallmatrix} \right],
\]
where $\delta p(n_1,n_2,n_3)$ is the left logarithmic derivative and $$q_1=\frac{\ln(\frac{n_1+1}{1-n_1})(n_1^2-1)+2n_1 }{4(n_1-1)n_1^2}, \ \ q_2= \frac{\ln(\frac{n_1+1}{1-n_1})(n_1-1)^2-2n_1}{4(n_1-1)^2n_1^2}.$$

In the next, we use the trivializations $\s^*s$ and $\tilde \s^*s$ of sections of tractor bundles from Section \ref{solproj}, i.e., we view the tractors as matrices with $w_i$--entries that become functions in one of the above coordinates. Then $D^\fk$ (part of the tractor connection) is a matrix of one--forms $D^\fk w_i=\sum_j L_{X_j}(w_i)\theta_j$, where $L$ is the Lie derivative in directions of the left--invariant vector fields for the basis $X_i\in \fk$ defined in Section \ref{proj-alp}. Further, the coefficients of the polynomials $Q$ in the construction of the splitting operators and the projections $\pi_0,\pi_1$ to $\mathcal{H}^0(V), \mathcal{H}^1(V)$ are well known,  cf. \cite{Hammerl}, and we do not review them in detail.

\noindent
{\bf (1) Standard representation.}
In the case of standard representation, we get
\begin{gather*}
\mathcal{D}^{st}\circ \pi_0([0 \:|\:w_2,w_3,w_4])=\pi_1\circ \nabla^{\rho\circ \alpha}\circ (\id-{ \frac13}\partial^*\nabla^{\rho\circ \alpha})( [0 \:|\:w_2,w_3,w_4])=\\
\left[ \begin{smallmatrix} \sum_j L_{X_j}(w_2)\theta_j\\ \sum_j L_{X_j}(w_3)\theta_j\\ \sum_j L_{X_j}(w_4)\theta_j \end{smallmatrix} \right]+
 \left[ \begin{smallmatrix}
\theta_1&0&0&0\\ 
\theta_2&- \theta_2& - \theta_1&0\\ 
\theta_3&\theta_3& -\theta_1& \theta_1 \end{smallmatrix} \right].\left[ \begin{smallmatrix}-\frac13 \sum_j L_{X_j}(w_{j+1})\\ w_2\\w_3\\w_4 \end{smallmatrix} \right]
\end{gather*}
for the pullback by $\s$ and $\mathcal{D}^{st}\circ \pi_0([0 \:|\:w_2,w_3,w_4])$ equals to
\begin{gather*}
\pi_1\circ \nabla^{\rho\circ \alpha}\circ (\id-{\frac13}\partial^*\nabla^{\rho\circ \alpha})( [0 \:|\:w_2,w_3,w_4]) =  \left[ \begin{smallmatrix} \sum_j \partial_{n_j}w_2dn_j\\ \sum_j \partial_{n_j}w_3dn_j\\ \sum_j \partial_{n_j}w_4dn_j \end{smallmatrix} \right]
+\\
 \left[ \begin {smallmatrix}
dn_1&0&0&0\\
dn_2&0&0&0\\ 
q_1(n_2dn_1-n_1dn_2)+dn_3
& q_2(n_2dn_1-n_1dn_2)
&0&0\end{smallmatrix} \right]. 
\left[ \begin{smallmatrix} -\frac13 (\sum_j \partial_{n_j}w_{j+1}+ q_2 (n_2\partial_{n_3}w_{2}-n_1\partial_{n_3}w_{3})) \\w_2\\w_3\\w_4 \end{smallmatrix} \right]
\end{gather*}
for the pullback by $\tilde \s$. These provide the following eight PDEs (the ninth one is linearly dependent on the others) for the pullback by $\s$
\begin{gather*}
L_{X_1}(w_2)- {\scriptstyle \frac13 \sum_j} L_{X_j}(w_{j+1})=0, \ \ 
L_{X_1}(w_3)-w_3=0, \ \ 
L_{X_1}(w_4)+w_4-w_3=0\\
L_{X_2}(w_2)=0,\ \ 
L_{X_2}(w_3)-w_2- {\scriptstyle \frac13 \sum_j} L_{X_j}(w_{j+1})=0,\ \ 
L_{X_2}(w_4)=0\\
L_{X_3}(w_2)=0,\ \ 
L_{X_3}(w_3)=0.
\end{gather*}
This system has two--parameter family of solutions in $(y_1,y_2,y_3)$--coordinates $$w_2(y_1, y_2, y_3) = 0, w_3(y_1, y_2, y_3) = C_1e^{y_1}, w_4(y_1, y_2, y_3)= \frac12C_1e^{y_1}+C_2e^{-y_1}.$$

We do not write down analogous but much more complicated PDEs for the pullback by $\tilde \s$ because we know that solutions in normal coordinates are first order polynomials in variables $n_i$ and it is not hard to see from the formula for the BGG operator that in normal coordinates, the two--parameter family of solutions takes form
$$w_2(n_1,n_2,n_3)=0,w_3(n_1,n_2,n_3)=c_1,w_4(n_1,n_2,n_3)=c_2.$$

Finally, let us return to the exponential coordinates. We observe that $\rho\circ \alpha(x_1,x_2,x_3)$ acts on $[w_3,w_4]^t \in \mathcal{S}^{\infty}$ by the matrix $[ \begin{smallmatrix} 
 -x_1&0\\ 
 -x_1& x_1 \end{smallmatrix}]$. 
Thus 
$$\exp(-\rho\circ \alpha(x_1,x_2,x_3))= [ \begin{smallmatrix} 
e^{x_1}&0\\ 
\sinh(x_1)& e^{-x_1} \end{smallmatrix}]$$ 
and the solutions in the exponential coordinates $(x_1,x_2,x_3)$ take form
\begin{gather*}
w_2(x_1,x_2,x_3)=0,w_3(x_1,x_2,x_3)=w_3e^{x_1},w_4(x_1,x_2,x_3)=w_3\sinh(x_1)+w_4 e^{-x_1},
\end{gather*}
which are globally defined, i.e., $\mathcal{S}^{\infty}_K= \mathcal{S}^{\infty}.$
The transition between the different coordinates and the change $p(n_1,n_2,n_3)$ between the sections $\s$ and $\tilde \s$ provide the following relation between the parameters
\[
w_3=c_1=C_1,w_4=c_2=\scriptstyle{\frac12} C_1+C_2.
\] 

\noindent
{\bf (2) Dual representation.}
In the case of dual representation, we get
\begin{gather*}
\mathcal{D}^{st}\circ \pi_0([w_1 \:|\:0,0,0])=\pi_1\circ \nabla^{\rho\circ \alpha}\circ (\id+\partial^*\nabla^{\rho\circ \alpha})([w_1 \:|\:0,0,0]) = \\
[\sum_{i} L_{X_i}L_{X_1}(w_1)\theta_i,\sum_{i} L_{X_i}L_{X_2}(w_1)\theta_i,\sum_{i} L_{X_i}L_{X_3}(w_1)\theta_i ] - \\
[w_1\:|\:L_{X_1}(w_1),L_{X_2}(w_1),L_{X_3}(w_1)]. \left[ \begin{smallmatrix} \theta_1&0&0\\ 
0&0&0\\ 
- \theta_2& - \theta_1&0\\ 
\theta_3& -\theta_1& \theta_1 \end{smallmatrix} \right]
\end{gather*}
for the pullback by $\s$ and $\mathcal{D}^{st}\circ \pi_0([w_1 \:|\:0,0,0])$ equals to
\begin{gather*}
\pi_1\circ \nabla^{\rho\circ \alpha}\circ (\id+\partial^*\nabla^{\rho\circ \alpha})([w_1 \:|\:0,0,0]) = \\
[\sum_{i} \partial_{n_i}\partial_{n_1}w_1dn_i,\sum_{i} \partial_{n_i}\partial_{n_2}w_1dn_i,\sum_{i} \partial_{n_i}\partial_{n_3}w_1dn_i] - \\
\left[w_1,\partial_{n_1}w_1- q_1n_2\partial_{n_3}w_1,\partial_{n_2}w_1+q_1 n_1\partial_{n_3}w_1,\partial_{n_3}w_1\right]
.
\left[ \begin {smallmatrix}
0&0&0
\\ 
0&0&0\\
0&0&0\\ 
 q_2(n_2dn_1-n_1dn_2)
&0&0\end{smallmatrix} \right]
\end{gather*}
for the pullback by $\tilde \s$. These provide the following six PDEs (the other three are linearly dependent on others)
\begin{gather*}
L_{X_1}L_{X_1}(w_1)-w_1=0,\ \ 
L_{X_1}L_{X_2}(w_1)+L_{X_2}(w_1)+L_{X_3}(w_1)=0,\\ 
L_{X_1}L_{X_3}(w_1)-L_{X_3}(w_1)=0, \\ 
L_{X_2}L_{X_2}(w_1)=0,\ \
L_{X_2}L_{X_3}(w_1)=0,\ \
L_{X_3}L_{X_3}(w_1)=0
\end{gather*}
for the pullback by $\s$ that have three--parameter family of solutions in $(y_1,y_2,y_3)$--coordinates $$w_1(y_1, y_2, y_3) = (C_1y_2+C_2)e^{-y_1}+C_3e^{y_1}.$$

Again, we do not write down the analogous but much more complicated PDEs for the pullback by $\tilde \s$, because we know in advance that the solutions in normal coordinates are first order polynomials in variables  $n_i$  and it is not hard to see from the formula for the BGG operator that in normal coordinates, the three--parameter family of solutions takes form $$w_1(n_1,n_2,n_3)=c_1n_1+c_2n_2+c_3.$$

Finally, let us return the exponential coordinates. We observe that $\rho\circ \alpha(x_1,x_2,x_3)$ acts on $[w_1,w_2,w_3]^t \in \mathcal{S}^{\infty}$ by the matrix $-\left[ \begin{smallmatrix} 0&x_1&0\\ 
x_1&0&0\\ 
x_2&- x_2& -x_1\\ 
 \end{smallmatrix} \right]^t$ and thus 
$$\exp(-\rho\circ \alpha(x_1,x_2,x_3))= {\footnotesize \left[ \begin {array}{ccc}  \cosh(
x_1)&\sinh(x_1)&x_2e^{-x_1}\\ 
\sinh(x_1)&\cosh(x_1)&-x_2e^{-x_1}\\ 
0&0&e^{-x_1}\end {array} \right] 
}.$$ 
Thus the solutions in  the exponential coordinates take form 
$$w_1(x_1,x_2,x_3)=w_1\cosh(
x_1)+w_2\sinh(x_1)+w_3x_2e^{-x_1},$$
which are globally defined, i.e., $\mathcal{S}^{\infty}_K= \mathcal{S}^{\infty}.$ The transition between the different coordinates and the change $p(n_1,n_2,n_3)$ between the sections $\s$ and $\tilde \s$ provide the following relation between the parameters
\[
w_3=c_1=C_2,\ \ 2w_2=c_2=C_3-C_2,\ \ 2w_1=c_3=C_2+C_3.
\] 

Let us now discuss the corresponding holonomy reductions in detail. 
\begin{itemize}
\item The stabilizer of an element of $\R^{4*}$ is an opposite parabolic subgroup $P_1^t$ of $PGL(4,\R)$ and thus, on the open orbit we get a Cartan geometry of type $(P_1^t,P_1^t\cap P_1=GL(3,\R))$ which is exactly an affine connection, and normality of the solution implies that this connection is Ricci flat. 
\item Moreover, there is a single closed orbit that carries a projective geometry of type $(PGL(3,\R),P_1)$.
\end{itemize}
 Since we consider $a=0$ in the construction of the extension in Section \ref{proj-alp}, we work with a closed connection in the projective class. Therefore, to obtain the change of this connection to a Ricci flat connection, we need to put $\U_i:=L(X_i)(\ln(\frac{1}{w_1(y_1, y_2, y_3)}))$ into $i$--th position in $\fg_1$ into the matrix
 \begin{gather*}
p_{ric}(y_1,y_2,y_3)=  {\footnotesize \left[ \begin {array}{cccc}1&\U_1=\frac{(C_1y_2+C_2)e^{-y_1}-C_3e^{y_1}}{w_1(y_1, y_2, y_3) }&\U_2=\frac{-C_1e^{-y_1}}{w_1(y_1, y_2, y_3) }&0
\\ 
0&1&0&0\\
0&0&1&0\\ 
0&0&
0&1\end {array} \right]}
\end{gather*}
for the change of the Weyl structure (and trivialization of the Cartan bundle). Then the new pullback
$\bar \s^*\omega:=\Ad_{p_{ric}(y_1,y_2,y_3)}^{-1}\s^*\omega+\delta p_{ric}(y_1,y_2,y_3)$ equals to
\begin{gather*}
\left[ \begin {smallmatrix}
-\U_2\theta_1-\U_1\theta_2&0&0&0\\ 
\theta_1&\U_2\theta_1&0&0\\
\theta_2&\frac{-2C_3e^{y_1}}{w_1(y_1, y_2, y_3) }\theta_2&-\theta_1+\U_1\theta_2&0\\ 
\theta_3&\frac{2(C_1y_2+C_2)\theta_3}{w_1(y_1, y_2, y_3)e^{y_1} }&\frac{-C_2e^{-y_1}-C_3e^{y_1}}{w_1(y_1, y_2, y_3) }\theta_1+\U_1\theta_3+\U_1y_2\theta_1&\theta_1\end{smallmatrix} \right],
\end{gather*}
which is defined for $w_1(y_1, y_2, y_3)\neq 0$, i.e., 
\begin{itemize}
\item everywhere for $C_1=\U_2=0$,  and 
\item outside of the set $y_2=\frac{-C_2-C_3e^{2y_1}}{C_1}$ for $C_1\neq 0$. 
\end{itemize}
We emphasize that the corresponding connection in the projective class  on the submanifold where $w_1(y_1, y_2, y_3)\neq 0$ is Ricci flat, because we have only normal solutions. Moreover, for $C_1=C_3=0$ or $C_1=C_2=0$ we obtain a $K$--invariant Ricci flat connection in the projective class. Further, in the case $C_1\neq 0$, the projective geometry on the set $y_2=\frac{-C_2-C_3e^{2y_1}}{C_1}$ is not $K$--homogeneous. We return to this case later, after we compute the remaining automorphisms.

\noindent
{\bf (3, 4) Symmetric powers.}
Since all solutions are normal, they can be obtained by coupling the solutions from (1) and (2), see Section \ref{normal}. 
Therefore, we get tractors with entries 
\begin{gather*}
w_6(y_1, y_2, y_3) = c_1e^{2y_1},\ \   w_9(y_1, y_2, y_3) ={\scriptstyle \frac12} c_1e^{2y_1}+c_2, \\ w_{10}(y_1, y_2, y_3)= {\scriptstyle \frac14}c_1e^{2y_1}+{\scriptstyle \frac12}c_2+c_3e^{-2y_1}
\end{gather*}
 for $V=S^2\R^4$ and 
\begin{gather*}
w_1(y_1, y_2, y_3) =c_1y_2^2e^{-2y_1}+c_2y_2e^{-2y_1}+c_3y_2+c_4e^{-2y_1}+c_5+c_6e^{2y_1}
\end{gather*}
 for $V=S^2\R^{4*}$. Let us note that these solutions are globally defined, but too degenerate too apply results of \cite{FG}.

\noindent
{\bf (5) Skew--symmetric powers.}
Let us work only in the normal coordinates in order to view the difference between the normal and all solutions. The formula for $D^{st}$ is analogous to the cases (1) and (2) with splitting operator $\id-\frac12\partial^*\nabla^{\rho\circ \alpha}$. These provide the following six  PDEs (the other three are linearly dependent on others)
\begin{gather*}
\partial_{n_1}w_3=-\partial_{n_3}w_6,\ \ 
\partial_{n_2}w_3=\partial_{n_3}w_5,\ \
\partial_{n_3}w_3=0,\\
\partial_{n_1}w_5=\partial_{n_2}w_6-q_2n_1w_3+q_1(n_2\partial_{n_3}w_5+n_1\partial_{n_3}w_6),\\
\partial_{n_2}w_5=-q_1 n_1\partial_{n_3}w_5, \ \ 
\partial_{n_2}w_6= q_2n_2w_3+q_1 n_2\partial_{n_3}w_6.
\end{gather*}
The solutions take form 
\begin{gather*}
w_3(n_1, n_2, n_3) = -4 n_1c_1-(n_1+1) c_2+4 c_3,\\
w_5(n_1, n_2, n_3)={\scriptstyle \frac{(n_1^2-1)\ln(\frac{n_1+1}{1-n_1})+2 n_1}{1-n_1} } c_1 +{\scriptstyle \frac{1}{1-n_1} }c_2 +  {\scriptstyle \frac{(n_1^2-1)\ln(\frac{n_1+1}{1-n_1})-2 n_1}{1-n_1} } c_3+(n_1+1)c_4,\\
w_6(n_1, n_2, n_3) ={\scriptstyle \frac{n_2\ln(\frac{n_1+1}{1-n_1})(n_1^2-1)+(4(1-n_1)n_3+2 n_2) n_1}{(1-n_1)n_1} }c_1 + {\scriptstyle \frac{(1-n_1) n_3+n_2}{1-n_1}}c_2 \\
+ {\scriptstyle \frac{n_2\ln(\frac{n_1+1}{1-n_1})(n_1^2-1)-2 n_2 n_1}{(1-n_1)n_1} }c_3+n_2c_4+c_5.
\end{gather*}
Let us discuss which of them are polynomial. Clearly, $c_3=-c_1$ simplifies the non--polynomial part of the solution $w_5$ to $c_1 \frac{4n_1}{1-n_1}+c_2\frac{1}{1-n_1}$, which becomes polynomial for $c_2=-4c_1$. What remains is the following three--parameter family of normal solutions
\begin{gather*}
w_3(n_1, n_2, n_3) = 0,\ 
w_5(n_1, n_2, n_3)=4c_1+c_4(n_1+1),\ 
w_6(n_1, n_2, n_3) =n_2c_4+c_5.
\end{gather*}

\noindent
{\bf (6) Adjoint representation.}
The formula for $D^{st}$ is analogous to the cases (1) and (2) with the splitting operator
 $\id-\frac32\partial^*\nabla^{\rho\circ \alpha}+\frac{9}{16}(\partial^*\nabla^{\rho\circ \alpha})^2-\frac{1}{16}(\partial^*\nabla^{\rho\circ \alpha})^3$. The formula for $D^{aut}$ has an additional part that is an insertion into the Weyl curvature symmetrized as an element of $\otimes^2 \fg_{-1}^*\otimes \fg_{-1}$.  These provide the following fifteen PDEs, where the insertion into the Weyl curvature is in bold font
\begin{gather*}
 \partial_{  y_3}^{2}w_2=0,\ \
 \partial_{  y_3}^{2}w_3  =0,\ \
 \partial_{  y_3}\partial_{  y_2}w_2=0,\ \
y_2\partial_{  y_3}^{2}w_2 +  \partial_{  y_3}\partial_{  y_1}w_2 -  \partial_{  y_3}\partial_{  y_2}w_3=0,\\
\partial_{  y_2}^{2}   w_4 =0,\ \  \partial_{  y_2}^{2}w_2=0,\ \
4 \partial_{  y_3}w_2  -2 \partial_{  y_3}\partial_{  y_2}   w_3 + \partial_{  y_3}^{2} w_4  =0,\\
4 \partial_{  y_2}   w_2 +2  \partial_{  y_3}   w_2  +2   y_2 \partial_{  y_3}\partial_{  y_2}   w_2  +2 \partial_{  y_2}\partial_{  y_1}w_2 - \partial_{  y_2}^{2}   w_3  =0,\\
   y_2 \partial_{  y_3}^{2}   w_3 -2  \partial_{  y_3} w_3  + \partial_{  y_3}\partial_{  y_1}   w_3=0,\\
4  \partial_{  y_2}   w_2  -\partial_{  y_2}^{2}   w_3+2  \partial_{  y_3}\partial_{  y_2}w_4=0,\\
2   y_2  \partial_{  y_3}\partial_{  y_2}   w_3-y_2^{2}   \partial_{  y_3}^{2}   w_2      -2   y_2  \partial_{  y_3}\partial_{  y_1}   w_2  +2 \partial_{  y_2}\partial_{  y_1}   w_3  - \\
 \partial_{  y_1}^{2}   w_2  -2   y_2 \partial_{  y_3}   w_2 +2 \partial_{  y_3}   w_3  -2 \partial_{  y_1}   w_2 =0,\\
2   y_2 \partial_{  y_3}   w_2  -   y_2 \partial_{  y_3}\partial_{  y_2}   w_3  +   y_2 \partial_{  y_3}^{2}   w_4 -2  \partial_{  y_3}   w_3  +2 \partial_{  y_1}   w_2 -  \partial_{  y_2}\partial_{  y_1}   w_3  +  \partial_{  y_3}\partial_{  y_1}   w_4  =0,\\  
 y_2^{2} \partial_{  y_3}^{2}   w_3  -2   y_2 \partial_{  y_3}   w_3  +2   y_2 \partial_{  y_3}\partial_{  y_1}   w_3 + \partial_{  y_1}^{2}   w_3  -2 \partial_{  y_1}   w_3 =0,\\
2  \partial_{  y_2}   w_4  + \partial_{  y_3}   w_4 +   y_2  \partial_{  y_3}\partial_{  y_2}   w_4 - \partial_{  y_2}   w_3 +  \partial_{  y_2}\partial_{  y_1}   w_4 +    w_2+{\bf w_2}  =0,\\
y_2^{2}   \partial_{  y_3}^{2}   w_4   -2   y_2 \partial_{  y_3}   w_3 +2  y_2 \partial_{  y_3}\partial_{  y_1}   w_4 +2   y_2 \partial_{  y_3}   w_4  + \\
2 \partial_{  y_1}   w_4 -2  \partial_{  y_1}   w_3  + \partial_{  y_1}^{2}  w_4+{\bf w_3}=0
   \end{gather*}
   
The solutions take form 
\begin{gather*}
w_2(y_1, y_2, y_3) = C_1, w_3(y_1, y_2, y_3) = C_2 y_2+C_3+C_4 e^{2 y_1}, \\
w_4(y_1, y_2, y_3) = {\scriptstyle \frac12}(2 C_5 y_2- C_6)e^{-2 y_1}+C_4 {\scriptstyle\frac{A+2}{4}} e^{2 y_1}+{\scriptstyle \frac12}((- C_1- C_2) A- \\C_1+ C_2)) y_2+ (C_2 y_3+ C_3 y_1) A+C_7,
\end{gather*}
 where $A=0$ for solutions of $D^{st}$ and $A=1$ for solutions of $D^{aut}$. The infinitesimal automorphisms take form $w_2(y_1, y_2, y_3)X_1+w_3(y_1, y_2, y_3)X_2+w_4(y_1, y_2, y_3)X_3$ and we can apply the splitting operator to obtain an extension $\alpha_f$ from the full Lie algebra of infinitesimal automorphisms of this example
\[\alpha_f(C_1,\dots,C_7):=  \left[ \begin {smallmatrix} 
-\frac12  C_2& C_1&0&0\\ 
 C_1&-\frac12  C_2&0&0\\ 
  C_3+ C_4&- C_3+ C_4&-C_1+\frac12  C_2&0\\
   -\frac12  C_6+\frac34  C_4+ C_7&\frac12  C_6+\frac54  C_4+ C_7& C_5- C_1& C_1+\frac12  C_2
   \end {smallmatrix} \right].
\]
For $C_1=x_1,C_2=0,C_3=x_2,C_4=0,C_5=0,C_6=0,C_7=x_3$ we recover the original extension $\alpha$ and $x_1(\partial_{y_1}+y_2\partial_{y_3})+x_2\partial_{y_2}+( -x_1 y_2+x_2y_1 +x_3)\partial_{y_3}=x_1\partial_{y_1}+x_2(\partial_{y_2}+y_1\partial_{y_3})+x_3\partial_{y_3}$ are the right--invariant vector fields on $K$.

Let us now return to the projective geometry on the set $$M_0:=\{(y_1,y_2,y_3): y_2=\frac{-C'_2-C'_3e^{2y_1}}{C'_1}\}$$ from the point (2), where we add $'$ to distinguish the parameters. It is easy to compute that the automorphisms with $C_1=x_1,C_2=x_1, C_3=\frac{x_1C_3'}{C_1'}, C_4=0,C_6=-x_2+x_3,C_7=x_2+x_3$ are tangent to $M_0$ and we can restrict $\alpha_f$ to  \[
\alpha(x_1,x_2,x_3):= \left[ \begin {smallmatrix} 
-\frac23x_1& x_1&0\\ 
 x_1&-\frac23x_1&0\\ 
   x_2&x_3& \frac43x_1
   \end {smallmatrix} \right]
\]
for $\fk=x_1e_1+x_2e_2+x_3e_3$ with $[e_1,e_2]=2e_2-e_3,[e_1, e_3] =2e_3-e_2$. It is not hard to check that this is normal extension $\fk\to \frak{sl}(3,\R)$ that gives the flat projective geometry on $M_0$.

\subsection{Solutions in C-projective geometry.} \label{exp-cproj}
We present here solutions for the example in Sections \ref{cproj-alp} and \ref{cproj-S}. The Lie algebra $\fk(\C)=\frak{sl}(2,\C)+\C^2$ of infinitesimal automorphisms decomposes into the radical $\langle  e_1-2e_5,e_2\rangle$ and semisimple part $\langle e_1,e_3,e_4 \rangle$. We use real coordinates $x_{2j-1}+{\rm i}x_{2j}$ corresponding to $e_j$. Then it is reasonable to fix the complement $\frak{c}$ as ${\sf c}=[x_1,x_2,x_5,x_6,x_3,x_4,x_3,-x_4,0,0]$, where $x_1$ corresponds to the Cartan subalgebra, $x_2,x_3,x_4$ to the maximal compact subalgebra $\frak{su}(2) \subset \frak{sl}(2,\C)$ and $x_5,x_6$ are contained in the radical. We use the exponential coordinates from Section \ref{exponential-coord} to present the results for tractor bundles of interest.

\noindent
{\bf (1) Standard and conjugate representation.}
 We compute that $\rho\circ \alpha({\sf c})$ acts on solutions $ [z_1,z_3]^t \in \mathcal{S}^{\infty}$ by the matrix 
$\left[ \begin{smallmatrix}
0&0\\ 
x_5+{\rm i}x_6&0
 \end{smallmatrix} \right]$
and thus $\exp(-\rho\circ \alpha({\sf c}))$ equals to
$ \left[ \begin {smallmatrix} 
1 & 0\\
-x_5-{\rm i}x_6&1
\end {smallmatrix} \right] 
.$
Thus only the radical acts on the space of solutions and the formulas  
$$ z_1=c_1,\ \ 
z_3=-c_1(x_5+{\rm i}x_6)+c_3
$$ 
are globally defined, i.e., $\mathcal{S}^{\infty}_K=\mathcal{S}^{\infty}$, and coincide in the exponential coordinates from Sections \ref{exponential-coord} and \ref{1.2}.

\noindent
{\bf (2) Dual and conjugate dual representation.}
For the dual representation, we compute that $\rho\circ \alpha({\sf c})$ acts on solutions $ [w_1,w_2,w_3,w_4,w_7,w_8]^t \in \mathcal{S}^{\infty}$ by the matrix 
$$
\left[ \begin {smallmatrix} 0&0&-x_{{1}}&-x_{{2}}&-x_{{3}}&-x_{{4}}
\\ 0&0&x_{{2}}&-x_{{1}}&x_{{4}}&-x_{{3}}
\\ 0&0&{1 \over 2}x_{{1}}&{1 \over 2}x_{{2}}&{1 \over 2}x_{{3}}&{1 \over 2}x
_{{4}}\\ 0&0&-{1 \over 2}x_{{2}}&{1 \over 2}x_{{1}}&-{1 \over 2}x_{{4}
}&{1 \over 2}x_{{3}}\\ 0&0&-x_{{3}}&x_{{4}}&-{1 \over 2}x_{{1}}&
-{1 \over 2}x_{{2}}\\ 0&0&-x_{{4}}&-x_{{3}}&{1 \over 2}x_{{2}}&-
{1 \over 2}x_{{1}}
\end {smallmatrix} \right] 
$$
For the conjugate case, the matrix differs only in signs of the imaginary parts $x_2, x_4$.

It turns out that the action of the radical on solutions is trivial. Thus we deal only with the action of the semisimple part and the Iwasawa decomposition according to Proposition \ref{modify-exp}.
We consider the change of the basis of the space of solutions given by the transition matrix  
$$
T= \left[ \begin {smallmatrix} 2&0&-2&0&0&0\\ 2&0&-2
&2&0&-1\\ 0&0&1&0&0&0\\ 0&0&1&0&0&
{1 \over 2}\\ 0&1&0&0&0&0\\ 0&1&0&0&1&0
\end {smallmatrix} \right] 
.$$
Then the action of the Cartan subalgebra on the space of solutions becomes diagonal. In the new basis, ${\sf A}$ acts for both cases on solutions by the diagonal matrix $A=diag(1,e^{{x_1 \over 2}},e^{-{x_1 \over 2}},1,e^{{x_1 \over 2}},e^{-{x_1 \over 2}})$.
In the dual case, we view the action of ${\sf K}$ as a composition of the following three matrices 
$$
R_1= \left[ \begin {smallmatrix} 1&0&0&0&0&0\\ 0&\sin
 ( x_{{2}} ) +\cos ( x_{{2}} ) &0&0&\sin ( 
x_{{2}} ) &0\\ 0&0& \cos ( x_{{2}} )-\sin ( x_{{2}}
 ) &0&0&-{1 \over 2}\sin ( x_{{2}}
 ) \\ 0&0&0&1&0&0\\ 0&-2
\sin ( x_{{2}} ) &0&0& \cos
 ( x_{{2}} ) -\sin ( x_{{2}} )&0\\ 0&0&4\sin ( x_{{
2}}) &0&0&\sin ( x_{{2}}) +\cos ( x_{{2}}
 ) \end {smallmatrix} \right],
$$
$$
R_2= \left[ \begin {smallmatrix} 1&0&0&0&0&0\\ 0&\cos
 ( x_{{3}} ) &\sqrt {2}\sin ( x_{{3}} ) &0&0&0
\\ 0&-{1 \over 2}\sqrt {2}\sin ( x_{{3}} ) &
\cos ( x_{{3}} ) &0&0&0\\ 0&0&0&1&0&0
\\ 0&0&0&0&\cos ( x_{{3}} ) &{1 \over 2}\sqrt {
2}\sin ( x_{{3}} ) \\ 0&0&0&0&-\sqrt {2}
\sin ( x_{{3}} ) &\cos ( x_{{3}} ) \end {smallmatrix}
 \right],
$$
$$
R_3=\left[ \begin {smallmatrix} 1&0&0&0&0&0\\ 0&\cos
 ( x_{{4}} ) &-\sqrt {2}\sin ( x_{{4}} ) &0&0&-1
/2\sqrt {2}\sin ( x_{{4}} ) \\ 0&-{1 \over 2}
\sqrt {2}\sin ( x_{{4}} ) &\cos ( x_{{4}} ) &0&-
{1 \over 2}\sqrt {2}\sin ( x_{{4}} ) &0\\ 0&0&0
&1&0&0\\ 0&0&2\sqrt {2}\sin ( x_{{4}}
 ) &0&\cos ( x_{{4}} ) &{1 \over 2}\sqrt {2}\sin ( x_{
{4}} ) \\ 0&2\sqrt {2}\sin ( x_{{4}}
 ) &0&0&\sqrt {2}\sin ( x_{{4}} ) &\cos ( x_{{4}
} ) \end {smallmatrix} \right].
$$
In the conjugate case, the matrices $R_i$ differ only in signs of imaginary parts $x_2, x_4$. Altogether, $\exp(-\rho\circ \alpha({\sf c}))$ equals to $TAR_3R_2R_1T^{-1}$. We compute that in the dual case, all the solution  take in exponential coordinates  form $z_1$, where
\begin{gather*}
\Re(z_1)=c_{{1}}+ ( {{ e}^{-{x_{{1}} \over 2}}} (  2
\sin ( x_{{2}} ) \sin ( x_{{4}} ) \sin ( 
x_{{3}} )  )-2\cos ( x_{{2
}} ) \cos ( x_{{3}} ) \cos ( x_{{4}} ) +2 ) c_{{3}}+ 
\\
 ( 2\sin ( x_{
{4}} ) \sin ( x_{{3}} ) \cos ( x_{{2}} ) +
2\sin ( x_{{2}} ) \cos ( x_{{3}} ) \cos
 ( x_{{4}} )  ) {{ e}^{-{x_{{1}} \over 2}}}c_{{4}}+
 \\
 ( \cos ( x_{{2}} ) \cos ( x_{{4}} ) \sin
 ( x_{{3}} ) -\sin ( x_{{2}} ) \cos ( x_{{3
}} ) \sin ( x_{{4}} )  ) \sqrt {2}{{ e}^{-{x_{{1}} \over 2}}}c_{{7}}+
\\
 ( \cos ( x_{{3}} ) \sin ( x_
{{4}} ) \cos ( x_{{2}} ) +\sin ( x_{{2}}
 ) \cos ( x_{{4}} ) \sin ( x_{{3}} ) 
 ) \sqrt {2}{{ e}^{-{x_{{1}} \over 2}}}c_{{8}},
\end{gather*}
\begin{gather*}
\Im(z_1)=c_{{2}}+ ( 2\sin ( x_{{4}} ) \sin ( x_{{3}}
 ) \cos ( x_{{2}} )+ 2\sin ( x_{{2}} ) 
\cos ( x_{{3}} ) \cos ( x_{{4}} )  ) {
{ e}^{-{x_{{1}} \over 2}}}c_{{3}}+ 
\\
( {{ e}^{-{x_{{1}} \over 2}}}
 (2\sin ( x_{{2}} ) \sin
 ( x_{{4}} ) \sin ( x_{{3}} )  ) -2\cos ( x_{{2}} ) \cos ( x_{{3}} ) 
\cos ( x_{{4}} )  +2
 ) c_{{4}}- 
\\
 ( \cos ( x_{{3}} ) \sin ( x_{{
4}} ) \cos ( x_{{2}} ) +\sin ( x_{{2}} ) 
\cos ( x_{{4}} ) \sin ( x_{{3}} )  ) 
\sqrt {2}{{ e}^{-{x_{{1}} \over 2}}}c_{{7}}+ 
\\
 ( \cos ( x_{{2}}
 ) \cos ( x_{{4}} ) \sin ( x_{{3}} ) -\sin
 ( x_{{2}} ) \cos ( x_{{3}} ) \sin ( x_{{4}
} )  ) \sqrt {2}{{ e}^{-{x_{{1}} \over 2}}}c_{{8}}.
\end{gather*}
In the conjugate case, all the solution in  exponential coordinates differ just in the signs of imaginary parts $x_2, x_4$. Both are globally defined.


\noindent
{\bf (4) Representation on Hermitian matrices of order $4$.} 
All solutions are normal, so we can use coupling to describe all of them using the normal solutions from part (2). The coupling provides the following relations between the parameters describing the spaces of solutions, where we denote by the index $d$ the dual solution, and by the index $c$ the conjugate solutions:
$c_1= 2c_1^d c_4^c$, $c_4=2 c_1^d c_8^c$, $c_5= 2c_3^d c_8^c$, $c_7= c_1^d c_1^c$, $c_8= 2c_1^d c_3^c$, $c_9= c_3^d c_3^c$, $c_{13}= 2c_1^d c_7^c$, $c_{14}= 2c_3^d c_7^c$ and $c_{16}= c_7^d c_7^c$.
This gives all solutions as
\begin{gather*}
w_7(x_i)= -( 4\sin ( x_{{3}} ) \sin ( x_{{4}} ) 
\cos ( x_{{2}} ) +4\sin ( x_{{2}} ) \cos
 ( x_{{3}} ) \cos ( x_{{4}} )  ) {{ e}^
{-{x_{{1}} \over 2}}}c_{{1}}+ 
\\
(  2\cos ( x_
{{4}} ) \sin ( x_{{3}} ) \cos ( x_{{2}} ) 
 -2\sin ( x_{{2}} ) \sin
 ( x_{{4}} ) \cos ( x_{{3}} )) \sqrt {2}{{ e}^{-{x_{{1}} \over 2}}}c_{{13}}+ 
\\
(  (  4\cos ( x_{{4}} ) \sin ( x_{{3}}
 ) \cos ( x_{{2}} )  -4
\sin ( x_{{2}} ) \sin ( x_{{4}} ) \cos ( 
x_{{3}} )) \sqrt {2}{{ e}^{-{x_{{1}} \over 2}
}}+ 
\\
 ( -2\sqrt {2}\sin ( 2x_{{3}} ) \cos
 ( 2x_{{2}} ) \cos ( 2x_{{4}} ) +2\sqrt {2
}\sin ( 2x_{{4}} ) \sin ( 2x_{{2}} ) 
 ) {{ e}^{-x_{{1}}}} ) c_{{14}}+ 
\\
( -\cos ( 2
x_{{4}} ) \cos ( 2x_{{3}} ) +1 ) {{ e}^{
-x_{{1}}}}c_{{16}}+ 
\\
 ( -2\sin ( x_{{4}} ) \cos
 ( x_{{3}} ) \cos ( x_{{2}} ) -2\sin ( x_
{{2}} ) \cos ( x_{{4}} ) \sin ( x_{{3}} ) 
 ) \sqrt {2}{{ e}^{-{x_{{1}} \over 2}}}c_{{4}}+
\\
 (  ( -4
\sin ( x_{{4}} ) \cos ( x_{{3}} ) \cos ( 
x_{{2}} ) -4\sin ( x_{{2}} ) \cos ( x_{{4}}
 ) \sin ( x_{{3}} )  ) \sqrt {2}{{ e}^{-{x_{{1}} \over 2}
}}+ 
\\
( 2\sqrt {2}\sin ( 2x_{{3}} ) \sin
 ( 2x_{{2}} ) \cos ( 2x_{{4}} ) +2\sqrt {2
}\sin ( 2x_{{4}} ) \cos ( 2x_{{2}} ) 
 ) {{ e}^{-x_{{1}}}} ) c_{{5}}+
c_{{7}}+  
\\
(  ( 
4\sin ( x_{{3}} ) \sin ( x_{{4}} ) \sin
 ( x_{{2}} ) -4\cos ( x_{{3}} ) \cos ( x_
{{4}} ) \cos ( x_{{2}} )  ) {{ e}^{-{x_{
{1}} \over 2}}}+4 ) c_{{8}}+ 
\\
(  ( 8\sin ( x_{{3}}
 ) \sin ( x_{{4}} ) \sin ( x_{{2}} ) -8
\cos ( x_{{3}} ) \cos ( x_{{4}} ) \cos ( x_
{{2}} )  ) {{ e}^{-{x_{{1}} \over 2}}}+ 
\\
 ( 2\cos
 ( 2x_{{4}} ) \cos ( 2x_{{3}} ) +2 ) {
{ e}^{-x_{{1}}}}+4 ) c_{{9}},
\end{gather*}
which are all globally defined. 

Let us focus on the holonomy reduction determined by the normal solution
\begin{gather*}
\nu=3-8{{ e}^{-{x_{{1}} \over 2}}}\cos ( x_{{2}} ) \cos ( 
x_{{3}} ) \cos ( x_{{4}} ) +8{{ e}^{-{x_{{1}} \over 2}}}\sin ( x_{{2}} ) \sin ( x_{{4}} ) \sin
 ( x_{{3}} ) + 
\\
4{{ e}^{-x_{{1}}}}  \cos^{2} ( x_
{{3}} )    \cos^{2}( x_{{4}} ) 
 -2{{ e}^{-x_{{1}}}}  \cos^{2} ( x_{{3}}
 )  -2{{ e}^{-x_{{1}}}}  \cos^{2} ( x_{{
4}} )  +4{{ e}^{-x_{{1}}}}
\end{gather*}
The stabilizer of the corresponding Hermitian matrix $h$ is as follows
$$
h=\left[ \begin {smallmatrix} -1&0&0&0\\ 0&1&0&0
\\ 0&0&0&0\\ 0&0&0&1\end {smallmatrix}
 \right],\ \ 
stab(h)=\left[ \begin {smallmatrix} {\rm i} ( -{ x_8}-{ x_{10}}-{ x_{12}}
 ) &{ x_{30}}-{\rm i}{ x_{29}}&0&{ x_5}-{\rm i}{ x_6}
\\ { x_{30}}+{\rm i}{ x_{29}}&{\rm i}{ x_8}&0&-{ x_{21}}+{\rm i}{
 x_{22}}\\ { x_3}+{\rm i}{ x_4}&{ x_{17}}+{\rm i}{ x_{18}}&{\rm i}
{ x_{10}}&{ x_{19}}+{\rm i}{ x_{20}}\\ { x_5}+{\rm i}{ x_6}&
{ x_{21}}+{\rm i}{ x_{22}}&0&{\rm i}{ x_{12}}\end {smallmatrix} \right] 
$$
which is a semidirect product of $\frak{su}(1,2)\oplus\frak{u}(1)$ acting on $\C^3=[{ x_3}+{\rm i}{ x_4},{ x_{17}}+{\rm i}{ x_{18}},{ x_{19}}+{\rm i}{ x_{20}}]$. 

There are several orbits characterized by the property that the solution $\nu$ is vanishing or not. The open orbits with $\nu \neq 0$ carry an affine connection preserving a complex structure and Hermitian product of signature $(1,1,1)$ or $(2,0,1)$. 
\begin{rem}
If the Hermitian product was not degenerate, then the normality of the solution would imply that the product is  K\" ahler--Einstein, \cite{CEMN}.
In our situation, we can only say that the affine connection has a special Rho--tensor.
\end{rem}

The closed orbit with $\nu=0$ is a five--dimensional submanifold in complex manifold and inherits a CR--structure. The reduction shows that the CR--structure is a product of $\C$ with a three--dimensional Levi non--degenerate CR--submanifold in $\C^2$. Moreover, the reduction provides an affine connection preserving the induced CR--structure.

\subsection{Solutions in $(2,3,5)$ distributions.} \label{exp-234}
Since this example is semisimple, we present here the solutions for the example from Sections \ref{235-alp} and \ref{235-S} in the coordinates on $K/H=(SL(2,\R)\times SO(3))/diag(SO(2))$ as follows
$$\left[ \begin {smallmatrix} \cos ( { x_1} ) { x_2}&\cos
 ( { x_1} ) { x_2}{ x_3}-{\frac {\sin ( { 
x_1} ) }{{ x_2}}}\\ \sin ( { x_1}
 ) { x_2}&\sin ( { x_1} ) { x_2}{ x_3}+{
\frac {\cos ( { x_1} ) }{{ x_2}}}\end {smallmatrix} \right]\times \left[ \begin {smallmatrix} \cos ( { x_4} ) \cos ( {
 x_5} ) &\sin ( { x_4} ) &\cos ( { x_4}
 ) \sin ( { x_5} ) \\ \noalign{\medskip}-\sin
 ( { x_4} ) \cos ( { x_5} ) &\cos ( {
 x_4} ) &-\sin ( { x_4} ) \sin ( { x_5}
 ) \\ -\sin ( { x_5} ) &0&\cos
 ( { x_5} ) \end {smallmatrix} \right].
 $$
 In the case of $SL(2,\R)$, these correspond to Iwasawa decomposition, where $x_1$ is compact, $x_2$ is abelian and $x_3$ is the nilpotent part, and in the case of $SO(3)$, $x_4,x_5$ correspond to two rotations that do not generate $diag(SO(2))$ together with $x_1$.

Let us note that in the cases (1) and (2), the solutions always sit in trivial representations so they are constant. In particular, we did not find any additional solutions for the second skew--symmetric power and thus infinitesimal automorphisms, so we do not express the solutions in the coordinates. We focus here on the case (3), i.e., metrizability.

\noindent
{\bf (3) Metrizability.} Firstly, we shall find the coordinate expression to identify submetrics with $K\times_H S^2 (\fk^{-1}/\fk^0)^*=K\times_H V_{-6}$. 
We can easily compute the Maurer--Cartan forms in our coordinates by the usual formula $k^{-1}dk$ and compute dual elements to $e_4,e_5$ as
\begin{gather*}
\theta_4={(x_2^4x_3^2-2x_2^4x_3-x_2^4+1) dx_1 \over 4x_2^2}-{(x_3-1) dx_2 \over 2x_2}-{dx_3 \over 4} -{\cos(x_5)dx_4 \over 4}+{dx_5 \over 4},\\
\theta_5=-{(x_2^4x_3^2-x_2^4+1)dx_1 \over 2x_2^2}+ {x_3 dx_2 \over x_2} +{dx_3 \over 2} -{dx_5 \over 2}.
\end{gather*}
This gives identification
$$K\times_H V_{-6}= \left[ \begin {array}{cc}s_1(x_i)  \theta_4\odot \theta_4&s_2(x_i) \theta_5\odot \theta_4\\
s_2(x_i)  \theta_4\odot \theta_5&s_3(x_i) \theta_5\odot \theta_5\end {array} \right].$$

The solutions in the representation $S^4\R^2$ of $\frak{sl}(2,\R)$ are submetrics that take in our coordinates form
\begin{gather*}
s_1(x_i)={\scriptstyle \frac {1}{x_{{2}}^{4}}}(6+ ( 6+ ( 6x_{{3}}^{4}-12x_{{3}}^{2}+6
 ) x_{{2}}^{8}+ ( 12-36{x_{{3}}}^{2} ) x_{{2}}
^{4} )   \cos^{4} ( x_{{1}} )-
\\
24x_
{{2}}^{2} (  ( x_{{3}}^{2}-1 ) x_{{2}}^{4}
 -1) x_{{3}}\sin ( x_{{1}} )   \cos^{3}( x_{{1}
} )   + (  ( 36x_{{3}}^{2}-12
 ) x_{{2}}^{4} -12)   \cos^{2} ( x_{{1}} ) -
\\
24\sin ( x_{{1}} ) \cos ( x_{{1}}
 ) x_{{3}}x_{{2}}^{2})c_1+ 
{\scriptstyle \frac {1}{x_{{2}}^{4}}}(24 (  ( x_{{3}}^{2}-1 ) x_{{2}}^{4}
 -1) x_{{3}}x_{{2}}^{2}  \cos^{4} ( x_{{1}} ) 
  +6 ( ( x_{{3}}^{2}-1 ) x_{{2}}^
{4}-
\\
2x_{{3}}x_{{2}}^{2}  -1) \sin ( x_{{1}} ) 
 (  ( x_{{3}}^{2}-1 ) x_{{2}}^{4}+2x_{{3}}x_
{{2}}^{2}-1 )   \cos^{3} (x_{{1}} )  +
 (  18( x_{{3}}-x_{{3}}^{3} ) x_{{2}}^{6}+
\\ 
30x_{{3}}x_{{2}}^{2} )   \cos^{2} ( x_{{1}}  
 ) + ( ( 18x_{{3}}^{2}-6 ) x_{{2}}
^{4}-6 ) \sin ( x_{{1}} ) \cos ( x_{{1}} ) -
6x_{{3}}x_{{2}}^{2})c_2+
\end{gather*}
\begin{gather*}
{\scriptstyle \frac { 1}{x_{{2}}^
{4}}}(( ( 36x_{{3}}^{2}-12 ) x_{{2}}^
{4} - ( 6x_{{3}}^{4}-12x_{{3}}^{2}+6
 ) x_{{2}}^{8}-6 )   \cos^{4} ( x_{{1}} )   +24x_{
{2}}^{2} (  ( x_{{3}}^{2}-1 ) x_{{2}}^{4}-
\\
1
 ) x_{{3}}\sin ( x_{{1}} )   \cos^{3} ( x_{{1}
} )   + ( 6+ ( 6x_{{3}}^{4}-12x_{{3}
}^{2}+6 ) x_{{2}}^{8}+ ( 12-36x_{{3}}^{2} ) {
x_{{2}}}^{4} )   \cos^{2} ( x_{{1}} )  - 
\\
12 (  ( x_{{3}}^{2}-1 ) x_{{2}}^{4} -1) 
\sin ( x_{{1}} ) x_{{3}}x_{{2}}^{2}\cos ( x_{{1}}
 ) + ( 6x_{{3}}^{2}-2 ) x_{{2}}^{4})c_3+
 {\scriptstyle \frac {1}{x_{{2}}^{4}}}(24x_{{3}}x_{{2}}^{2} ( 1- 
\\
( x_{{3}}^{2}-1
 ) x_{{2}}^{4} )   \cos^{4} ( x_{{1}} ) 
 -6\sin ( x_{{1}} )  (  ( x_{{3
}}^{2}-1 ) x_{{2}}^{4}+2x_{{3}}x_{{2}}^{2} -1 ) 
 (  ( x_{{3}}^{2}-1 ) x_{{2}}^{4}-2x_{{3}}x_
{{2}}^{2} -
\\
1)   \cos^{3} ( x_{{1}} )   +
 (  30( x_{{3}}^{3}-x_{{3}} ) x_{{2}}^{6}-18
x_{{3}}x_{{2}}^{2} )   \cos^{2} ( x_{{1}} ) 
  +6\sin ( x_{{1}} )  (  ( x_{{3}}
^{4}-2x_{{3}}^{2}+1 ) x_{{2}}^{4}-
\\
3x_{{3}}^{2}+1
 ) x_{{2}}^{4}\cos ( x_{{1}} ) + 6( x_{3} -x_{{3}
}^{3}) x_{{2}}^{6})c_4+
{\scriptstyle \frac { 1}{x_{{2}}^{4}}}(( 6+ ( 6x_{{3}}^{4}-12x_{{3}}^{2}+6
 ) x_{{2}}^{8}+ ( 12-
\\
36x_{{3}}^{2} ) x_{{2}}
^{4} )   \cos^{4}( x_{{1}} )  -24x_
{{2}}^{2}\sin ( x_{{1}} ) x_{{3}} (  ( x_{{3
}}^{2}-1 ) x_{{2}}^{4} -1 )   \cos^{3} ( x_{{1}}
 )   -12x_{{2}}^{4} (  ( x_{{3}}^{4}-
2x_{{3}}^{2}+
\\
1 ) x_{{2}}^{4}-3x_{{3}}^{2}+1 ) 
  \cos^{2} ( x_{{1}} )  + 24( x_{{3}}
^{3}-x_{{3}} ) x_{{2}}^{6}\sin ( x_{{1}} ) 
\cos ( x_{{1}} ) +6x_{{2}}^{8}(x_{{3}}^{4}-2x_{{3}}^{2}+1) 
c_5\\
s_2(x_i)= {\scriptstyle \frac {12 }{x_{{2}}^{2}}}(\cos ( x_{{1}} )  ( x_{{2}}^{2} 
(  ( x_{{3}}^{2}-1 ) x_{{2}}^{4} -3) x_{{3}}
  \cos^{3} ( x_{{1}} )  + ( 1+ ( 1-
\\
3
x_{{3}}^{2} ) x_{{2}}^{4} ) \sin ( x_{{1}}
 )   \cos^{2} ( x_{{1}} )   +3\cos
 ( x_{{1}} ) x_{{2}}^{2}x_{{3}}-\sin ( x_{{1}}
 )  ))c_1+
 {\scriptstyle \frac {1}{x_{{2}}^{2}}}( (  ( 36x_{{3}}^{2}-
\\
12 ) x_{{2}}^
{4} -12)   \cos^{4} ( x_{{1}} )   +12x_{
{2}}^{2}\sin ( x_{{1}} ) x_{{3}} (  ( x_{{3}
}^{2}-1 ) x_{{2}}^{4} -3)   \cos^{3} ( x_{{1}}
 )  + ( 15+ ( 9-
\\
27x_{{3}}^{2}
 ) x_{{2}}^{4} )   \cos^{2} ( x_{{1}} ) 
  +18\sin ( x_{{1}} ) \cos ( x_{{1}}
 ) x_{{3}}x_{{2}}^{2}-3)c_2+
 {\scriptstyle \frac {1}{x_{{2}}^{2}}}(6x_{{3}}x_{{2}}^{2}-12x_{{2}}^{2} (  ( x_{{3}}^{2}-1 ) 
x_{{2}}^{4} -
\\
3) x_{{3}}  \cos^{4} ( x_{{1}} ) 
  + (  ( 36x_{{3}}^{2}-12 ) x_{{2}
}^{4} -12) \sin ( x_{{1}} )   \cos^{3} ( x_{{1}}
 )   +12x_{{2}}^{2} (  ( x_{{3}}^{
2}-1 ) x_{{2}}^{4} 
\\
-3) x_{{3}}  \cos^{2} ( x_{{1}}
 )   + ( 6+ ( 6-18{x_{{3}}}^{2} ) 
x_{{2}}^{4} ) \sin ( x_{{1}} ) \cos ( x_{{1}}
 ) )c_3+
 {\scriptstyle \frac { 1}{x_{{2}}^{2}}}(( 12+ (12 -
\\
36x_{{3}}^{2} ) x_{{2}}^{4}
 )   \cos^{4} ( x_{{1}} )   -12x_{{2}
}^{2}\sin ( x_{{1}} ) x_{{3}} (  ( x_{{3}}^
{2}-1 ) x_{{2}}^{4} -3 )   \cos ^{3} ( x_{{1}}
 )  + (  ( 45x_{{3}}^{2}-15
 ) x_{{2}}^{4} -
\\
9)   \cos^{2} ( x_{{1}} ) 
  +12x_{{2}}^{2} ( ( x_{{3}}^{2}-1
 ) x_{{2}}^{4} -{\scriptstyle {3 \over 2}}) x_{{3}}\sin ( x_{{1}} ) \cos
 ( x_{{1}} ) + ( 3-9x_{{3}}^{2} ) x_{{2}}
^{4})c_4+
\\
{\scriptstyle \frac {1}
{x_{{2}}^{2}}}(12x_{{2}}^{2} (  ( x_{{3}}^{2}-1 ) x
_{{2}}^{4} -3) x_{{3}}  \cos^{4} ( x_{{1}} ) 
  + ( 12+ ( 12-36x_{{3}}^{2} ) x_{{2}
}^{4} ) \sin ( x_{{1}} )   \cos^{3} ( x_{{1}}
 )   + 
 \\
(  24( x_{{3}}-
x_{{3}}^{3}
 ) x_{{2}}^{6}+36x_{{3}}x_{{2}}^{2} )   \cos^{2}
 ( x_{{1}} )   +36x_{{2}}^{4} ( x_{{3}
}^{2}-{\scriptstyle {1 \over 3}} ) \sin ( x_{{1}} ) \cos ( x_{{1}}
 ) + 12( x_{{3}}^{3}-x_{{3}} ) x_{{2}}^{6})c_5\\
 s_3(x_i) = 24 ( 1+ ( x_{{2}}^{4}x_{{3}}^{2}-1 )   
\cos^{2} ( x_{{1}} )   -2\sin ( x_{{1}}
 ) \cos ( x_{{1}} ) x_{{3}}x_{{2}}^{2} ) 
  \cos^{2} ( x_{{1}} )   c_1+ 
  \\
    12\cos ( x_{{1}} )  ( 2  \cos^{2} ( x_{{1}}
 )   \sin ( x_{{1}} ) x_{{2}}^{4}x_{{3}
}^{2}+4x_{{2}}^{2}x_{{3}}  \cos^{3} ( x_{{1}} ) 
  -3\cos ( x_{{1}} ) x_{{2}}^{2}x_{{3}}-
\\
2\sin ( x_{{1}} )   \cos^{2} ( x_{{1}} ) 
  +\sin ( x_{{1}} )  )c_2+
( 24( 1-x_{{2}}^{4}x_{{3}}^{2} )   \cos^{4}
 ( x_{{1}} )   +48\sin ( x_{{1}} ) 
  \cos^{3} ( x_{{1}} )   x_{{3}}x_{{2}}^{2}+
  \\
 24( x_{{2}}^{4}x_{{3}}^{2}-1 )   \cos^{2} ( 
x_{{1}} )   -24\sin ( x_{{1}} ) \cos
 ( x_{{1}} ) x_{{3}}x_{{2}}^{2}
+4)c_3+(60
  \cos^{2} ( x_{{1}} )   x_{{3}}x_{{2}}^{2}-
\\
48  \cos ^{4} ( x_{{1}} )  x_{{3}}x_{{2}}
^{2}+ 24( 1-x_{{2}}^{4}x_{{3}}^{2}) \sin ( x_
{{1}} )   \cos ^{3} ( x_{{1}} )  +
 ( 24x_{{2}}^{4}x_{{3}}^{2}-12 ) \sin ( x_{{1}}
 ) \cos ( x_{{1}} ) -
\\ 12x_{{3}}x_{{2}}^{2}
)c_4+24 ( \cos ( x_{{1}} ) +1 )  (  ( x_
{{2}}^{4}x_{{3}}^{2}-1 )   \cos^{2} ( x_{{1}} ) 
  -2\sin ( x_{{1}} ) \cos ( x_{{1}}
 ) x_{{3}}x_{{2}}^{2}-
\\
x_{{2}}^{4}x_{{3}}^{2} ) 
 ( \cos ( x_{{1}} ) -1 ) c_5.
\end{gather*}

The solutions in representation $\frak{sl}(2,\R)\otimes \frak{so}(3)$ are submetrics that take in our coordinates form
\begin{gather*}
s_1(x_i)={\scriptstyle {1 \over x_2^2}} (( 1+ 
(1 -{x_{{3}}}^{2} ) x_{{2}}^{4}
 ) \cos ( 2x_{{1}} ) +2x_{2}^{2}x_{{3}}\sin
 ( 2x_{{1}} ) +2x_{{2}}^{2}\cos ( x_{{4}}
 ) -1+ \\
 ( 1-x_{{3}}^{2} ) x_{{2}}^{4}) c_{10}+ 
2\sin(x_4)c_{11}-
{ \scriptstyle \frac { 1}{x_{{2}}^{2}}}(  (  ( x_{{3}}^{2}-1 ) x_{{2}}^
{4} -1) \cos ( 2x_{{1}} ) -2x_{{2}}^{2}x_{{3}}
\sin ( 2x_{{1}} ) - 
\\
2x_{{2}}^{2}\cos ( x_{{4}}
 ) -1+ ( 1-x_{{3}}^{2} ) x_{{2}}^{4} ) c_{13}+2\sin(x_4)c_{14}
+{\scriptstyle \frac {1 }{x_{{2}}^{2}}} (4x_{{2}}^{2}x_{{3}}\cos ( 2x_{{1}} ) +
\\
 (  ( 2x_{{3}}^{2}-2 ) x_{{2}}^{4}-2 ) 
\sin ( 2x_{{1}} ))c_7
\\
 s_2(x_i)=  
(\sin ( 2x_{{1}} )-
x_{{2}}^{2}x_{{3}}(\cos ( 2x_{{1}} ) +
1)+
 {\scriptstyle {1 \over 2}}(\cos ( x_{{4}}-x_{{5}}
 ) -\cos ( x_{{4}}+x_{{5}} ) 
))c_{10} +
\\
{\scriptstyle {1 \over 2}}(\sin ( x_{{4}}-x_{{5}} ) -\sin ( x_{{4}}+x_{
{5}} ) 
)c_{11}+
{\scriptstyle \frac {1}{2x_{{2}}^{2}}}( ( 1+ ( 1-x_{{3}}^{2} ) x_{{2}}^{4}
 ) \cos ( 2x_{{1}} ) +2x_{{2}}^{2}\cos ( x
_{{5}} ) + 
\\
( x_{{3}}^{2}-1 ) x_{{2}}^{4}+2x_{{
2}}^{2}x_{{3}}\sin ( 2x_{{1}} ) + 
1)c_{12}+(x_{{2}}^{2}x_{{3}}(1-\cos ( 2x_{{1}} ))+\sin ( 2x_{{1}} ) +
\\
{\scriptstyle {1 \over 2}}(\cos ( x_{{4}}-x_{{5}}
 ) -\cos ( x_{{4}}+x_{{5}} )) 
)c_{13}+
{\scriptstyle {1 \over 2}} ( \sin ( x_{{4}}-x_{{5}} ) -\sin ( x_{{4}}+
x_{{5}} )  ) c_{{14}}+
 \end{gather*}
\begin{gather*}
{\scriptstyle \frac { 1}{x_{{2}}^{2}}}
((  ( x_{{3}}^{2}-1 ) x_{{2}}^{4}-1
 ) \sin ( 2x_{{1}} ) +2x_{{2}}^{2}x_{{3}}\cos
 ( 2x_{{1}} ) ) c_6+2(x_{{2}}^{2}x_{{3}}\sin ( 2x_{{1}} ) + 
\\
\cos ( 
2x_{{1}} ) +\cos ( x_{{4}} ) 
 )c_7 + 2\sin(x_4) c_8 +
 {\scriptstyle \frac { 1 }{2x_{{2}}^{2}}}
(( 1-x_{{3}}^{2}) x_{{2}}^{4}+2x_{{2}
}^{2}x_{{3}}\sin ( 2x_{{1}} ) + \\
2x_{{2}}^{2}\cos
 ( x_{{5}} ) -1+ ( 1+ ( 1-x_{{3}}^{2}) 
x_{{2}}^{4} ) \cos ( 2x_{{1}} ))c_9 
\\
s_3(x_i)=2(\sin ( 2 x_{{1}} )-
x_{{2}}^{2}x_{{3}}(\cos ( x_{{1}} ) -1) 
)c_{12}+4(x_{{2}}^{2}x_{{3}}\sin ( 2x_{{1}} ) +\cos ( 
2x_{{1}} ) +
\\
\cos ( x_{{5}} ) 
)c_6+2(\cos ( x_{{4}}-x_{{5}} ) -
\cos ( x_{{4}}+x_{{5}}) 
)c_7+2(\sin ( x_{{4}}-x_{{5}} ) -
\\
\sin ( x_{{4}}+x_{{5}}
) )c_8+
2(\sin ( 2 x_{{1}} )-
x_{{2}}^{2}x_{{3}}(\cos ( x_{{1}} ) +1) 
)c_9.
\end{gather*}

Finally, the solutions in $S^2\R^3$ representation of $\frak{so}(3)$ are submetrics that take in our coordinates form
\begin{gather*}
s_1(x_i)=c_{{15}}+2\cos ( 2x_{{4}} ) c_{{18}}- ( \cos
 ( 2x_{{4}} ) +1 ) c_{{19}}+c_{{20}}\sin (2
x_{{4}} ) 
\\
s_2(x_i)=(  \cos ( 2x_{{4}}-x_{{5}}
 )  -\cos ( x_{{5}}+2x_{{4}} )) c_{{18}}+  {\scriptstyle {1 \over 2}}(\cos ( x_{{5}}+2x_{{4}} - 
\\
\cos (2x_{{4}} -x_{{5}
} )) 
 ) c_{{19}}+  {\scriptstyle {1 \over 2}}(\sin ( 2x_{{4}}-x_{{5}}
 ) -\sin ( x_{{5}}+2x_{{4}} )  ) c_{{20}
}- 
\\
 {\scriptstyle {1 \over 2}}(\sin ( x_{{4}}-x_{{5}} ) + 
\sin ( 
x_{{4}}+x_{{5}} )  ) c_{{21}}-  {\scriptstyle {1 \over 2}}(\cos ( x_
{{4}}-x_{{5}} ) +\cos ( x_{{4}}+x_{{5}} ) 
 ) c_{{22}}
\\
s_3(x_i)=c_{{15}}+ ( {\scriptstyle {1 \over 2}}(\cos ( 2x_{{5}}+2x_{{4}} )  + \cos ( 2x_{{4}}-2x_{{5}}
 )) -\cos
 ( 2x_{{4}} ) ) c_{{18}}+ 
\\
 ( {\scriptstyle {1 \over 2}}-{\scriptstyle {1 \over 4}}(\cos ( 2x_{{5}}+2
x_{{4}} ) -\cos ( 2x_{{4}}-2x_{{5}} )) +{\scriptstyle {1 \over 2}}
\cos ( 2x_{{4}} ) +{\scriptstyle {3 \over 2}}\cos ( 2x_{{5}} ) 
 ) c_{{19}}+
\\
 ( {\scriptstyle {1 \over 4}}(\sin ( 2x_{{4}}-2x_{{5}}
 ) +\sin ( 2x_{{5}}+2x_{{4}} )) -{\scriptstyle {1 \over 2}}\sin
 ( 2x_{{4}} )  ) c_{{20}}+ 
\\
 {\scriptstyle {1 \over 2}}(\sin ( 
2x_{{5}}+x_{{4}} ) -\sin ( x_{{4}}-2x_{{5}}
 )  ) c_{{21}}+ {\scriptstyle {1 \over 2}}(\cos ( 2x_{{5}}+x_{{4}} ) -\cos ( x_{{4}}-2x_{{5}} )   ) c_{{
22}}.
\end{gather*}

All of these solutions are globally defined because $x_2>0$.

\subsection{Solutions in CR geometry and theory of PDEs.} \label{exp-crlagr}
We present here solutions for the examples in Sections \ref{CR-alp} and \ref{CR-S}. We see from the brackets given in Section \ref{lagr-alp} that 
$$\fk_{CR} \simeq \frak{sl}(2,\R)^{CR} \oplus \frak{heis}_{12}\oplus \frak{so}(2), \ \ \  \fk_{LC} \simeq \frak{sl}(2,\R)^{LC} \oplus \frak{heis}_{12}\oplus \frak{so}(1,1),$$ where $\frak{sl}(2,\R)^{CR}\simeq  \langle e_2,e_1+{\scriptstyle {3 \over 8}} e_6 -{\scriptstyle {1 \over 8}}e_7,e_4\rangle$, $\frak{sl}(2,\R)^{LC}\simeq  \langle e_2,e_1-{\scriptstyle {3 \over 4}} e_6 +{\scriptstyle {1 \over 4}}e_7,e_4\rangle$,  $\frak{heis}_{12}\simeq  \langle e_1,e_3,e_5\rangle$, $\frak{so}(2)\simeq  \langle e_6-3e_7\rangle \simeq \frak{so}(1,1)$ and $\fh=\langle e_6,e_7 \rangle$. In particular, the only differences between $\fk_{CR}$ and $\fk_{LC}$ are the action of $\langle e_6-3e_7\rangle$ on $\frak{heis}_{12}$ and the projection of $\fh$ into $\frak{sl}(2,\R)$ given as $\frak{so}(2)$ in the CR case and the Cartan subalgebra in the LC case. 
This however causes a big difference in the complements $\fc$ compatible with the Iwasawa decomposition from Section \ref{exponential-coord}. In the CR case, we consider the complement $\fc_{CR}$ with coordinates ${\sf c}_{CR}=[x_2+x_3, {x_2 \over \sqrt{2}}, x_4, \sqrt{2}x_1, x_5, {\scriptstyle {3 \over 8}}x_2, {\scriptstyle -{1 \over 8}}x_2]$, where $x_1$ exponentiates in $\textsf{A}$ and the rest is in $\textsf{N}\exp(\fn)$. In the LC case, we consider the complement $\fc_{LC}$ with coordinates ${\sf c}_{LC}=[x_3, x_1, x_4, x_2-x_1, x_5, 0, 0]$, where $x_1$ exponentiates in $\textsf{K}$  and  the rest is in $\textsf{N}\exp(\fn)$.

We consider several possibilities for the choice of the group $K$ with the Lie algebra  $\fk_{CR}$ or $\fk_{LC}$. For the simple part, we can consider $SL(2,\R)$ or $PGL(2,\R)=SL(2,\R)/\{\id,-\id\}$. For the Heisenberg part, we can consider the Lie group $H_{12}$ of lower triangular matrices or $H_{12}/k\mathbb{Z}$, where  
$$k\mathbb{Z} \simeq \{\left(\begin{smallmatrix} 1&0&0\\  
0&1&0\\ 
k\mathbb{Z}&0&1
 \end{smallmatrix}\right)\}.$$
Finally, for the remaining part, we consider only $SO(2)$ in the CR case and the connected component of identity $\R^+$ of $SO(1,1)$ in the LC case. To get the extension of $(K,H)$, we need to determine $H$ and integrate $\alpha|_\fh$ (if possible).

In the CR case, $H=SO(2)\times SO(2)$ or $H=PSO(2)\times SO(2)$ depending on whether we take $SL(2,\R)$ or $PGL(2,\R)$. Then $\alpha|_{\fh}$ maps ${\rm i}x\in \frak{so}(2)$ onto $diag(-\frac14{\rm i}x,\frac34{\rm i}x,-\frac14{\rm i}x,-\frac14{\rm i}x)$ for the first factor and  onto $diag(2{\rm i}x,2{\rm i}x,-6{\rm i}x,2{\rm i}x)$ for the second factor. Since the target is $PSU(1,3)$, the kernel includes fourth root of unity. This implies that $diag(\frac{k}{2}\pi {\rm i},\frac{k}{2}\pi {\rm i},\frac{k}{2}\pi {\rm i},\frac{k}{2}\pi {\rm i})$ exponentiates to identity for each $k\in \mathbb{Z}.$ Therefore, the $\alpha$ image of $2k\pi  {\rm i}\in \frak{so}(2)$ exponentiates to identity and thus integrates to a homomorphism $i: SO(2)\times SO(2)\to P_{1,3}$. Since the $\alpha$--image of $\pi  {\rm i}\in \frak{so}(2)$ does not exponentiate to identity, the extension  does not exist for $H=PSO(2)\times SO(2)$. The choices on the Heisenberg part do not influence the existence of the extension. So we discuss the global existence of solutions of BGG operators on homogeneous parabolic geometries given by extensions $(\alpha,i)$ of $(SL(2,\R)\times (H_{12}\rtimes SO(2)), SO(2)\times SO(2))$ and $(SL(2,\R)\times (H_{12}/k\mathbb{Z}\rtimes SO(2)), SO(2)\times SO(2))$ to $(PSU(1,3),P_{1,3})$, respectively.

In the LC case, $H=diag(a,\frac1a)\times \R^+, a>0$ is always product of the image of exponential map on the Cartan subalgebra $\textsf{a}$ in $\frak{sl}(2,\R)$ with the connected component of identity $\R^+$ of $SO(1,1)$. Then $\alpha|_{\fh}$ maps $x\in \mathbb{R}$  onto $diag(\frac12x,-\frac32x,\frac12x,\frac12x)$ for the first factor and  onto $diag(2x,2x,-6x,2x)$ for the second factor. This poses no obstruction for the integrability and we discuss the global existence of solutions of BGG operators on homogeneous parabolic geometries given by extension $(\alpha,i)$ of $(SL(2,\R)\times (H_{12}\rtimes  \R^+),H)$, $(SL(2,\R)\times (H_{12}/k\mathbb{Z}\rtimes \R^+),H)$, $(PGL(2,\R)\times (H_{12}\rtimes  \R^+),H)$ and $(PGL(2,\R)\times (H_{12}/k\mathbb{Z}\rtimes \R^+),H)$ to $(PGL(4,\R),P_{1,3})$, respectively.

Let us finally note that there are no solutions for the standard and dual representations, i.e. in the case (1). Further, there are no new infinitesimal automorphisms in the case (3) and the solutions of the metrizability problem (4) sit in trivial representations. Thus we only discuss solutions on the symmetric powers.

\noindent
{\bf (2) Second symmetric powers.} 
In the CR case, we compute that $\rho\circ \alpha({\sf c}_{CR})$ acts on solutions $ [w_1, w_{11}, w_{12}, w_{14}, w_{15}, w_2, w_4, w_5]^t \in \mathcal{S}^{\infty}$ by the matrix 
$$
\left[ \begin {smallmatrix} 
0&{1 \over 2}x_{{2}}&{\frac {11}{36}}
\sqrt {2}x_{{1}}&-{\frac {11}{36}}x_{{5}}&0&{\frac {11}{72}}\sqrt 
{2}x_{{2}}&-{\frac {11}{36}}x_{{4}}&0\\ -{1 \over 2}x_{{
2}}&0&{\frac {11}{72}}\sqrt {2}x_{{2}}&-{\frac {11}{36}}x_{{4}}&0&
-{\frac {11}{36}}\sqrt {2}x_{{1}}&{\frac {11}{36}}x_{{5}}&0
\\ {\frac {18}{11}}\sqrt {2}x_{{1}}&{\frac {9}{11}
}\sqrt {2}x_{{2}}&0&0&-{1 \over 8}x_{{4}}&{1 \over 2}x_{{2}}&0&{1 \over 8}x_{{5}}
\\ 0&0&0&0&{1 \over 16}\sqrt {2}x_{{2}}&0&-{1 \over 2}x_{{2}}&-
{1 \over 8}\sqrt {2}x_{{1}}\\ 0&0&0&2\sqrt {2}x_{{2}}&0&0
&4\sqrt {2}x_{{1}}&{1 \over 2}x_{{2}}\\ {\frac {9}{11}}
\sqrt {2}x_{{2}}&-{\frac {18}{11}}\sqrt {2}x_{{1}}&-{1 \over 2}x_{{2}}&0
&-{1 \over 8}x_{{5}}&0&0&-{1 \over 8}x_{{4}}\\ 0&0&0&{1 \over 2}x_{{2}
}&{1 \over 8}\sqrt {2}x_{{1}}&0&0&{1 \over 16}\sqrt {2}x_{{2}}
\\ 0&0&0&-4\sqrt {2}x_{{1}}&-{1 \over 2}x_{{2}}&0&2
\sqrt {2}x_{{2}}&0
\end {smallmatrix} \right] 
$$
We again use transition such that the action of the Cartan subalgebra on the space of solutions becomes diagonal.
Explicitly, we consider the change of the basis of the space of solutions given by the transition matrix  
$$
T= \left[ \begin {smallmatrix} 
 {1 \over 2}&0&0&0&{1 \over 2}&0&0&0
\\ 0&{1 \over 2}&0&0&0&{1 \over 2}&0&0\\ -{\frac {
9}{11}}\sqrt {2}&0&0&0&{\frac {9}{11}}\sqrt {2}&0&0&0
\\ 0&0&0&{1 \over 2}&0&0&0&{1 \over 2}\\ 0&0&{1 \over 2}&0
&0&0&{1 \over 2}&0\\ 0&{\frac {9}{11}}\sqrt {2}&0&0&0&
-{\frac {9}{11}}\sqrt {2}&0&0\\ 0&0&-{1 \over 16}\sqrt {2}
&0&0&0&{1 \over 16}\sqrt {2}&0\\ 0&0&0&2\sqrt {2}&0&0&0&
-2\sqrt {2}
\end {smallmatrix} \right] 
.$$
If we apply this transition, the part ${\sf A}$ acts on solutions by the diagonal matrix 
$A=diag(e^{x_1},e^{x_1},e^{x_1},e^{x_1},e^{-x_1},e^{-x_1},e^{-x_1},e^{-x_1})$ and
the nilpotent part acts by 
$$
N_{CR}=\left[ \begin{smallmatrix} 
1&0&-{\frac {11}{288}}\sqrt {2}x_{{
4}}&{\frac {11}{36}}x_{{5}}&0&0&0&0\\ 0&1&{\frac {
11}{288}}\sqrt {2}x_{{5}}&{\frac {11}{36}}x_{{4}}&0&0&0&0
\\ 0&0&1&0&0&0&0&0\\ 0&0&0&1&0&0&0
&0\\ 0&-x_{{2}}&-{\frac {11}{288}}x_{{2}}\sqrt {2}
x_{{5}}&-{\frac {11}{36}}x_{{2}}x_{{4}}&1&0&{\frac {11}{288}}
\sqrt {2}x_{{4}}&{\frac {11}{36}}x_{{5}}\\ x_{{2}}
&0&-{\frac {11}{288}}x_{{2}}\sqrt {2}x_{{4}}&{\frac {11}{36}}x_{{2
}}x_{{5}}&0&1&-{\frac {11}{288}}\sqrt {2}x_{{5}}&{\frac {11}{36}}x
_{{4}}\\ 0&0&0&-4\sqrt {2}x_{{2}}&0&0&1&0
\\ 0&0&-{1 \over 8}\sqrt {2}x_{{2}}&0&0&0&0&1
\end {smallmatrix} \right].
$$
Altogether, $\exp(-\rho \circ \alpha({\sf c}_{CR}))$ equals to $TN_{CR}AT^{-1}$ and
 all the solutions take in exponential coordinates form 
\begin{gather*}
-{\scriptstyle {576 \over 11}}(w_1 +{\rm i} w_{11})= -{\scriptstyle \frac {288}{11}}({{ e}^{-x_{{1}}}}+{{ e}^{
x_{{1}}}}+{\rm i}x_{{2}}{{ e}^{-x_{{1}}}}
)c_1+ {\scriptstyle \frac {288}{11}}x_{{2}}({{ e}^{-x_{{1}}}}-{\rm i}{
{\rm e}^{x_{{1}}}}-{\rm i}{{ e}^{-x_{{1}}}}
)c_{11}+
\\
8\,\sqrt {2} ( {\rm i}x_{{2}}{{e}^{-x_{{1}}}}-{{ e}^{-x_{{1}}}}+
{{ e}^{x_{{1}}}} ) c_{12}+8( {\rm i}{{ e}^{-x_{{1}}}}-x_{{2}}{{ e}^{-x_{{1}}}}+{\rm i}{{e}
^{x_{{1}}}} )  ( ix_{{5}}-x_{{4}} )c_{14}-
\\
{\rm i}\sqrt {2} ( x_{{2}}{{e}^{-x_{{1}}}}+{\rm i}{{e}^{-x_{{1}}}}-{\rm i}
{{e}^{x_{{1}}}} )  ( {\rm i}x_{{5}}-x_{{4}} )c_{15} +8\sqrt {2} ( {{\rm e}^{-x_{{1}}}}x_{{2}}+{\rm i}{{e}^{-x_{{1}}}}-
{\rm i}{{ e}^{x_{{1}}}} )c_{2}-
\\
8i ( {\rm i}{{e}^{-x_{{1}}}}-{{ e}^{-x_{{1}}}}x_{{2}}+{\rm i}{
{e}^{x_{{1}}}} )  ( {\rm i}x_{{5}}-x_{{4}} ) c_4 -\sqrt {2} ( {{ e}^{-x_{{1}}}}x_{{2}}+{\rm i}{{\rm e}^{-x_{{1}}}}-{\rm i}{
{ e}^{x_{{1}}}} )  ( {\rm i}x_{{5}}-x_{{4}} )
c_5
\end{gather*}
Since $k\mathbb{Z}$ corresponds to the coordinate $x_3$ which does not appear in the solution, the solution exists globally on both $(SL(2,\R)\times (H_{12}\rtimes SO(2)))/ (SO(2)\times SO(2))$ and $(SL(2,\R)\times (H_{12}/k\mathbb{Z}\rtimes SO(2)))/( SO(2)\times SO(2))$.
 
 In the LC case, we compute that $\rho\circ \alpha({\sf c}_{LC})$ acts on solutions $ [w_1, w_{2}, w_{4}, w_5]^t \in \mathcal{S}^{\infty}$ by the matrix 
$$
\left[ \begin {smallmatrix} 
0&{\frac {11}{18}}({ x_2}-{ x_1})&-{\frac {11}{18}}{x_5}&0\\ {
\frac {18}{11}}{ x_1}&0&0&-{1 \over 4}{ x_5}\\ 0&0&0
&{1 \over 4}({ x_2}-{ x_1})\\ 0&0&4{ x_1}&0
\end {smallmatrix} \right] 
$$
Then the part ${\sf K}$ and the nilpotent part act on solutions by the following matrices 
$$R_1=\left[ \begin {smallmatrix} 
\cos ( x_{{1}} ) &{\frac {11}{
18}}\sin ( x_{{1}} ) &0&0\\ -{\frac {18
}{11}}\sin ( x_{{1}} ) &\cos ( x_{{1}} ) &0&0
\\ 0&0&\cos ( x_{{1}} ) &{1 \over 4}\sin
 ( x_{{1}} ) \\ 0&0&-4\sin ( x_{{1
}} ) &\cos ( x_{{1}} )
 \end {smallmatrix} \right],
\ \ \ \ 
N_{LC}=\left[ \begin{smallmatrix} 
1&-{\frac {11}{18}}x_{{2}}&{\frac {11}{
18}}x_{{5}}&-{\frac {11}{72}}x_{{2}}x_{{5}}\\ 0&
1&0&{1 \over 4}x_{{5}}\\ 0&0&1&-{1 \over 4}x_{{2}}
\\ 0&0&0&1
\end {smallmatrix} \right].
$$
Altogether, $\exp(-\rho \circ \alpha({\sf c}_{LC}))$ equals to $N_{LC}R_1$ and all the solutions take in exponential coordinates form 
\begin{gather*}
 -{\scriptstyle {144 \over 11}}w_1=
-{\scriptstyle \frac {144}{11}}\cos ( x_{{1}} ) c_{{1}}+ 8( x_{
{2}}\cos ( x_{{1}} ) -\sin ( x_{{1}} ) 
 ) c_{{2}}-8x_{{5}}\cos ( x_{{1}} ) c_{{4}}+
\\
2x_{{5}}
 ( x_{{2}}\cos ( x_{{1}} ) -\sin
 ( x_{{1}} )  ) c_{{5}}.
\end{gather*}

Again, $k\mathbb{Z}$ corresponds to the coordinate $x_3$ and thus does not pose obstruction for global existence of the solutions. On the other hand, $-\id$ corresponds to $x_1=\pi$, which provides different value for the solutions than $x_1=2\pi$ and thus the solutions exist globally only on the homogeneous spaces $(SL(2,\R)\times (H_{12}\rtimes  \R^+))/H$ and $(SL(2,\R)\times (H_{12}/k\mathbb{Z}\rtimes \R^+)),H$), i.e., $\mathcal{S}^{\infty}_{PGL(2,\R)\times (H_{12}/k\mathbb{Z}\rtimes \R^+)}=\{0\}.$

\subsection{Solutions in the theory of systems of ODEs.} \label{exp-path}
We present here solutions for the example in Sections \ref{path-alp} and \ref{path-S}. Here, the choice of the complement $\fc$ will depend on the choice of the groups $K$ and $H$.

Firstly, we consider the unipotent Lie groups $K=\exp(\fk)$ and $H=\exp(\fh)=H_{12}$, where ${\sf c}_1=[x_1,\dots, x_5,0]$ gives the complement. We also consider the quotient $K/k\mathbb{Z}$ of $K$ by the normal subgroup generated by $\exp(kx_2)$ for $x_2\in \mathbb{Z}$. There is no obstruction for the existence of the extension $(\alpha,i)$ of $(K,H)$ and $(K/k\mathbb{Z},H)$ to $(PGl(4,\R),P_{1,2})$, respectively.

Next, we consider $K_f=GL(2,\R)\exp(\fk)$, where the product is not a direct, because $\frak{gl}(2,\R)\cap \fk=\langle e_3 \rangle$ and $\exp(x_3e_3)$  is an open subset (missing one point) of the sphere $GL(2,\R)/(P_1\times \R)$, and $H_f=P_1\times \R\times \exp(\fh)$. So we need to choose a different complement ${\sf c}_2$ to cover this part globally, which is not contained in $\fk$, namely $\frak{so}(2)\subset \fk_f$. To compute with this new complement, we do not need to recompute the results in Sections \ref{path-alp} and \ref{path-S} for $\fk_f$. It is sufficient to add $$-x_3\rho(\left[ \begin{smallmatrix} 
0&1&0&0\\ 0&
0&0&0\\ 0&0&0&0
\\ 0&0&0&0
\end {smallmatrix} \right])$$ to the formula for $\Phi$. Then $x_3$ exponentiates into $\textsf{K}$ part and the rest is contained in $\exp(\fn)$ part. It is clear that there is an extension $(\alpha_f,i_f)$ of $(K_f,H_f)$ to $(PGL(4,\R),P_{1,2})$, but as we mentioned above, we do not need it explicitly.

\noindent
{\bf (1) Standard and dual representations.}
For the standard representation, the solution $w_4 \in \mathcal{S}^{\infty}$ sits in trivial representation. 
For the dual representation, solutions take form $ [w_1,w_2,w_3]^t \in \mathcal{S}^{\infty}$
and we compute 
$$\rho\circ \alpha({\sf c}_1)=
\left[ \begin{smallmatrix} 0&-x_3&-x_1\\ 
0&0&-x_4\\ 
0&0&0\\ 
 \end{smallmatrix} \right], \ \ \ \rho\circ \alpha({\sf c}_2)=
\left[ \begin{smallmatrix} 0&-x_3&-x_1\\ 
x_3&0&-x_4\\ 
0&0&0\\ 
 \end{smallmatrix} \right].$$
Thus $\exp(-\rho\circ \alpha({\sf c}_1))$ equals to
$${\footnotesize \left[ \begin {array}{ccc} 
1 & x_3 & {1 \over 2}x_3x_4+x_1 \\
0&1&x_4 \\ 0&0&1
\end {array} \right] 
}$$ 
and all the solution  take in exponential coordinates of ${\sf c}_1$ form
 $$ { w_1} (x_i ) 
={ c_1}+{ x_3}{ c_2}+ ( {\scriptstyle {1 \over 2}}{ x_3}{ x_4}+{ x_1}) { c_3}.  
$$ 
Simultaneously, we compute that
$\exp(-\rho\circ \alpha({\sf c}_2))$ equals to $NR_1$ for
$$N= \left[ \begin {smallmatrix}1 &0 &x_1\\
0 &1&x_4\\ 0&0&1\end {smallmatrix} \right],R_1=\left[ \begin {smallmatrix} \cos ( x_{{3}} ) &\sin ( x
_{{3}}) &0\\-
\sin ( x_{{3}} ) &\cos ( x_{{3}} ) &0\\ 0&0&1\end {smallmatrix} \right]
$$
and all the solution take in exponential coordinates of ${\sf c}_2$ form
\begin{gather*}
 { w_1} ( x_i ) =
\cos ( x_{{3}} ) c_{{1}}+\sin ( x_{{3}} ) c_{{2}
}+x_1 c_{{3}}.
\end{gather*}
These are globally defined for all of the cases we consider.


\noindent
{\bf (2) Second powers.}
For the second symmetric power of the standard representation, the solutions $w_{10} \in \mathcal{S}^{\infty}$ sits in trivial representation. For the second skew--symmetric power of the standard representation, solutions take form $ [w_4,w_5,w_6]^t \in \mathcal{S}^{\infty}$ and we compute
$$\rho\circ \alpha({\sf c}_1)= \left[ \begin{smallmatrix} 0&0&0\\ 
x_3&0&0\\ 
x_1&x_4&0\\ 
 \end{smallmatrix} \right], \ \ \rho\circ \alpha({\sf c}_2)= \left[ \begin{smallmatrix} 0&-x_3&0\\ 
x_3&0&0\\ 
x_1&x_4&0\\ 
 \end{smallmatrix} \right].$$
Thus $\exp(-\rho\circ \alpha({\sf c}_1))$ equals to 
$$ \left[ \begin {smallmatrix} 
1 &0 & 0 \\
-x_3&1&0 \\ -x_1-x_4+{1 \over 2}x_3x_4&-x_4&1
\end {smallmatrix} \right] 
$$ 
and we compute that solutions take in exponential coordinates for ${\sf c}_1$ form
$$ 
{ w_6} ( x_i) = ( -x_{{1}}+{\scriptstyle {1 \over 2}}x_{{3}}x_{{4}}
 ) { c_1}-x_{{4}}{ c_2}+{ c_3}
.$$
Simultaneously, $\exp(-\rho\circ \alpha({\sf c}_2))$ equals to $NR_1$ for
$$N= \left[ \begin {smallmatrix}1 &0 &0\\
0 &1&0\\ -x_1&-x_4&1\end {smallmatrix} \right],R_1=\left[ \begin {smallmatrix} \cos ( x_{{3}} ) &\sin ( x
_{{3}}) &0\\-
\sin ( x_{{3}} ) &\cos ( x_{{3}} ) &0\\ 0&0&1\end {smallmatrix} \right]
$$
and we compute that solutions take in exponential coordinates for ${\sf c}_2$ form
\begin{gather*}
w_{{6}}( x_i ) = 
(x_4\sin(x_3)-x_1\cos(x_3) )c_{{1}}- 
(x_1\sin(x_3)+x_4\cos(x_3)) c_{{2}}+c_{{3}} .
\end{gather*}
These are globally defined for all of the cases we consider.

Let us now swap to duals. For the second symmetric power of the dual representation, solutions take form $ [w_1,w_2,w_3,w_4,w_5,w_6,w_7,w_8]^t \in \mathcal{S}^{\infty}$ and we compute that $\rho\circ \alpha({\sf c}_1)$ and $\rho\circ \alpha({\sf c}_2)$ act as follows 
$$\left[ \begin{smallmatrix} 
0&-2x_{{3}}&0&-2x_{{1}}&0&0&-2x
_{{2}}&0\\ 0&0&-x_{{3}}&-x_{{4}}&-x_{{1}}&0&-x_{{5}}
&-x_{{2}}\\ 0&0&0&0&-2x_{{4}}&0&0&-2x_{{5}}
\\ 0&0&0&0&-x_{{3}}&-x_{{1}}&0&0
\\ 0&0&0&0&0&-x_{{4}}&0&0\\ 0&0&0&0
&0&0&x_{{4}}&-x_{{1}}\\ 0&0&0&0&0&0&0&-x_{{3}}
\\ 0&0&0&0&0&0&0&0 
 \end{smallmatrix} \right],
\left[ \begin {smallmatrix} 0&-2x_{{3}}&0&-2x_{{1}}&0&0&-2x
_{{2}}&0\\ x_3&0&-x_{{3}}&-x_{{4}}&-x_{{1}}&0&-x_{{5}}
&-x_{{2}}\\ 0&2x_3&0&0&-2x_{{4}}&0&0&-2x_{{5}}
\\ 0&0&0&0&-x_{{3}}&-x_{{1}}&0&0
\\ 0&0&0&x_3&0&-x_{{4}}&0&0\\ 0&0&0&0
&0&0&x_{{4}}&-x_{{1}}\\ 0&0&0&0&0&0&0&-x_{{3}}
\\ 0&0&0&0&0&0&x_3&0\end {smallmatrix} \right]
.$$
Thus $\exp(-\rho\circ \alpha({\sf c}_1))$ equals to 
$$ 
\left[ \begin {smallmatrix} 1&*&*&*&*&
*&
*
&*
\\ 0&1&{ x_3}&{ x_4}&\frac32{ x_3}{ x_4}+{
 x_1}&\frac12{ x_3} x_4^{2}+{ x_4}{ x_1}&{ x_5}-\frac18{ x_3
} x_4^{3}-\frac13 x_4^{2}{ x_1}&{ x_2
}-\frac{1}{40} x_3
^{2} x_4^{3}+\frac13{{ x_1}}{ x_4}({ x_1}+ {\frac {1}{8}} {{ x_4}}{ x_3})+{\frac {3}{2}}{ x_5}{ x_3}\\ 
0&0&1&0&2{ x_4}& x_4^{2}&-{1 \over 3} 
x_4^{3}&-{1 \over 12}{ x_3} x_4^{3}+\frac13 x_4^{2}{ x_1}+2
{ x_5}\\ 0&0&0&1&{ x_3}&\frac12{ x_3}{ x_4
}+{ x_1}&-\frac16{ x_4} \left( { x_3}{ x_4}+3{ x_1}
 \right) &-{1 \over 24} x_3^{2}{{ x_4}}^{2}+\frac12x_1^{2}
\\ 0&0&0&0&1&{ x_4}&-\frac12 x_4^{2}&-\frac16{
 x_3} x_4^{2}+\frac12{ x_4}{ x_1}\\ 0&0
&0&0&0&1&-{ x_4}&-\frac12{ x_3}{ x_4}+{ x_1}
\\ 0&0&0&0&0&0&1&{ x_3}\\ 0&0&0&0
&0&0&0&1\end {smallmatrix} \right],
$$ 
where the first line is too long to write it down and can be read of the solution in exponential coordinates for ${\sf c}_1$ that takes form
\begin{gather*}
 {w_1} ( x_i ) =c_{{1
}}+2x_{{3}}c_{{2}}+x_{{3}}^{2}c_{{3}}+ ( x_{{3}}x_{{4}}+2x_
{{1}} ) c_{{4}}+ x_{{3}}( x_{{3}}x_{{4}}+2x_{{1}}
 ) c_{{5}}+
\\
{\scriptstyle  {1 \over 4} }( { x_3}{ x_4}+2{ x_1}) ^{2} c_{{6}}+
 ( -{\scriptstyle {1 \over 20}}x_{{3}}^{2}x_
{{4}}^{3}-{\scriptstyle {1 \over 4}}x_{{3}}x_{{4}}^{2}x_{{1}}+x_{{3}}x_{{5}}-{\scriptstyle {1 \over 3}}x_{{1
}}^{2}x_{{4}}+2x_{{2}} ) c_{{7}}+
\\
 ( -{\scriptstyle \frac {1}{120}}
x_{{3}}^{3}x_{{4}}^{3}+x_{{3}}^{2}x_{{5}}+{\scriptstyle {1 \over 6}}x_{{3}}x_{{1}}^{
2}x_{{4}}+2x_{{3}}x_{{2}}+{\scriptstyle {1 \over 3}}x_{{1}}^{3} ) c_{{8}}
.
\end{gather*}
All of these solutions exist in the case of extension of $(K,H)$, but only those not containing $x_2$ exist globally in the case of extension of $(K/k\mathbb{Z},H)$, i.e, $\mathcal{S}^{\infty}_{K/k\mathbb{Z}}=[w_1,w_2,w_3,w_4,w_5,w_6,0,0]\subset \mathcal{S}^{\infty}.$

For ${\sf c}_2$ we decompose the action $\exp(-\rho\circ \alpha)$ into the nilpotent part
$$ 
N=\left[ \begin {smallmatrix} 
1&0&0&2x_{{1}}&0&{x_{{1}}}^{2}&-\frac13
{x_{{1}}}^{2}x_{{4}}+2x_{{2}}&\frac13{x_{{1}}}^{3}
\\ 0&1&0&x_{{4}}&x_{{1}}&x_{{4}}x_{{1}}&-\frac13{x_{{4
}}}^{2}x_{{1}}+x_{{5}}&\frac13{x_{{1}}}^{2}x_{{4}}+x_{{2}}
\\ 0&0&1&0&2x_{{4}}&{x_{{4}}}^{2}&-\frac13{x_{{4}}}^
{3}&\frac13{x_{{4}}}^{2}x_{{1}}+2x_{{5}}\\ 0&0&0&1&0
&x_{{1}}&-\frac12x_{{4}}x_{{1}}&\frac12{x_{{1}}}^{2}\\ 0
&0&0&0&1&x_{{4}}&-\frac12{x_{{4}}}^{2}&\frac12x_{{4}}x_{{1}}
\\ 0&0&0&0&0&1&-x_{{4}}&x_{{1}}\\ 0
&0&0&0&0&0&1&0\\ 0&0&0&0&0&0&0&1
\end {smallmatrix} \right]
$$
and the compact part
$$ 
R_1=\left[ \begin {smallmatrix} 
 \cos ^{2} ( x_{{3}} ) 
 &2\sin ( x_{{3}} ) \cos ( x_{{3}}
 ) &-  \cos ^{2} ( x_{{3}} ) +1&0&0&0&0
&0\\ -\sin ( x_{{3}} ) \cos ( x_{{3}
} ) &2  \cos^{2} ( x_{{3}} )   -1&\sin
 ( x_{{3}} ) \cos ( x_{{3}} ) &0&0&0&0&0
\\ -  \cos^{2
} ( x_{{3}} )   +1&-2\sin ( x_{{3}} ) \cos ( x_{{3}} ) &
  \cos^{2} ( x_{{3}} )  &0&0&0&0&0
\\ 0&0&0&\cos ( x_{{3}} ) &\sin ( x_
{{3}} ) &0&0&0\\ 0&0&0&-\sin ( x_{{3}}
 ) &\cos ( x_{{3}} ) &0&0&0\\ 0&0&0
&0&0&1&0&0\\ 0&0&0&0&0&0&\cos ( x_{{3}}
 ) &\sin ( x_{{3}} ) \\ 0&0&0&0&0&0
&-\sin ( x_{{3}} ) &\cos ( x_{{3}} )
\end {smallmatrix} \right].
$$
Thus $\exp(-\rho \circ \alpha({\sf c}_2))$ is given by $NR_1$ and solutions take in exponential coordinates for ${\sf c}_2$ form
\begin{gather*}
 w_1(x_i) =  \cos^{2} ( x_{{3}} ) 
  c_1+2\sin ( x_{{3}} ) \cos ( x_{{3}}
 ) c_2+(1-  \cos^{2} ( x_{{3}} ) )c_3+2x_{{
1}}\cos ( x_{{3}} ) c_4+\\
2x_{{1}}\sin ( x_{{3}} ) c_5+
x_{{1}}^{2}c_6+{\scriptstyle \frac13 }(( 6x_{{2}} -x_{{1}}^{2}x_{{4}}) \cos ( x_{{
3}} ) -\sin ( x_{{3}} ) x_{{1}}^{3})c_7+
\\
{\scriptstyle \frac13}( ( 6x_{{2}}-x_{{1}}^{2}x_{{4}} ) \sin ( x_{{
3}} ) +\cos ( x_{{3}} ) x_{{1}}^{3}
)c_8,
\end{gather*}
which exists globally.

For the second skew--symmetric power of the dual representation, solutions take form $ [w_1,w_2,w_3]^t \in \mathcal{S}^{\infty}$ and we compute 
$$\rho\circ \alpha({\sf c}_1)= \left[ \begin{smallmatrix} 0&-x_4&x_1\\ 
0&0&-x_3\\ 
0&0&0\\ 
 \end{smallmatrix} \right], \ \ \rho\circ \alpha({\sf c}_2)= \left[ \begin{smallmatrix} 0&-x_4&x_1\\ 
0&0&-x_3\\ 
0&x_3&0\\ 
 \end{smallmatrix} \right].$$
Thus $\exp(-\rho\circ \alpha({\sf c}_1))$ equals to 
$$ \left[ \begin {smallmatrix} 
1 &x_4 & {1 \over 2}x_3x_4-x_1 \\
0&1&x_3 \\ 0&0&0
\end {smallmatrix} \right] 
$$ 
and we compute that solutions take in exponential coordinates for ${\sf c}_1$ form
$$ w_1(x_i) = c_1+x_4c_2+({\scriptstyle {1 \over 2}}x_3x_4-x_1)c_3
.$$
Simultaneously, $\exp(-\rho\circ \alpha({\sf c}_2))$ equals to $NR_1$ for
$$N= \left[ \begin {smallmatrix}1 &x_4 &-x_1\\
0 &1&0\\ 0&0&1\end {smallmatrix} \right],R_1=\left[ \begin {smallmatrix} 1&  0& 0\\ 
0& \cos ( x_{{3}} ) &\sin ( x
_{{3}}) \\ 0& -
\sin ( x_{{3}} ) &\cos ( x_{{3}} ) \end {smallmatrix} \right]
$$
and we compute that solutions take in exponential coordinates for ${\sf c}_2$ form
\begin{gather*}
w_{{1}}( x_i ) = 
{ c_1}+ (x_4\cos(x_3)+x_1\sin(x_3)) { c_2}+(x_4\sin(x_3)-x_1\cos(x_3)) { c_3}.
\end{gather*}
These solutions again exist globally.

\end{document}